\documentclass[onefignum,onetabnum]{siamonline171218}

\usepackage{lipsum}
\usepackage{amsfonts}
\usepackage{amssymb}
\usepackage{graphicx}
\usepackage{epstopdf}
\usepackage{algorithmic}
\usepackage{mathtools}
\usepackage{array}
\usepackage{ragged2e}
\usepackage{subcaption}
\usepackage{titletoc}
\usepackage{bm}

\ifpdf
  \DeclareGraphicsExtensions{.eps,.pdf,.png,.jpg}
\else
  \DeclareGraphicsExtensions{.eps}
\fi

\newcolumntype{L}[1]{>{\raggedright\let\newline\\\arraybackslash\hspace{0pt}}m{#1}}
\newcolumntype{C}[1]{>{\centering\let\newline\\\arraybackslash\hspace{0pt}}m{#1}}
\newcolumntype{R}[1]{>{\raggedleft\let\newline\\\arraybackslash\hspace{0pt}}m{#1}}

\def\P{{\mathbb{P}}} 
\def\E{{\mathbb{E}}} 

\newcommand{\tr}{\textrm{tr}}
\newcommand{\R}[0]{\mathbb{R}}

\newcommand{\znote}[1]{}
\newcommand{\pnote}[1]{}
\newcommand{\snote}[1]{}
\newcommand{\mnote}[1]{}
\renewcommand{\znote}[1]{\textcolor{blue}{\textbf{[ZF: #1]}}}
\renewcommand{\pnote}[1]{\textcolor{red}{\textbf{[PR: #1]}}}
\renewcommand{\snote}[1]{\textcolor{green}{\textbf{[SZ: #1]}}}
\renewcommand{\mnote}[1]{\textcolor{orange}{\textbf{[MU: #1]}}}
\newcommand{\gammaMax}{\gamma_{u}^{\rho,\textup{max}}}
\newcommand{\gammaRho}{\gamma^{\rho}_u}

\newcommand{\bigO}{\mathcal O}

\usepackage{enumitem}
\setlist[enumerate]{leftmargin=.5in}
\setlist[itemize]{leftmargin=.5in}

\newsiamremark{remark}{Remark}
\newsiamremark{hypothesis}{Hypothesis}
\crefname{hypothesis}{Hypothesis}{Hypotheses}
\newsiamthm{claim}{Claim}
\newsiamthm{assumption}{Assumption}

\crefname{parameter}{Preconditioner: theoretical hyperparameter settings}{Preconditioner: theoretical hyperparameter settings}

\headers{SketchySGD}{Z. Frangella, P. Rathore, S. Zhao, M. Udell}

\title{SketchySGD: Reliable Stochastic Optimization via Randomized Curvature Estimates\thanks{Submitted to the editors December 22nd, 2023.
\funding{The authors gratefully acknowledge support from NSF Award 2233762, the Office of Naval Research, and the Alfred P. Sloan Foundation.}}}

\author{Zachary Frangella\thanks{Department of Management Science and Engineering, Stanford University, Stanford, CA
  (\email{zfran@stanford.edu}, \email{udell@stanford.edu})}, 
\and Pratik Rathore\thanks{Department of Electrical Engineering, Stanford University, Stanford, CA 
  (\email{pratikr@stanford.edu})},
\and Shipu Zhao\thanks{Department of Systems Engineering, Cornell University, Ithaca, NY (\email{sz533@cornell.edu})}.
\and Madeleine Udell\footnotemark[2]}

\usepackage{amsopn}

\makeatletter
\newcommand*{\addFileDependency}[1]{
  \typeout{(#1)}
  \@addtofilelist{#1}
  \IfFileExists{#1}{}{\typeout{No file #1.}}
}
\makeatother

\newcommand*{\myexternaldocument}[1]{%
    \externaldocument{#1}%
    \addFileDependency{#1.tex}%
    \addFileDependency{#1.aux}%
}

\ifpdf
\hypersetup{
  pdftitle={SketchySGD: Reliable Stochastic Optimization via Randomized Curvature Estimates},
  pdfauthor={Z. Frangella, P. Rathore, S. Zhao, M. Udell}
}
\fi

\newif\ifpreprint
\preprinttrue

\ifpreprint
\else
\myexternaldocument{ex_supplement}
\fi

\begin{document}

\maketitle

\begin{abstract}
  We introduce SketchySGD, a stochastic quasi-Newton method that uses sketching to approximate the curvature of the loss function. 
  SketchySGD improves upon existing stochastic gradient methods in machine learning by using randomized low-rank approximations to the subsampled Hessian and by introducing an automated stepsize that works well across a wide range of convex machine learning problems.
  We show theoretically that SketchySGD with a fixed stepsize converges linearly to a small ball around the minimum. 
  Further, in the ill-conditioned setting we show SketchySGD converges at a faster rate than SGD for least-squares problems.
  We validate this improvement empirically with ridge regression experiments on real data. 
  Numerical experiments on both ridge and logistic regression problems with dense and sparse data, show that SketchySGD equipped with its default hyperparameters can achieve comparable or better results than popular stochastic gradient methods, even when they have been tuned to yield their best performance.  
  In particular, SketchySGD is able to solve an ill-conditioned logistic regression problem with a data matrix that takes more than $840$GB RAM to store, while its competitors, even when tuned, are unable to make any progress.  
  SketchySGD's ability to work out-of-the box with its default hyperparameters and excel on ill-conditioned problems is an advantage over other stochastic gradient methods, most of which require careful hyperparameter tuning (especially of the learning rate) to obtain good performance and degrade in the presence of ill-conditioning.
\end{abstract}



\ifpreprint
\section{Introduction}
\label{section:Introduction}
Modern large-scale machine learning requires stochastic optimization:
evaluating the full objective or gradient even once 
is too slow to be useful.
Instead, stochastic gradient descent (SGD) and
variants are the methods of choice \cite{robbins1951stochastic,moulines2011non,johnson2013accelerating,defazio2014saga,allenzhu2018katyusha}.
However, stochastic optimizers sacrifice stability 
for their improved speed. 
Parameters like the learning rate are difficult to choose and important to get right, with slow convergence or divergence looming on either side of the best parameter choice \cite{nemirovski2009robust}.
Worse, most large-scale machine learning problems are ill-conditioned: typical condition numbers among standard test datasets range from $10^{4}$ to $10^{8}$ (see \cref{table:datasets}) or even larger, resulting in painfully slow convergence even given the optimal learning rate.
This paper introduces a method, SketchySGD, that uses a principled theory to address ill-conditioning  and yields a theoretically motivated learning rate that 
robustly works for modern machine learning problems.
\cref{fig:intro} depicts the performance of stochastic optimizers using learning rates tuned for each optimizer on a ridge regression problem with the E2006-tfidf dataset (see \Cref{section:Experiments}).
SketchySGD improves the objective substantially, while the other stochastic optimization methods (SGD, SVRG, SAGA, and Katyusha) do not. 
\begin{figure}[htbp]
    \centering
    \begin{subfigure}[b]{0.45\textwidth}
        \includegraphics[width=\textwidth]{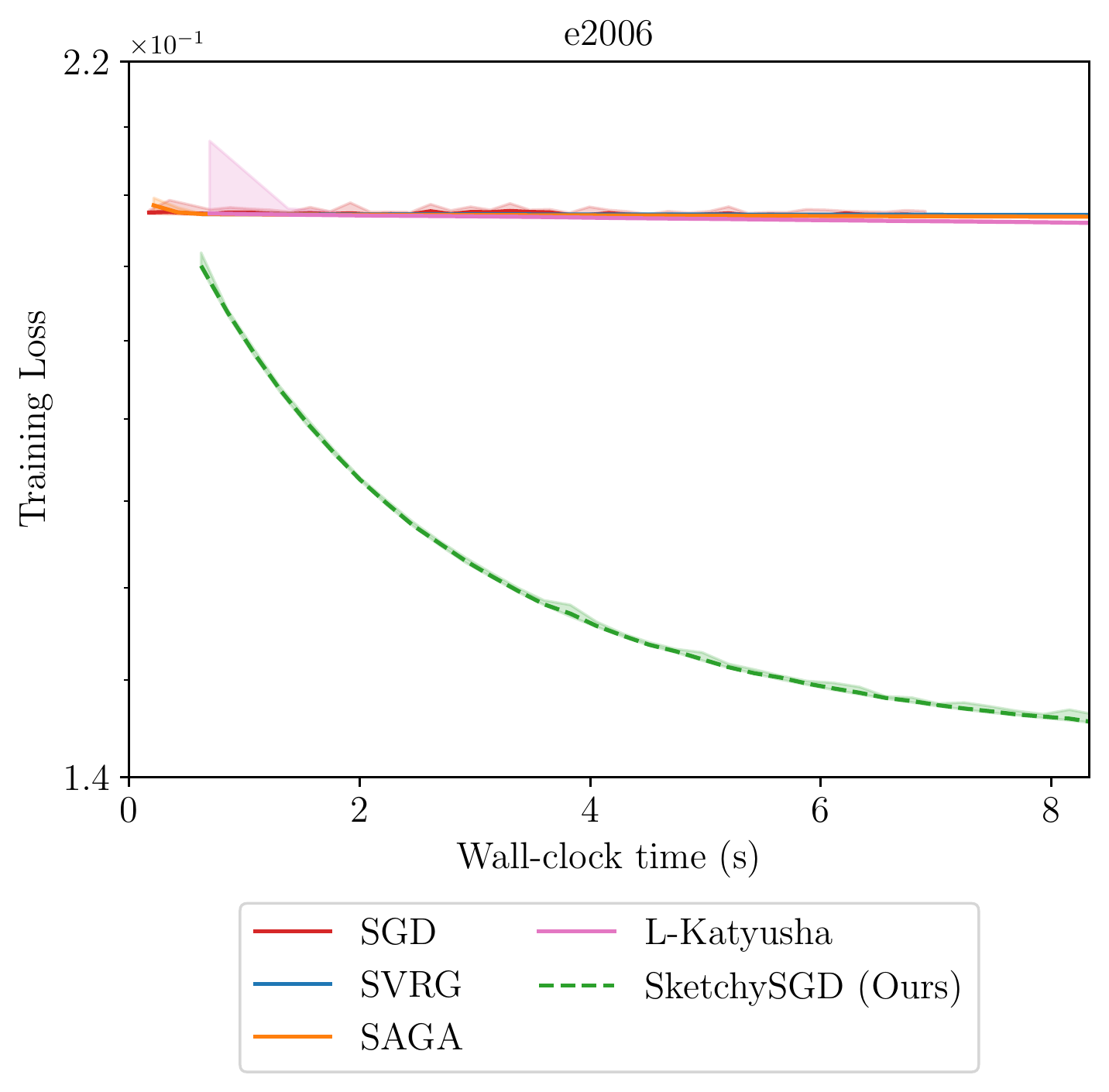}
        \label{fig:e2006_tuned_time_intro}
    \end{subfigure}
  \caption{SketchySGD outperforms standard stochastic gradient optimizers, even when their parameters are tuned for optimal performance. Each optimizer was allowed 40 full data passes.}
  \label{fig:intro}
\end{figure}

Second-order optimizers based on the Hessian, such as Newton's method and quasi-Newton methods,
are the classic remedy for ill-conditioning. 
These methods converge at super-linear rates under mild assumptions and are faster than first-order methods both in theory and in practice \cite{nocedal1999numerical, boyd2004convex}. 
Alas, it has proved difficult to design second-order methods that can use stochastic gradients.
This deficiency limits their use in large-scale machine learning.
Stochastic second-order methods using stochastic Hessian approximations but full gradients are abundant in the literature \cite{pilanci2017newton,lacotte2021adaptive,tong2021accelerating}. 
However, a practical second-order stochastic optimizer must replace both the Hessian and gradient by stochastic approximations. 
While many interesting ideas have been proposed, 
existing methods require impractical conditions for convergence: 
for example, a batch size for the gradient and Hessian that grows with the condition number \cite{roosta2019sub}, or that increases geometrically at each iteration \cite{bollapragada2019exact}.
These theoretical conditions are impossible to implement in practice. 
Convergence results without extremely large or growing batch sizes have been established under interpolation, i.e., if the loss is zero at the solution \cite{meng2020fast}. 
This setting is interesting for deep learning, but 
is unrealistic for convex machine learning models.
Moreover, 
most of these methods lack practical guidelines for hyperparameter selection, 
making them difficult to deploy in real-world machine learning pipelines.
Altogether, the literature does not present any 
stochastic second-order method that can function
as a drop-in replacement for SGD, 
despite strong empirical evidence that --- for perfectly chosen parameters --- they yield improved performance. 
One major contribution of this paper is a theory 
that matches how these methods are used in practice,
 and therefore is able to offer practical parameter selection rules that make SketchySGD (and even some previously proposed methods) practical.
We provide a more detailed discussion of how SketchySGD compares to prior stochastic second-order optimizers in \Cref{section:related}. 
  
SketchySGD
accesses second-order information using only 
minibatch Hessian-vector products to form a sketch,
and produces a preconditioner for SGD using the sketch,
which is updated only rarely (every epoch or two).
This primitive is compatible with standard practices of modern large-scale machine learning, as it can be computed by automatic differentiation.
Our major contribution is a tighter theory that enables 
practical choices of parameters that makes SketchySGD 
a drop-in replacement for SGD and variants
that works out-of-the-box, without tuning, 
across a wide variety of problem instances.

How? 
A standard theoretical argument in convex optimization shows that 
the learning rate in a gradient method should be set as the inverse of the smoothness parameter of the objective to guarantee convergence \cite{boyd2004convex,nesterov2018lectures}.
This choice generally results in a tiny stepsize and very slow convergence.
However, in the context of SketchySGD, the preconditioned smoothness constant 
is generally around $1$,
and so its inverse provides a reasonable learning rate.
Moreover, it is easy to estimate, 
again using minibatch Hessian-vector products
to measure the largest eigenvalue of a preconditioned minibatch Hessian.  

Theoretically, 
we establish SketchySGD converges to a small ball around the minimum on for both smooth convex functions and smooth and strongly convex functions, 
which suffices for good test error 
\cite{agarwal2012fast,hardt2016train,loh2017lower}.
By appealing to the modern theory of SGD for finite-sum optimization \cite{gower2019sgd} in our analysis, 
we avoid vacuous or increasing batchsize requirements for the gradient. 
As a corollary of our theory, we obtain that SketchySGD converges linearly to the optimum whenever the model interpolates the data, recovering the result of \cite{meng2020fast}.
In addition, we show that SketchySGD's required number of iterations to converge, improves upon SGD, when the objective is quadratic. 

Numerical experiments verify that SketchySGD yields comparable or superior performance to SGD, SAGA, SVRG, stochastic L-BFGS \cite{moritz2016linearly}, and loopless Katyusha \cite{kovalev2020lkatyusha} equipped with tuned hyperparameters that attain their best performance. 
Experiments also demonstrate that SketchySGD's 
default hyperparameters, including the rank of the preconditioner
and the frequency at which it is updated,
work well across a wide range of datasets.

\subsection{SketchySGD}
\label{subsection:sketchysgd}
\begin{algorithm}[tb]
   \caption{SketchySGD (Practical version)}
   \label{alg:SketchySGD_pract}
    \begin{algorithmic}
       \STATE {\bfseries Input:} initialization $w_0$,  hvp oracle $\mathcal{O}_{H}$, ranks $\{r_j\}$, regularization $\rho$, preconditioner update frequency $u$, stochastic gradient batch size $b_g$, stochastic Hessian batch sizes $\{b_{h_j}\}$
        \FOR {$k = 0,1,\dots, m-1$}
                \STATE Sample a batch $B_k$
                \STATE Compute stochastic gradient $g_{B_{k}}(w_k)$
                \IF[Update preconditioner]{$k \equiv 0 \pmod u$}
                    \STATE Set $j = j+1$ 
                    \STATE Sample a batch $S_j$ \hfill \COMMENT{$|S_j| = b_{h_j}$}
                    \STATE $\Phi = \text{randn}(p, r_j)$ \hfill \COMMENT{Gaussian test matrix}
                    \STATE $Q = $ \texttt{qr\_econ} $(\Phi)$
                    \STATE Compute sketch $ Y = H_{S_j}(w_k)Q$ \hfill \COMMENT{$r$ calls to $\mathcal{O}_{H_{S_j}}$}
                    \STATE $[\hat{V},\hat{\Lambda}] = $ \texttt{RandNysApprox}$(Y,Q,r_j)$
                    \STATE $\eta_j = \texttt{get\_learning\_rate}(\mathcal{O}_{H_{S_j}},\hat{V},\hat{\Lambda},\rho)$
                \ENDIF
            \STATE Compute $v_k = (\hat{H}_{S_j}+\rho I)^{-1}g_{B_k}(w_k)$ via \eqref{eq:NysSMWSolve} 
            \STATE $w_{k+1} = w_{k}-\eta_j v_{k}$ \hfill \COMMENT{Update parameters}
        \ENDFOR
\end{algorithmic}
\end{algorithm} 
SketchySGD finds $w\in \mathbb{R}^p$ to minimize
the convex empirical risk
\begin{equation}
\label{eq:ERM-Prob}
    \mbox{minimize} \quad f(w) := \frac{1}{n}\sum_{i=1}^{n}f_i(w),
\end{equation}
given access to a gradient oracle for each $f_i$.
SketchySGD is formally presented as \cref{alg:SketchySGD}. 
SketchySGD tracks two different sets of indices: $k$, which counts the number of total iterations, 
and $j$, which counts the number of (less frequent) preconditioner updates.
SketchySGD updates the preconditioner (every $u$ iterations) by sampling a minibatch $S_j$ and forming a low-rank $\hat{H}_{S_j}$ using Hessian vector products with the minibatch Hessian $H_{S_j}$ evaluated at the current iterate $w_k$.
Given the Hessian approximation, it uses $\hat H_{S_j}+\rho I$ as a preconditioner, where $\rho>0$ is a regularization parameter.  
Subsequent iterates are then computed as
\begin{equation}
\label{eq:SketchySGDIter}
    w_{k+1} = w_k-\eta_j(\hat{H}_{S_j}+\rho I)^{-1}g_{B_k}(w_k),
\end{equation}
where $g_{B_k}(w_k)$ is the minibatch stochastic gradient and $\eta_{j}$ is the learning rate, which is automatically determined by the algorithm. 

The SketchySGD update may be interpreted as a preconditioned stochastic gradient step with Levenberg-Marquardt regularization \cite{levenberg1944method,marquardt1963algorithm}.
Indeed, let $P_j = \hat{H}_{S_j}+\rho I$ and define the preconditioned function $f_{P_j}(z) = f(P^{-1/2}_{j}z)$, which represents $f$ as a function of the preconditioned variable $z \in \R^p$.
Then \eqref{eq:SketchySGDIter} is equivalent to \footnote{See \Cref{subsection:PSGD} for a proof.}
\[
z_{k+1} = z_k-\eta_j\hat{g}_{P_j}(z_k), \quad w_{k+1} = P_j^{-1/2}z_k,
\]
where $\hat{g}_{P_j}(z_k)$ is the minibatch stochastic gradient as a function of $z$.
Thus, SketchySGD first takes a step of SGD in preconditioned space and then maps back to the original space. 
As preconditioning induces more favorable geometry, SketchySGD chooses better search directions and uses a stepsize adapted to the (more isotropic) preconditioned curvature. 
Hence SketchySGD makes more progress than SGD to ultimately converge faster. 
\paragraph{Contributions}

\begin{enumerate}
    \item We develop a new robust stochastic quasi-Newton algorithm that is fast and generalizes well 
    by accessing only a subsampled Hessian and stochastic gradient. 
    \item We devise an heuristic (but well-motivated) automated learning rate for this algorithm that works well in both ridge and logistic regression. 
    More broadly, we present default settings for all hyperparameters of SketchySGD, which allow it to work out-of-the-box.   
    \item We show that SketchySGD with a fixed learning rate converges to a small ball about the minimum for smooth and convex, and smooth and strongly convex objectives. 
    Additionally, we show SketchySGD converges at a faster rate than SGD for ill-conditioned least-squares problems. We verify this improved convergence in numerical experiments.
    \item We present experiments showing that SketchySGD with its default hyperparameters can match or outperform SGD, SVRG, SAGA, stochastic L-BFGS, and loopless Katyusha on ridge and logistic regression problems.
    \item We present proof-of-concept experiments on tabular deep learning, which
    show \newline SketchySGD scales well and is competitive with other stochastic second-order optimizers in deep learning, potentially providing an avenue for interesting future research.
\end{enumerate}

\subsection{Roadmap}
\Cref{section:SketchySGD} describes the SketchySGD algorithm in detail, explaining how to compute $\hat{H}_{S_k}$ and the update in \eqref{eq:SketchySGDIter} efficiently. \Cref{section:related} surveys previous work on stochastic quasi-Newton methods, particularly in the context of machine learning.
\Cref{section:Theory} establishes convergence of SketchySGD in convex machine learning problems. 
\Cref{section:Experiments} provides numerical experiments showing the superiority of SketchySGD relative to competing optimizers. 

\subsection{Notation}
\label{subsection:Notation}
Throughout the paper $B_{k}$ and $S_{j}$ denote subsets of $\{1,\ldots,n\}$ that are sampled independently and uniformly without replacement.
The corresponding stochastic gradient and minibatch Hessian are given by
\begin{align*}g_{B_k}(w) = \frac{1}{b_{g_k}}\sum_{i\in B_{k}}g_{i}(w), \quad H_{S_j}(w) = \frac{1}{b_{h_j}}\sum_{i\in S_{j}}\nabla^{2} f_{i}(w),
\end{align*}
where $b_{g_k} = |B_k|, b_{h_j} = |S_j|.$
For shorthand we often omit the dependence upon $w$ and simply write $g_{B_k}$ and $H_{S_j}$.
We also define $H(w)$ as the Hessian of the objective $f$ at $w$.
Throughout the paper $\|\cdot\|$ stands for the usual $2$-norm for vectors and operator norm for matrices.
For any matrix $A\in \R^{p\times p}$ and $u,v\in \R^p$, $\|v\|^2_A \coloneqq v^{T}Av$, and $\langle u,v \rangle_{A} = u^{T}Av$.
Given $w\in \R^p$, we define $M(w) = \max_{1\leq i\leq n}\|\nabla^2 f_i(w)\|$ and $G(w) = \max_{1\leq i\leq n}\|\nabla f_i(w)\|.$ 
Given any $\beta >0$, we use the notation $H^{\beta}_{S_j}$ to denote $H_{S_j}+\beta I$.
We abbreviate positive-semidefinite as psd, and positive-definite as pd. 
We denote the convex cone of real $p\times p$ symmetric psd (pd) matrices by $\mathbb{S}^{p}_{+}(\R)(\mathbb{S}^{p}_{++}(\R))$
We denote the Loewner order on the convex cone of psd matrices by $\preceq$, where $A\preceq B$ means $B-A$ is psd.
Given a psd matrix $A\in \mathbb{R}^{p\times p}$, we enumerate its eigenvalues in descending order, $\lambda_{1}(A)\geq \lambda_2(A)\geq \cdots \geq \lambda_p(A)$. 
Finally given a psd matrix $A$ and $\beta>0$ we define the effective dimension by $d^{\beta}_{\textrm{eff}}(A) = \tr(A(A+\beta I)^{-1})$, which provides a smoothed measure of the eigenvalues greater than or equal to $\beta$.
\section{SketchySGD: efficient implementation and hyperparameter selection}
\label{section:SketchySGD}
We now formally describe SketchySGD (\Cref{alg:SketchySGD}) and its efficient implementation. 

\paragraph{Hessian vector product oracle} 
SketchySGD relies on one main computational primitive, a (minibatch) Hessian vector product (hvp) oracle, to compute a low-rank approximation of the (minibatch) Hessian.
Access to such an oracle naturally arises in machine learning problems. In the case of generalized linear models (GLMs), the Hessian is given by $H(w) = \frac{1}{n} A^{T}D(w)A$, where $A\in \mathbb{R}^{n\times p}$ is the data matrix and $D\in \mathbb{R}^{n\times n}$ is a diagonal matrix. Accordingly, an hvp between $H_{S_j}(w)$ and $v \in \R^p$ is given by
\[
H_{S_j}(w)v =  \frac{1}{b_{h_j}}\sum_{i\in S_j}d_{i}(w)a_{i}\left(a_{i}^{T}v\right).
\]
For more complicated losses, an hvp can be computed by automatic differentiation (AD) \cite{pearlmutter1994fast}.
The general cost of $r$ hvps with $H_{S_j}(w)$ is $O(b_{h_j}pr)$.
In contrast, explicitly instantiating a Hessian entails a heavy $O(p^2)$ storage and $O(np^2)$ computational cost.
Further computational gains can be made when the subsampled Hessian enjoys more structure, such as sparsity. 
If $H_{S_j}(w)$ has $s$-sparse rows then the complexity of $r$ hvps enjoys a significant reduction from $O(b_{h_j}pr)$ to  $O(b_{h_j}sr)$. 
Hence, computing hvps with $H_{S_j}(w)$ is extremely cheap in the sparse setting. 



\paragraph{Randomized low-rank approximation} The hvp primitive allows for efficient randomized low-rank approximation to the minibatch Hessian by sketching. 
Sketching reduces the cost of fundamental numerical linear algebra operations without much loss in accuracy \cite{woodruff2014sketching,martinsson2020randomized} by computing the quantity of interest from a \emph{sketch}, or randomized linear image, of a matrix.
In particular, sketching enables efficient computation of a near-optimal low-rank approximation to $H_{S_j}$ \cite{halko2011finding,cohen2015dimensionality,tropp2017practical}. 

SketchySGD computes a sketch of the subsampled Hessian using hvps and returns a randomized low-rank approximation $\hat{H}_{S_j}$ of $H_{S_j}$ in the form of an eigendecomposition $\hat{V}\hat{\Lambda}\hat{V}^{T}$, where $\hat{V}\in \mathbb{R}^{p\times r}$ and $\hat{\Lambda}\in \mathbb{R}^{r\times r}$. 
In this paper, we use the randomized Nystr{\"o}m approximation, following the stable implementation in \cite{tropp2017fixed}.
The resulting algorithm \texttt{RandNysApprox} appears in \Cref{section:RandNysAppx}.
The cost of forming the Nystr{\"o}m approximation is $O(b_{h_j}pr+pr^2)$, as we need to perform $r$ minibatch hvps to compute the sketch, and we must perform a skinny SVD at a cost of $O(pr^2)$.
The procedure is extremely cheap, as we find empirically we can take $r$ to be $10$ or less, so constructing the low-rank approximation has negligible cost. 

\begin{remark}
If the objective for $f$ includes an $l_2$-regularizer $\gamma$, so that the subsampled Hessian has the form $H_{S_j}(w) = \frac{1}{b_{h_j}}\sum_{i\in S_j}\nabla^2f_i(w)+\gamma I$, we do not include $\gamma$ in the computation of the sketch in \cref{alg:SketchySGD}. 
The sketch is only computed using minibatch hvps with $ \frac{1}{b_{h_j}}\sum_{i\in S_j}\nabla^2f_i(w)$.
\end{remark}

\paragraph{Setting the learning rate} 
SketchySGD automatically selects the learning rate $\eta$ whenever it updates the preconditioner.
The learning rate update rule is inspired by the analysis of gradient descent (GD) on smooth convex functions, 
which shows GD converges for a fixed learning rate $ \eta = 1/L$, where $L$ is the smoothness constant. 
In the preconditioned setting, $L$ is replaced by $L_P$, its preconditioned analogue. 
SketchySGD approximates the ideal learning rate $1/L_P$ by setting the learning rate as
\begin{equation}
\label{eq:SSG_LR}
\eta_{\textrm{SketchySGD}} = \frac{\alpha}{\lambda_{1}\left((\hat{H}_{S_j}+\rho I)^{-1/2}H_{S'}(w_j)(\hat{H}_{S_j}+\rho I)^{-1/2}\right)},    
\end{equation}
where $S'$ is a fresh minibatch that is sampled independently of $S_j$, and $\alpha$ is scaling factor whose default value is $1/2$.
The following logic provides intuition for this choice: if $f$ has a Lipschitz Hessian \cref{prop:NysPrecondLem}.
\begin{align*}
&\lambda_1\left((\hat{H}_{S_j}+\rho I)^{-1/2}H_{S'}(w_j)(\hat{H}_{S_j}+\rho I)^{-1/2}\right) \\
&\overset{(1)}{\approx} \lambda_1\left(\hat{H}_{S_j}+\rho I)^{-1/2}H(w_j)(\hat{H}_{S_j}+\rho I\right)^{-1/2} \overset{(2)}{\leq} 1+\zeta \overset{(3)}{\approx} L_{P}^{\textrm{loc}}(w_j),
\end{align*}
where $\zeta\in (0,1)$, and $L_{P}^{\textrm{loc}}(w)$ is the local preconditioned smoothness constant in some appropriately sized ball centered about $w_j$.
Here $(1)$ is due to \cref{lemma:SubsampAppx}, $(2)$ follows from \cref{prop:NysPrecondLem}, and $(3)$ follows $f$ having a Lipschitz Hessian.
Hence SketchySGD 
is expected to work well as long as the local quadratic model at $w_j$ provides an accurate description of the curvature.
Moreover, this step size $\eta_{\textrm{SketchySGD}} = \bigO(1)$ is much larger (in the preconditioned space) than the standard step size of $1/L$ (in the original space) for an ill-conditioned problem,
and should allow faster convergence to the minimum. 
These predictions are verified experimentally in \Cref{section:Experiments}.

Crucially, our proposed learning rate can be efficiently computed via matvecs with $H_{S_j}$ and $(\hat{H}_{S_j}+\rho I)^{-1/2}$ using techniques from randomized linear algebra, such as randomized powering and the randomized Lanczos method \cite{kuczynski1992estimating,martinsson2020randomized}.
As an example, we show how to compute the learning rate using randomized powering in \cref{alg:learning_rate}.
Note for any vector $v$, $(\hat{H}_{S_j}+\rho I)^{-1/2}v$ can be computed efficiently via the formula
\begin{equation}
\label{eq:RootNysSMWSolve}
    (\hat{H}_{S_j}+\rho I)^{-1/2}v = \hat{V}\left(\hat{\Lambda}+\rho I\right)^{-1/2}\hat{V}^{T}v+\frac{1}{\sqrt{\rho}}(v-\hat{V}\hat{V}^{T}v).
\end{equation}

\begin{algorithm}[tb]
   \caption{\texttt{get\_learning\_rate}}
   \label{alg:learning_rate}
    \begin{algorithmic}
       \STATE {\bfseries Input:} hvp oracle $\mathcal{O}_{H}$, Nystr{\"o}m approximation factors $\hat{V}, \hat{\Lambda}$, regularization $\rho$, maximum number of iterations $q$
       \STATE $z$ = \text{randn}$(p,1)$
       \STATE $y_0 = z/\|z\|$
       \FOR{$i=1,\dots,q$}
       \STATE Compute $v = \left(\hat{H}_{S_j}+\rho I\right)^{-1/2}y_{i-1}$ \hfill \COMMENT{Use \eqref{eq:NysSMWSolve}}
       \STATE Compute $v' = H_{S'}v$ by calling oracle $\mathcal{O}_{H_{S_k}}$
       \STATE Compute $y_i = \left(\hat{H}_{S_j}+\rho I\right)^{-1/2}v'$ \hfill \COMMENT{Use \eqref{eq:RootNysSMWSolve}}
       \STATE $\lambda_i = y_{i-1}^{T}y_i$
       \STATE $y_i = y_i/\|y_i\|$
       \ENDFOR
       \STATE Set $\eta = 1/\lambda_q$
       \STATE {\bfseries Return:} $\eta$
    \end{algorithmic}
\end{algorithm} 



\paragraph{Computing the SketchySGD update \eqref{eq:SketchySGDIter} fast}
Given the rank-$r$ approximation $\hat H_{S_j} = \hat V \hat \Lambda \hat V^\top$ to the minibatch Hessian $H_{S_j}$, the main cost of SketchySGD relative to standard SGD is computing the search direction $v_k = (\hat{H}_{S_j}+\rho I)^{-1} g_{B_k}$. This (parallelizable) computation requires $O(pr)$ flops, by the matrix inversion lemma \cite{higham2002accuracy}:
\begin{equation}
\label{eq:NysSMWSolve}
    v_k = \hat{V}\left(\hat{\Lambda}+\rho I\right)^{-1}\hat{V}^{T}g_{B_k}+\frac{1}{\rho}(g_{B_k}-\hat{V}\hat{V}^{T}g_{B_k}),
\end{equation}
The SketchySGD preconditioner is  easy to use, fast to compute, 
and allows SketchySGD to scale to massive problem instances.
  
\paragraph{Default parameters for \cref{alg:SketchySGD}}
We recommend setting the ranks $\{r_j\}$ to a constant value of $10$, the regularization $\rho = 10^{-3}L$, and the stochastic Hessian batch sizes $\{b_{h_j}\}$ to a constant value of $\left\lfloor \sqrt{n_{\mathrm{tr}}} \right\rfloor$, where $n_{\mathrm{tr}}$ is the size of the training set. 
When the Hessian is constant, i.e. as in least-squares/ridge-regression, we recommend using the preconditioner throughout the optimization, which corresponds to $u = \infty$. 
In settings where the Hessian is not constant, we recommend setting $u = \left\lceil \frac{n_{\textrm{tr}}}{b_g} \right\rceil$, which corresponds to updating the preconditioner after each pass through the training set.

\section{Comparison to previous work} 
\label{section:related}
Here we review prior work on stochastic quasi-Newton methods, with particular emphasis on those developed for convex optimization problems, which is the main focus of this paper. 
\paragraph{Stochastic quasi-Newton methods for convex optimization}
Many authors have developed stochastic quasi-Newton methods for large-scale machine learning.  
Broadly, these schemes can be grouped by whether they use a stochastic approximations to both the Hessian and gradient or just a stochastic approximation to the Hessian. 
Stochastic quasi-Newton methods that use exact gradients with a stochastic Hessian approximation constructed via sketching or subsampling include \cite{byrd2011use,erdogdu2015convergence,berahas2016multi,pilanci2017newton,gower2019rsn,ye2021Appx,chen2022san,yuan2022sketched}.
Methods falling into the second group include \cite{moritz2016linearly,bollapragada2018progressive,roosta2019sub,bollapragada2019exact,meng2020fast}.

Arguably, the most common strategy employed by stochastic quasi-Newton methods is to subsample the Hessian as well as the gradient.
These methods either directly apply the inverse of the subsampled Hessian to the stochastic gradient \cite{roosta2019sub,bollapragada2019exact,meng2020fast}, or they do an L-BFGS-style update step with the subsampled Hessian \cite{moritz2016linearly,bollapragada2018progressive,meng2020fast}.
However, the theory underlying these methods requires large or growing gradient batch sizes \cite{roosta2019sub,bollapragada2018progressive,bollapragada2019exact}, periodic full gradient computation \cite{moritz2016linearly}, or interpolation \cite{meng2020fast}, which are unrealistic assumptions for large-scale convex problems.
Further, many of these methods lack practical guidelines for setting hyperparameters such as batch sizes and learning rate, leading to the same tuning issues that plague stochastic first-order methods.

SketchySGD improves on many of these stochastic quasi-Newton methods by providing principled guidelines for selecting hyperparameters and requiring only a modest, constant batch size.
\cref{t-2ndOrdComp} compares SketchySGD with a representative subset of stochastic second-order methods on gradient and Hesssian batch sizes, and iteration complexity required to reach a \emph{fixed suboptimality} $\epsilon>0$.
Notice that SketchySGD is the only method that allows computing the gradient with a small constant batch size!
Hence SketchySGD empowers the user to select the gradient batch size that meets their computational constraints.
Although for fixed-suboptimality $\epsilon$, SketchySGD only converges at rate of $\bigO\left(\log(1/\epsilon)/\epsilon\right)$, this is often sufficient for machine learning optimization, as it is well-known that high-precision solutions do not improve generalization \cite{bottou2007tradeoffs,agarwal2012fast,loh2017lower}.
Moreover, most of the methods that achieve $\epsilon$-suboptimality in $\bigO(\log(1/\epsilon))$ iterations, require full gradients, which result in costly iterations when $n$ and $p$ are large.
In addition to expensive iterations, full gradient methods can also be much slower in yielding a model with good generalization error \cite{amir2021never}. 
Thus, despite converging linearly to the minimum, full gradient methods are expensive and can be slow to reach good generalization error. 

SketchySGD also generalizes subsampled Newton methods; 
by letting the rank parameter $r_j \to \mathrm{rank}(H_{S_j}) \leq p$, SketchySGD reproduces the algorithm of \cite{roosta2019sub,bollapragada2019exact,meng2020fast}.
This follows as the error in the randomized Nystr{\"o}m approximation is identically zero when $r_j = \mathrm{rank}(H_{S_j})$, so that $\hat H_{S_j} = H_{S_j}$ \cite{tropp2017fixed}.
By using a full batch gradient $b_g = n$ (and a regularized exact low-rank approximation to the subsampled Hessian), SketchySGD reproduces the method of \cite{erdogdu2015convergence}. 
In particular, these methods are made practical by the analysis and practical parameter selections in this work.
\begin{table}[!htbp]
	\caption{\textbf{Comparison of stochastic 2nd-order methods.} Fix $\epsilon>0$, and suppose $f$ is of the form \eqref{eq:ERM-Prob}, and is strongly convex. 
    This table compares the required gradient and Hessian batch sizes of various stochastic quasi-Newton methods, and the number of iterations required to output a point satisfying $f(w)-f(w_\star)\leq \epsilon$.
    Here $\gamma,\zeta\in (0,1)$, while $G(w)$ and $M(w)$ are as in \Cref{subsection:Notation}, and $\tau^\rho(H(w))$ denotes the $\rho$-dissimilarity, which is defined in \Cref{def:rho-sim}. 
    The $\rho$-dissimilarity offers the tightest characterization of the required Hessian minibatch size required to ensure a non-trivial approximation of $H^\rho$. In many settings of interest, it is much smaller than $n$, see \cref{prop:tau_rho} and the corresponding discussion. 
    \label{t-2ndOrdComp}} 
	\begin{center}
    \scriptsize
		\begin{tabular}{C{3cm}C{2cm}C{3cm}C{5cm}}
			\hline
			\textbf{Method} & \textbf{Gradient batch size} & \textbf{Hessian batch size} & \textbf{Iteration complexity}\\ 
            \hline
			Newton Sketch \cite{pilanci2017newton,lacotte2021adaptive} & Full & Full & $\bigO\left(\kappa^2\log(1/\epsilon)\right)$\\
			\hline
			Subsampled Newton \newline (Full gradients) \cite{roosta2019sub,ye2021Appx} & Full & $\widetilde{\bigO}\left(\frac{M(w)/\mu}{\zeta^2}\right)$ & $\bigO\left(\kappa^2\log(1/\epsilon)\right)$ \\ 
			\hline
			Subsampled Newton \newline (Stochastic gradients) \cite{roosta2019sub} & $\widetilde{\bigO}\left(G(w)^2/\gamma^{2}\right) $ & $\widetilde{\bigO}\left(\frac{M(w)/\mu}{\zeta^2}\right)$ & $\bigO(\kappa^2\log(1/\epsilon))$ \\ 
			\hline
			Subsampled Newton\newline (Low-rank) \cite{erdogdu2015convergence,ye2021Appx} & Full & $\widetilde{\bigO}\left(\frac{M(w)/\lambda_{r+1}(H(w))}{\zeta^2}\right)$ & $\bigO\left(\kappa^2\log(1/\epsilon)\right)$ \\
			\hline
            SLBFGS \cite{moritz2016linearly} & Full evaluation \newline every epoch & $b_h$ & $\bigO\left(\kappa^2\log(1/\epsilon)\right)$\\
			\hline
			SketchySGD (\cref{alg:SketchySGD}) & $b_g$ & $\widetilde{\bigO}\left(\tau^\rho(H(w))/\zeta^2\right)$ & $\bigO\left(\left[\frac{\mathcal L_P}{\gamma_\ell}\frac{\rho}{\mu}+\frac{\sigma^2/(\gamma_\ell\mu)}{\epsilon}\right]\log(1/\epsilon)\right)$ \\
			\hline
		\end{tabular}
	\end{center}
\end{table}

\paragraph{Stochastic quasi-Newton methods for non-convex optimization}
In the past decade, there has been a surge of interest in stochastic second-order methods for non-convex optimization, primarily driven by deep learning. 
Similar to the convex setting, many of these methods are based on subsampling the Hessian, which is then combined
with cubic regularization \cite{kohler2017sub,tripuraneni2018stochastic,xu2020newton} or trust region methods \cite{yao2021inexact,roosta2022newton}.
Although these methods come with strong theoretical guarantees, they have not proven popular in deep learning, due to subproblems that are expensive to solve.
In order to maintain computational tractability, most stochastic quasi-Newton methods designed for deep learning
forgo theoretical guarantees in favor of scalability and good empirical performance.
Popular stochastic second-order optimizers in deep learning include K-FAC \cite{grosse2016kronecker}, Shampoo \cite{gupta2018shampoo}, and AdaHessian \cite{yao2021adahessian}.
Despite recent advances in stochastic second-order methods for deep learning, the advantage of stochastic quasi-Newton methods over SGD and its variants is unclear, so stochastic first-order methods have remained the most popular optimization algorithms for deep learning.

\section{Theory}
\begin{algorithm}[tb]
   \caption{SketchySGD (Theoretical version)}
   \label{alg:SketchySGD}
    \begin{algorithmic}
       \STATE {\bfseries Input:} initialization $w_0$, learning rate $\eta$, hvp oracle $\mathcal{O}_{H}$, ranks $\{r_j\}$, regularization $\rho$, preconditioner update frequency $u$, stochastic gradient batch size $b_g$, stochastic Hessian batch sizes $\{b_{h_j}\}$
       \FOR{$s = 1,2\dots$}
            \FOR {$k = 0,2,\dots, m-1$}
                \STATE Sample a batch $B^{(s)}_k$
                \STATE Compute stochastic gradient $g_{B^{(s)}_{k}}(w^{(s)}_k)$
                \IF{$sk \equiv 0 \pmod u$}
                    \STATE Set $j = j+1$ 
                    \STATE Sample a batch $S_j$ \hfill \COMMENT{$|S_j| = b_{h_j}$}
                    \STATE $\Phi = \text{randn}(p, r_j)$ \hfill \COMMENT{Gaussian test matrix}
                    \STATE $Q = $ \texttt{qr\_econ} $(\Phi)$
                    \STATE Compute sketch $ Y = H_{S_j}(w_k)Q$ \hfill \COMMENT{$r$ calls to $\mathcal{O}_{H_{S_j}}$}
                    \STATE $[\hat{V},\hat{\Lambda}] = $ \texttt{RandNysApprox}$(Y,Q,r_j)$
                \ENDIF
            \STATE Compute $v^{(s)}_k = (\hat{H}_{S_j}+\rho I)^{-1}g_{B^{(s)}_k}(w^{(s)}_k)$ via \eqref{eq:NysSMWSolve} 
            \STATE $w^{(s)}_{k+1} = w^{(s)}_{k}-\eta v^{(s)}_{k}$ \hfill \COMMENT{Update parameters}
            \ENDFOR
            \STATE Set $\hat w^{(s+1)} = \frac{1}{m}\sum_{k=0}^{m-1}w^{(s)}_k$.
        \ENDFOR
\end{algorithmic}
\end{algorithm} 
In this section we present our main convergence theorems for SketchySGD.
As mentioned in the prequel, we do not directly analyze \cref{alg:SketchySGD_pract}, but a slightly modified version, which we present in \cref{alg:SketchySGD}.
The are two differences between \cref{alg:SketchySGD_pract} and \cref{alg:SketchySGD}. 
First, \cref{alg:SketchySGD_pract} uses an adaptive learning rate strategy, while \cref{alg:SketchySGD} uses a fixed learned rate. 
Second, \cref{alg:SketchySGD} breaks the optimization into \emph{stages} involving periodic averaging. 
At the end of each stage, \cref{alg:SketchySGD} sets the initial iterate for the next stage to be the average of the iterates from the previous stage.
In this sense, \cref{alg:SketchySGD} resembles the SVRG algorithm of \cite{johnson2013accelerating}, except there is no full gradient computation. 
Just as with SVRG, the addition of averaging is needed purely to facilitate analysis; in practice the periodic averaging in \cref{alg:SketchySGD} yields no benefits. 
We recommend always running \cref{alg:SketchySGD_pract} in practice, which is the version we use for all of our experiments (\Cref{section:Experiments}). 


\label{section:Theory} 
\subsection{Assumptions}
\label{sec:assump}
We show convergence of SketchySGD when $f$ is smooth (\Cref{assump:DiffandSmoothness}) and strongly convex (\Cref{assump:StrCvx}).

\begin{assumption}[Differentiability and smoothness]
\label{assump:DiffandSmoothness}
The function $f$ is twice differentiable and $L$-smooth. Further, each $f_i$ is $L_i$-smooth with $L_i \leq L_\textup{max}$ for every $i=1,\ldots,n$. 
\end{assumption}

\begin{assumption}[Strong convexity]
\label{assump:StrCvx}
    The function $f$ is $\mu$-strongly convex for some $\mu > 0$.
\end{assumption}

\subsection{Quadratic Regularity}
Our analysis rests on the idea of relative upper and lower quadratic regularity.
This is a generalization of upper and lower quadratic regularity, which was recently introduced by the authors in \cite{frangella2023promise}, and refines the ideas of relative convexity and relative smoothness introduced in \cite{gower2019rsn}.

\begin{definition}[Relative quadratic regularity]
\label{def:RelSmoothCvx}
Let $f$ be a twice differentiable function and $A(w):\R^{p}\mapsto \mathbb{S}^{p}_{++}(\R)$. 
Then $f$ is said to be \emph{relatively upper quadratically regular} with respect to $A$, if for all $w,w',w''\in \R^p$ there exists $0<\gamma_u<\infty$, such that
\begin{equation}
    f(w')\leq f(w)+\langle g(w), w'-w\rangle +\frac{\gamma_u}{2}\|w'-w\|^2_{A(w'')} \label{eq:up_quad}.
\end{equation}
Similarly, $f$ is said is to be \emph{relatively lower quadratically regular}, if for all $w,w',w''\in \R^p$ there exists $0<\gamma_{\ell}<\infty$, such that
\begin{equation}
    f(w')\geq f(w)+\langle g(w), w'-w\rangle +\frac{\gamma_\ell}{2}\|w'-w\|^2_{A(w'')} \label{eq:low_quad}.
\end{equation}
We say $f$ is \emph{relatively quadratically regular} with respect to $A$ if $\gamma_u<\infty$ and $\gamma_\ell >0$.
\end{definition}
When $A(w) = H(w)$, the Hessian of $f$, relative quadratic regularity reduces to quadratic regularity from \cite{frangella2023promise}.  
Quadratic regularity extends ideas from \cite{gower2019rsn}, by having the Hessian be evaluated at a point $w''\neq w$ in \eqref{eq:up_quad}--\eqref{eq:low_quad}.
This extension, while simple in nature, is essential for establishing convergence under lazy preconditioner updates.

An important thing to note about relative quadratic regularity is it holds for useful settings of $A$, under standard hypotheses, as shown by the following lemma.
\begin{lemma}[Smoothness and strong convexity implies quadratic regularity]
Let $h:\mathcal C\rightarrow \R$, where $\mathcal C$ is a closed convex subset of $\R^p$. Then the following items hold 
\label{lemma:rel_quad}
    \begin{enumerate}
        \item If $h$ is twice differentiable and $L$-smooth, then for any $\rho >0$, $f$ is relatively upper quadratically regular with respect to $\nabla^2h(w)+\rho I$.
        \item If $h$ is twice differentiable, $L$-smooth, and $\mu$-strongly convex, then $h$ is relatively quadratically regular with respect to $\nabla^2h(w)$.
    \end{enumerate}
\end{lemma}
A proof of \cref{lemma:rel_quad} may be found in \cref{subsection:quad_reg_pf}.
Importantly, \eqref{eq:up_quad} and \eqref{eq:low_quad} hold with non-vacuous values of $\gamma_u$ and $\gamma_\ell$.
In the case of least-squares $\gamma_u = \gamma_\ell = 1$. 
More generally, for strongly convex generalized linear models, it can be shown when $A(\cdot) = H(\cdot)$, that $\gamma_u$ and $\gamma_\ell $ are independent of the condition number \cite{frangella2023promise}, similar to the result of \cite{gower2019rsn} for relative smoothness and relative convexity.
Thus, for many popular machine learning problems, the ratio $\gamma_u/\gamma_\ell$ is independent of the conditioning of the data.

\subsection{Quality of SketchySGD preconditioner}
To control the batch size used to form the subsampled Hessian, we introduce \emph{$\rho$-dissimilarity}. 
\begin{definition}[$\rho$-dissimilarity]
\label{def:rho-sim}
    Let $H(w)$ be the Hessian at $w$. The \emph{$\rho$-dissimilarity} is
    \[
    \tau^{\rho}(H(w)) = \max_{1\leq i\leq n}\lambda_1\left((H(w)+\rho I)^{-1/2}(\nabla^2 f_i(w)+\rho I)(H(w)+\rho I)^{-1/2}\right).
    \]
\end{definition}
$\rho$-dissimilarity may be viewed as an analogue of coherence from compressed sensing and low-rank matrix completion \cite{candes2007sparsity,candes2012exact}. 
Similar to how the coherence parameter measures the uniformity of the rows of a matrix, $\tau^{\rho}(H(w))$ measures how uniform the curvature of the sample $\{\nabla^2 f_i(w)\}_{1\leq i\leq n}$ is.
Intuitively, the more uniform the curvature, the better the sample average $H(w)$ captures the curvature of each individual Hessian, which corresponds to smaller $\tau^\rho(H(w))$. 
On the other hand, if the curvature is highly non-uniform, curvature information of certain individual Hessians will be in disagreement with that of $H(w)$, leading to a large value of $\tau^{\rho}(H(w))$.

The following lemma provides an upper bound on the $\rho$-dissimilarity. 
In particular, it shows that the $\rho$-dissimilarity never exceeds $n$.
The proof may be found in \Cref{subsection:tau_rho_lessn_pf}.
\begin{lemma}[$\rho$-dissimilarity never exceeds $n$]
\label{lemma:tau_rho_bnd}
    For any $\rho\geq 0$ and $w \in \R^p$, the following inequality holds
    \[
    \tau^{\rho}(H(w))\leq \min\left\{n,\frac{M(w)+\rho}{\mu+\rho}\right\},
    \]
    where $M(w) = \max_{1\leq i\leq n}\lambda_{1}(\nabla^2f_i(w)).$
\end{lemma}
\Cref{lemma:tau_rho_bnd} provides a worst-case bound on the $\rho$-dissimilarity---if the curvature of the sample is highly non-uniform or $\rho$ is very small, then the $\rho$-dissimilarity can be large as $n$.
However, \cref{lemma:tau_rho_bnd} neglects the fact that in machine learning, the $f_i$'s are often similar to one another,  so $\tau^\rho(H(w))$ ought to be much smaller than $n$.

Clearly, for arbitrary data distributions the $\rho$-dissimilarity can be large. 
The following proposition shows when $f$ is a GLM, and the data satisfies an appropriate sub-Gaussian condition, the $\rho$-dissmilarity does not exceed the $\rho$-effective dimension of the population Hessian. 
 
\begin{proposition}[$\rho$-dissimilarity is small for GLMs the machine learning setting]
\label{prop:tau_rho}
Let $\ell:\R\mapsto \R$ be a smooth and convex loss, and define $f(w) =\frac{1}{n}\sum_{i=1}^n f_i(w) $ where $f_i(w) = \ell(x_i^Tw)$. 
Fix $w\in \R^p$. 
Assume $x_i$ are drawn i.i.d. from some unknown distribution $\P(x)$ for $i\in\{1,\dots,n\}$. 
Let $H_{\infty}(w) = \E_{x\sim \P}[\ell''(x^{T}w)xx^T]\succ 0$ be the population Hessian matrix, 
and set $\bar{d}^{\rho}_{\textup{eff}}(H_{\infty}(w)) = \max\{d^{\rho}_{\textup{eff}}(H_{\infty}(w)),1\}$. 
Suppose for some constant $\nu$ the following conditions hold:
    \begin{enumerate}
    \item [i.] The random vector $z = H_{\infty}(w)^{-1/2}\sqrt{\ell''(x^{T}w)}x$ is $\nu$ sub-Gaussian.
    \item[ii.] $n \geq C\bar d^{{\rho}}_{\textup{eff}}(H_\infty(w))\log\left(\frac{n}{\delta}\right)\log\left(\frac{d^{\rho}_{\textup{eff}}(H_{\infty}(w))}{\delta}\right)$.
    \end{enumerate}
    Then with probability at least $1-\delta$,
    \[
    \tau^{\rho}(H(w)) = \bigO\left(\bar d^{\rho}_{\textup{eff}}(H_{\infty}(w))\log\left(\frac{n}{\delta}\right)\right).
    \]

\end{proposition}
The proof of \cref{prop:tau_rho} is given in \Cref{subsection:tau_rho_pf}, 
and is based on showing the $\rho$-dissimlarity is well-behaved at the population level, 
and that for large enough $n$, the empirical Hessian concentrates the population Hessian.

\Cref{prop:tau_rho} shows if $f$ is a GLM, then for large datasets, 
$\tau^{\rho}(H(w)) = \tilde {\mathcal{O}}\left(d^{\rho}_{\textup{eff}}(H_{\infty}(w))\right)$ with high probability.
When the eigenvalues of $H_{\infty}(w)$ decay rapidly, the effective dimension $d^{\rho}_{\textup{eff}}(H_{\infty}(w))$ should be smaller than $(M(w)+\rho)/(\rho+\mu)$, 
so \cref{prop:tau_rho} yields a stronger bound than \cref{lemma:tau_rho_bnd}.
For example, 
when the eigenvalues of $H_{\infty}(w)$ at a sufficiently fast polynomial rate, 
it is easily verified that $d^{\rho}_{\textup{eff}}(H_{\infty}(w)) = \mathcal{O}(1/\sqrt{\rho})$ \cite{bach2013sharp}.
Consequently $\tau^{\rho}(H(w)) = \tilde{\mathcal{O}}(1/\sqrt{\rho})$, which is a significant improvement over the $\mathcal{O}(1/\rho)$ bound of \cref{lemma:tau_rho_bnd} when $\rho$ is small.
This is crucial, for it is desirable to set $\rho$ small, as this leads to a smaller preconditioned condition number, see \cref{prop:NysPrecondLem}.
As polynomial (or faster) decay of the eigenvalues values is common in machine learning problems \cite{derezinski2020precise}, 
$\tau^{\rho}(H(w))$ will typically be much smaller then $\bigO(1/\rho)$.


\begin{lemma}[Closeness in Loewner ordering between $H^\rho_S(w)$ and $H^\rho(w)$]
\label{lemma:SubsampAppx}
Let $\zeta \in(0,1)$, $w\in \R^p$ , and $\rho \geq 0$. 
Construct $H_{S}$ with batch size $b_{h} = \bigO\left(\frac{\tau^{\rho}(H(w))\log\left(\frac{d^{\rho}_{\textup{eff}}(H(w))}{\delta}\right)}{\zeta^2}\right)$.
Then with probability at least $1-\delta$ 
\[
(1-\zeta)H^\rho_{S}(w)\preceq H^\rho(w) \preceq (1+\zeta) H^\rho_{S}(w).
\]
\end{lemma}
The proof of \cref{lemma:SubsampAppx} is provided in \Cref{subsection:subsamp_appx_pf}.
\Cref{lemma:SubsampAppx} refines prior analyses such as \cite{ye2021Appx} (which itself refines the analysis of \cite{roosta2019sub}), where $b_h$ depends upon  $(M(w)+\rho)/(\mu+\rho)$, which \Cref{lemma:tau_rho_bnd} shows is always larger than $\tau^{\rho}(H(w))$.
Hence the dependence upon $\tau^{\rho}(H(w))$ in \Cref{lemma:SubsampAppx} leads to a tighter bound on the required Hessian batch size. 
More importantly, \cref{lemma:SubsampAppx} and the idealized setting of \cref{prop:tau_rho} help show why algorithms using minibatch Hessians with small batchsizes are able to succeed, a phenomenon that prior worst-case theory is unable to explain.  
As a concrete example, adopt the setting of \cref{prop:tau_rho}, 
assume fast eigenvalue decay of the Hessian, 
and set $\rho = \bigO(1/n)$.
Then \cref{lemma:SubsampAppx} gives $b_h = \tilde{\bigO}(\sqrt{n})$, whereas prior analysis based on $(M(w)+\rho)/(\mu+\rho)$ yields a vacuous batch size of $b_h = \tilde{\bigO}(n)$.
Thus, \cref{lemma:SubsampAppx} supports taking batch sizes much smaller then $n$. 
Motivated by this discussion, we recommend a default batch size of $b_h = \sqrt{n}$, which leads to excellent performance in practice; see \Cref{section:Experiments} for numerical evidence.

Utilizing our results on subsampling and ideas from randomized low-rank approximation, we are able to establish the following result, which quantifies how the SketchySGD preconditioner reduces the condition number. 
\begin{proposition}[Closeness in Loewner ordering between $H(w)$ and $\hat{H}_{S}^{\rho}$]
\label{prop:NysPrecondLem}
Let $\zeta \in $$(0,1)$ and $w\in \R^p$. 
Construct $H_{S}(w)$ with batch size $b_{h} = \bigO\left(\frac{\tau^\rho(H(w))\log\left(\frac{d^{\rho}_{\textup{eff}}(H(w))}{\delta}\right)}{\zeta^2}\right)$ and SketchySGD uses a low-rank approximation $\hat{H}_{S}$ to $H_{S}(w)$ with rank $r= \bigO(d^{\zeta\rho}_{\textup{eff}}(H_S(w))+\log(\frac{1}{\delta}))$.
Then with probability at least $1-\delta$,
\begin{equation}
    (1-\zeta)\frac{1}{1+\rho/\mu}\hat{H}_{S}^{\rho}\preceq H(w) \preceq (1+\zeta)\hat H^{\rho}_{S}.
\end{equation}
\end{proposition}
The proof of this proposition is given in \Cref{subsection:NysPrecondPf}.
\Cref{prop:NysPrecondLem} shows that with high probability, the SketchySGD preconditioner reduces the conditioner number from $L/\mu$ to $(1+\rho/\mu)$, which yields an $L/\rho$ improvement over the original value. 
The proposition reveals a natural trade-off between eliminating dependence upon $\mu$ and the size of $b_h$: as $\rho$ decreases to $\mu$ (and the preconditioned condition number becomes smaller), the batch size must increase.  
In practice, we have found that a fixed-value of $\rho = 10^{-3}L$ yields excellent performance for convex problems.
Numerical results showing how the SketchySGD preconditioner improves the conditioning of the Hessian throughout the optimization trajectory are presented in \Cref{fig:sensitivity_r}.

\cref{prop:NysPrecondLem} requires the rank of $\hat H_{S_j}$ to satisfy $r_j = \tilde \bigO\left(d^{\zeta\rho}_{\textrm{eff}}(H_{S_j}(w_j)\right)$, 
which ensures $\|\hat H_{S_j}-H_{S_j}\|\leq \zeta \rho$ holds with high probability (\cref{lemma:NysErrLemma})
so that the approximate Hessian matches the subsampled Hessian up to the level of the regularization $\rho$.
  
\subsection{Controlling the variance of the preconditioned stochastic gradient}
To establish convergence of SketchySGD, we must control the second moment of the preconditioned minibatch stochastic gradient.
Recall the usual approach for minibatch SGD.
In prior work, \cite{gower2019sgd} showed that when each $f_i$ is smooth and convex, the minibatch stochastic gradient of $f$ satisfies the following \emph{expected smoothness} condition:
\begin{align}
\label{eq:ES_bound}
&\E\|g_{B}(w)-g_{B}(w_\star)\|^2\leq 2\mathcal L(f(w)-f(w_\star)),\\
&\E|g_B(w)\|^2 \leq 2\mathcal L(f(w)-f(w_\star))+2\sigma^2,\\
&\mathcal L = \frac{n(b_g-1)}{b_g(n-1)}L+\frac{n-b_g}{b_g(n-1)}L_{\textrm{max}},~\sigma^2 = \frac{n-b_g}{b_g(n-1)}\frac{1}{n}\sum_{i=1}^{n}\|\nabla f_i(w_\star)\|^2.
\end{align}
Building on the analysis of \cite{gower2019sgd}, 
we prove the following proposition, which directly bounds the second moment of the preconditioned stochastic gradient.
\begin{proposition}
[Preconditioned expected smoothness and gradient variance]
\label{prop:PrecondSmoothGrad}
    Suppose that \cref{assump:DiffandSmoothness} holds, $P = \hat H_{S}^\rho$ is constructed at $w_P\in \R^p$, and $P$ satisfies $H(w_P)\preceq (1+\zeta)P$. 
    Then the following inequalities hold:
    \begin{align*}
        \mathbb E_k\|g_{B}(w)-g_{B}(w_\star)\|_{P^{-1}}^2\leq 2\mathcal L_P \left(f(w)-f(w_\star)\right),
    \end{align*}
    \begin{align*}
        \mathbb E_k\|g_{B}(w)\|_{P^{-1}}^2 \leq 4 \mathcal L_P \left(f(w)-f(w_\star)\right)+\frac{2\sigma^2}{\rho},
    \end{align*}
    where $\mathcal L_P$ and $\sigma^2$ are given by
    \[
    \mathcal L_P \coloneqq \left[\frac{n(b_g-1)}{b_g(n-1)}\gammaRho+\frac{n-b_g}{b_g(n-1)}\tau^\rho(H(w_P))\gammaMax\right](1+\zeta),~
    \sigma^2 \coloneqq \frac{n-b_g}{b_g(n-1)}\frac{1}{n}\sum_{i=1}^{n}\|\nabla f_i(w_\star)\|^2.
    \]
    Here, $\gamma^{\rho}_u$ is the relative upper quadratic regularity constant of $f$ with respect to $H(w)+\rho I$, and $\gammaMax = \max_{i\in[n]}\gamma_{i,u}^{\rho}$, where $\gamma_{i,u}^{\rho}$ is the relative upper quadratic regularity constant of $f_i$ with respect to $\nabla^2 f_{i}(w)+\rho I$.
\end{proposition}
The proof of this proposition may be found in \Cref{subsection:PrecondSmoothGradPf}.
\cref{prop:PrecondSmoothGrad} generalizes \eqref{eq:ES_bound} from \cite{gower2019sgd}.
The bounds differ in that \cref{prop:PrecondSmoothGrad} depends upon $\mathcal L_P$, the preconditioned analogue of $\mathcal L$, which we call the preconditioned expected smoothness constant.

In our convergence analysis, $\mathcal L_P$ plays the same role as the smoothness constant in gradient descent. 
\cref{prop:PrecondSmoothGrad} reveals the role of the gradient batch size in determining the expected smoothness constant. 
As the gradient batch size $b_g$ increases from $1$ to $n$, 
$\mathcal L_P$ decreases from $\tau^{\rho}(H(w_P))\gammaMax(1+\zeta)$ to $\gammaRho(1+\zeta)$.
Recall preconditioning helps globally
when $\gammaMax = \bigO(1)$.
In this case, \cref{prop:PrecondSmoothGrad} implies that the batch size $b_g =  \bigO(\tau^{\rho}(H(w_P))$ is needed to ensure $\mathcal L_P = \bigO(1)$.
The dependence upon the $\rho$-dissimilarity is consistent with \cite{frangella2023promise}, which shows a similar dependence when all the $f_i$'s are strongly convex. 
Hence, the $\rho$-dissimilarity plays a key role in determining the Hessian and gradient batch sizes. 

\subsection{Convergence of SketchySGD}
In this section, we present convergence results for \cref{alg:SketchySGD} when $f$ is convex and strongly convex.
We first state hypotheses governing the construction of the preconditioner at each update index $j$.
\begin{assumption}[Preconditioner hyperparameters]
\label{assmp:precond_settings}
Given update frequency $u$, number of stages $s$, number of inner iterations $m$, and $\zeta\in(0,1)$, \cref{alg:SketchySGD} sets hyperparameters as follows:
\begin{enumerate}
    \item The Hessian batchsize is set as 
    \[
    b_{h_j} = \bigO\left(\frac{\tau^{\rho}(H(w_j))\log\left(\frac{d^{\rho}_{\textup{eff}}(H(w_j))}{\delta}\right)}{\zeta^2}\right).
    \]
    \item The randomized Nystr{\"o}m approximation is constructed with rank 
    \[
      r_j = \bigO\left(d_{\textup{eff}}^{\zeta\rho}(H_{S}(w_j))+\log\left(\frac{sm/u}{\delta}\right)\right).
    \]
\end{enumerate}
\end{assumption} 
\cref{assmp:precond_settings}, along with \cref{prop:NysPrecondLem} and a union bound argument, ensures that the preconditioners constructed by \cref{alg:SketchySGD} faithfully approximate the Hessian with high probability throughout all iterations.
We formalize this claim in the following corollary. 
\begin{corollary}[Union bound]
\label{corr:NysLoewnUnionBnd}
Let $\mathcal{E}^{(1)}_{\frac{sm}{u}} = \bigcap_{j=1}^{\frac{sm}{u}}\mathcal{E}^{(1)}_{j}$, and $\mathcal{E}^{(2)}_{\frac{sm}{u}} = \bigcap_{j=1}^{\frac{sm}{u}}\mathcal{E}^{(2)}_{j}$
where 
\[
\mathcal{E}^{(1)}_{j} = \left\{(1-\zeta)\hat{H}_{S_j}^{\rho}\preceq H^{\rho}(w_j) \preceq  (1+\zeta)\hat H^{\rho}_{S_j}\right\}
\]
\[
\mathcal{E}^{(2)}_{j} = \left\{(1-\zeta)\frac{1}{1+\rho/\mu}\hat{H}_{S_j}^{\rho}\preceq H(w_j) \preceq (1+\zeta)\hat H^{\rho}_{S_j}\right\},
\]
Then under \cref{assmp:precond_settings},
\begin{enumerate}
    \item For convex and smooth $f$, $\P\left(\mathcal{E}^{(1)}_\frac{sm}{u}\right)\geq 1-\delta$.
    \item For strongly convex and smooth $f$, $\P\left(\mathcal{E}^{(2)}_\frac{sm}{u}\right)\geq 1-\delta$.
\end{enumerate}
\end{corollary}

\subsubsection{Convergence for convex $f$} 
Our convergence result requires only one stage $s$ when the function $f$ is convex. 
\begin{theorem}[SketchySGD convex convergence]
\label{theorem:sksgd_convex}
    Consider Problem \cref{eq:ERM-Prob} under \cref{assump:DiffandSmoothness}. 
    Run \cref{alg:SketchySGD} for $s = 1$ stage with $m$ inner iterations, using gradient batch size $b_g$, regularization $\rho>0$, learning rate $\eta = \min\left\{\frac{1}{4\mathcal L_P},\sqrt{\frac{\rho \|w_0-w_\star\|_{P_0}^2}{2\sigma^2 m}}\right\}$, update frequency $u$, $\zeta\in (0,1)$, and preconditioner hyperparameters specified in \cref{assmp:precond_settings}.
    Then conditioned on the event $\mathcal{E}^{(1)}_\frac{sm}{u}$ in \cref{corr:NysLoewnUnionBnd}, 
    \begin{align*}
    \E\left[f(\hat w)-f(w_\star)\right]\leq  \frac{8\mathcal L_P\|w_0-w_\star\|_{P_0}^2}{m}+\frac{\sqrt{2\rho}\sigma\|w_0-w_\star\|_{P_0}}{\sqrt{m}}.
    \end{align*}
\end{theorem}
The proof is given in \cref{section:thm_pfs}.
\paragraph{Discussion} \cref{theorem:sksgd_convex} shows that when $f$ is convex, \cref{alg:SketchySGD} equipped with appropriate fixed learning rate converges in expectation to an $\epsilon$-ball around the minimum after $m=\bigO\left(\frac{\mathcal L_P\|w_0-w_\star\|_{P_0}^2}{\epsilon}+\frac{\rho\sigma^2\|w_0-w_\star\|^2_{P_0}}{\epsilon^2}\right)$ iterations.
This convergence rate is consistent with previous results on
stochastic approximation using SGD for smooth convex objectives with bounded gradient variance
\cite{lan2020first}.
However, for SketchySGD, the iteration complexity depends on the preconditioned expected smoothness constant 
and the preconditioned initial distance to the optimum, which may be much smaller than their non-preconditioned counterparts.
Thus, SketchySGD provides faster convergence whenever preconditioning favorably transforms the problem.
We give a concrete example where SketchySGD yield an explicit advantage over SGD in \Cref{subsection:SSGD_Fast} below.

\subsubsection{Convergence for strongly convex $f$}
When $f$ is strongly convex, we can prove a stronger result that relies on the following lemma, 
a preconditioned analogue of the strong convexity lower bound. 
\begin{lemma}[Preconditioned strong convexity bound]
\label{lemma:PrecondStrnConvexity}
Let $P = \hat H_{S}+\rho I$. Assume the conclusion of \Cref{prop:NysPrecondLem} holds: $\hat H_S$ approximates $H$ well. 
Then
\[f(w)-f(w_\star)\geq\frac{\hat \gamma_\ell}{2}\|w-w_\star\|_{P}^2,\]
where $\hat \gamma_\ell = (1-\zeta)\frac{\mu}{\mu+\rho}\gamma_\ell$.
\end{lemma}
The proof of \cref{lemma:PrecondStrnConvexity} is given in \Cref{subsection:PreStrnCvxPf}.
We now state the convergence theorem for SketchySGD when $f$ is strongly convex, which makes use of several stages $s$.
\begin{theorem}[SketchySGD strongly convex convergence]
\label{theorem:SketchySGDConvex}
Instate \cref{assump:DiffandSmoothness} and \cref{assump:StrCvx}.
Run \Cref{alg:SketchySGD} for
\[
ms\geq \frac{32}{(1-\zeta)\gamma_\ell}\left(\mathcal L_P+\frac{2\sigma^2}{\epsilon\rho}\right)\left(1+\rho/\mu\right)\log\left(\frac{2(f(w_0)-f(w_\star))}{\epsilon}\right)~\text{iterations,}
\]
with gradient batchsize $b_g$, learning rate $\eta = \min\{1/4\mathcal L_P,\varepsilon \rho/(8\sigma^2)\}$, regularization $\mu \leq \rho\leq L_{\textup{max}}$, update frequency $u$, and preconditioner hyperparameters specified in \cref{assmp:precond_settings}.
Then conditioned on the event $\mathcal{E}^{(2)}_\frac{sm}{u}$ in \cref{corr:NysLoewnUnionBnd}, \cref{alg:SketchySGD} outputs a point $\hat w^{(s)}$ satisfying  
\[
\E\left[f(\hat w^{(s)})-f(w_\star)\right]\leq \epsilon.
\]
\end{theorem}
The proof is given in \cref{section:thm_pfs}, along with the exact values of $m$ and $s$.
An immediate corollary of \cref{theorem:SketchySGDConvex} is that, supposing the optimal model interpolates the data so $\sigma^2=0$, 
SketchySGD converges linearly to the optimum.
\begin{corollary}[Convergence under interpolation]
\label{corr:SketchySGDInterpol}
    Suppose $\sigma^2 = 0$, and instate the hypotheses of \cref{theorem:SketchySGDConvex}. 
    Run \cref{alg:SketchySGD} for 
    \[
    ms\geq \frac{32}{(1-\zeta)}\frac{\mathcal L_P}{\gamma_\ell}\left(1+\rho/\mu\right)\log\left(\frac{2(f(w_0)-f(w_\star))}{\epsilon}\right)~iterations
    \]
    with learning rate $\eta = \frac{1}{4\mathcal L_P}$.
    Then \cref{alg:SketchySGD} outputs a point $\hat w^{(s)}$ satisfying
    \[
    \E\left[f(\hat w^{(s)})-f(w_\star)\right]\leq \epsilon.
    \]
\end{corollary}
\paragraph{Discussion} 
\cref{theorem:SketchySGDConvex} shows that with an appropriate fixed learning rate, SketchySGD (\cref{alg:SketchySGD}) outputs an $\epsilon$-suboptimal point in expectation after $ms = \tilde{\bigO}\left(\frac{\mathcal L_P}{\gamma_\ell}\frac{\rho}{\mu}+\frac{2\sigma^2/\mu}{\epsilon\gamma_\ell}\right)$ iterations.
For smooth strongly convex $f$, minibatch SGD with a fixed learning rate can produce an $\epsilon$-suboptimal point (in expectation) after $\tilde{\bigO}\left(\frac{\mathcal L}{\mu}+\frac{\sigma^2}{\epsilon \mu^3}\right)$ iterations \footnote{This result follows from \cite{gower2019sgd}, using strong convexity.}. 
However, this comparison is flawed as minibatch SGD does not perform periodic averaging steps.
By \cref{theorem:SketchySGDConvex}, minibatch SGD with periodic averaging 
(a special case of \cref{alg:SketchySGD}, when the preconditioner is always the identity) 
only requires $\bigO\left(\frac{\mathcal L}{\mu}+\frac{\sigma^2}{\epsilon\mu}\right)$ iterations to reach an $\epsilon$-suboptimal point in expectation.
Comparing the two rates, we see SketchySGD has lower iteration complexity when $\mathcal L_P/\gamma_\ell = \bigO(1)$ and $\gamma_\ell = \Omega(1)$.
These relations hold when the objective is quadratic, which we discuss in detail more below (\cref{corr:sksgd-fast}).
Under interpolation, \Cref{corr:SketchySGDInterpol} shows SketchySGD with a fixed learning rate converges linearly to $\epsilon$-suboptimality in at most $\bigO\left(\frac{\mathcal L_P}{\gamma_\ell}\frac{\rho}{\mu}\log\left(\frac{1}{\epsilon}\right)\right)$ iterations.
When $\mathcal L_P/\gamma_\ell$ satisfies  $\mathcal L_P/\gamma_\ell = \bigO(1)$, which corresponds to the setting where Hessian information can help, SketchySGD enjoys a convergence rate of $\bigO\left(\frac{\rho}{\mu}\log\left(\frac{1}{\epsilon}\right)\right)$, faster than the $\bigO\left(\kappa\log\left(\frac{1}{\epsilon}\right)\right)$ rate of gradient descent.
This improves upon prior analyses of stochastic Newton methods under interpolation, which fail to show the benefit of using Hessian information. 
$\rho$-Regularized subsampled-Newton, a special case of SketchySGD, was only shown to converge
in at most $\bigO\left(\frac{\kappa L}{\rho} \log\left(\frac{1}{\epsilon}\right)\right)$ iterations in  \cite[Theorem 1, p.~3]{meng2020fast},
which is worse than the convergence rate of gradient descent by a factor of $\bigO(L/\rho)$.

\subsection{When does SketchySGD improve over SGD?}
\label{subsection:SSGD_Fast}
We now present a concrete setting illustrating when SketchySGD converges faster than SGD. 
Specifically, when the objective is quadratic and strongly convex, SketchySGD enjoys an improved iteration complexity relative to SGD. 
In general, improved global convergence cannot be expected beyond quadratic functions without restricting the function class, as it is well-known in the worst case, that second-order optimization algorithms such as Newton's method do not improve over first-order methods \cite{nemirovskij1983problem,arjevani2019oracle}.
Thus without imposing further assumptions, improved global convergence for quadratic functions is the best that can be hoped for. 
We now give our formal result.
\begin{corollary}[SketchySGD converges fast for quadratic functions]
\label{corr:sksgd-fast}
    Under the hypotheses of \cref{theorem:SketchySGDConvex}, and the assumptions that $f$ is quadratic, $u = \infty$, and $b_g = \tau^{\rho}(H)$, 
    the following holds:
    \begin{enumerate}
        \item If $\sigma^2>0$, after $ms = \bigO\left(\left[\frac{\rho}{\mu}+\frac{\sigma^2}{\epsilon\mu}\right]\log(1/\epsilon)\right)$ iterations, \cref{alg:SketchySGD} outputs a point $\hat{w}^{(s)}$ satisfying
        \[
        \E[f(\hat w^{(s)})-f(w_\star)]\leq \epsilon.
        \]
        \item If $\sigma^2 = 0$, after $ms = \bigO\left(\frac{\rho}{\mu}\log\left(\frac{1}{\epsilon}\right)\right)$ iterations, \cref{alg:SketchySGD} outputs a point $\hat{w}^{(s)}$ satisfying,
        \[
        \E[f(\hat w^{(s)})-f(w_\star)]\leq \epsilon.
        \]
    \end{enumerate}
\end{corollary}
Recall minibatch SGD has iteration complexity of $\bigO\left(\left[\frac{\mathcal L}{\mu}+\frac{\sigma^2}{\epsilon\mu}\right]\log(1/\epsilon)\right)$ when $\sigma^2>0$, and $\bigO\left(\frac{\mathcal L}{\mu}\log(1/\epsilon)\right)$ when $\sigma^2 = 0$.
Thus, \cref{corr:sksgd-fast} shows SketchySGD (\cref{alg:SketchySGD}) roughly enjoys an $\bigO\left(\mathcal L/\rho \right)$ improvment in iteration complexity relative to SGD, provided the gradient batch size satisfies $b_g = \bigO(\tau^\rho(H))$.
We find the prediction that SketchySGD outperforms its non-preconditioned counterpart on ill-conditioned problems, is realized by the practical version (\cref{alg:SketchySGD_pract}) in our experiments.
\begin{remark}[When does better iteration complexity imply fast computational complexity?]
To understand when SketchySGD has lower computational complexity than SGD, assume interpolation $(\sigma^2 = 0)$, and 
$L \asymp L_{\textrm{max}}$. 
Under these hypotheses, the optimal total computational complexity is $\tilde{\bigO}(\kappa p)$, which is achieved with $b_g = 1$.
By comparison, SketchySGD (with $b_g = \tau^\rho(H)$) has total computational complexity $\tilde{\bigO}\left(\frac{\rho}{\mu} \tau^\rho(H) p\right)$. 
Then if $\tau^\rho(H) = \bigO(1/\sqrt{\rho})$, and $\rho = \theta L_{\textrm{max}}$ where $\theta \in (0,1)$, SketchySGD enjoys an improved computational complexity on the order of $\bigO\left(\sqrt{L_{\textrm{max}}/\theta}\right)$.
Hence when $\tau^\rho(H)$ is not too large, SketchySGD also enjoys better computational complexity.
\end{remark}

\subsection{Proofs of \cref{theorem:sksgd_convex} and \cref{theorem:SketchySGDConvex}}
\label{section:thm_pfs}
We now turn to the proofs of \cref{theorem:sksgd_convex} and \cref{theorem:SketchySGDConvex}.
To avoid notational clutter in the proofs, we employ the following notation for the preconditioner at iteration $k$ of stage $s$.
\[P^{(s)}_k \coloneqq \hat{H}_{S_j}+\rho I,
\]
Here $j$ corresponds to the index of the current preconditioner, so that multiple values of $k$ may be mapped to the same index $j$.
For the convex case we omit the the superscript and simply write $P_k$, as $s = 1$.
 With this notational preliminaries out the way, we now commence with the proofs.
\paragraph{Proof of \cref{theorem:sksgd_convex}}
\begin{proof}
Expanding, and taking the expectation conditioned on $k$, we reach
\begin{align*}
&\E_{k}\|w_{k+1}-w_\star\|_{P_{k}}^2 = \|w_{k}-w_\star\|_{P_{k}}^2 -2\eta \langle P_{k}^{-1}g_{k},w_{k}-w_\star\rangle_{P_{k}} 
+\eta^2 \E_{k}\|g_{B_{k}}\|^2_{P_{k}^{-1}} \quad{(*)}.
\end{align*}
Now, by convexity and \cref{prop:PrecondSmoothGrad}, $(*)$ becomes
\[
\E_{k}\|w_{k+1}-w_\star\|_{P_{k}}^2 \leq \|w_{k}-w_\star\|_{P_{k}}^2+2\eta\left(2\eta \mathcal L_P-1\right)(f(w_k)-f(w_\star))+2\sigma^2/\rho. 
\]
Summing the above display from $k = 0,\cdots m-1$ yields
\begin{align*}
    \sum_{k=0}^{m-1}\E_{k}\|w_{k+1}-w_\star\|_{P_{k}}^2 &\leq \sum_{k=0}^{m-1}\|w_{k}-w_\star\|_{P_{k}}^2+2\eta m \left(2\eta \mathcal L_P-1\right) \frac{1}{m}\sum_{k=0}^{m-1}[f(w_k)-f(w_\star)] \\
    & +\frac{2 m \eta^2 \sigma^2}{\rho}.
\end{align*}
 Rearranging, and using convexity of $f$ in conjunction with $\hat w = \frac{1}{m}\sum_{k=0}^{m-1}w_k$, reaches
\begin{align*}
    & \sum_{k=0}^{m-1}\E_{k}\|w_{k+1}-w_\star\|_{P_{k}}^2+2\eta m \left(1-2\eta \mathcal L_P\right)[f(\hat w)-f(w_\star)]\leq \sum_{k=0}^{m-1}\|w_{k}-w_\star\|_{P_{k}}^2+\frac{2 m \eta^2 \sigma^2}{\rho}.
\end{align*}
Taking the total expectation over all iterations, we find
\begin{align*}
    & 2\eta m \left(1-2\eta \mathcal L_P\right)\E\left[f(\hat w)-f(w_\star)\right]\leq \|w_{0}-w_\star\|_{P_{0}}^2+\frac{2 m \eta^2 \sigma^2}{\rho}.
\end{align*}
Consequently, we have
\begin{align*}
    \E\left[f(\hat w)-f(w_\star)\right]\leq \frac{1}{2\eta(1-2\eta \mathcal L_P)m}\|w_0-w_\star\|_{P_0}^2+\frac{\eta \sigma^2}{(1-2\eta \mathcal L_P)\rho}.
\end{align*}
Setting $\eta = \min\left\{\frac{1}{4\mathcal L_P},\sqrt{\frac{\rho \|w_0-w_\star\|_{P_0}^2}{2\sigma^2 m}}\right\}$, we conclude
\begin{align*}
    \E\left[f(\hat w)-f(w_\star)\right]\leq \frac{8\mathcal L_P\|w_0-w_\star\|_{P_0}^2}{m}+\frac{\sqrt{2\rho}\sigma\|w_0-w_\star\|_{P_0}}{\sqrt{m}}.
\end{align*}

\paragraph{Strongly convex case}
Suppose we are in stage $s$ of \cref{alg:SketchySGD}. Following identical logic to the convex case, we reach
\begin{align*}
    &\sum_{k=0}^{m-1}\E_{k}\|w^{(s)}_{k+1}-w_\star\|_{P^{(s)}_k}^2+2\eta m \left(2\eta \mathcal L_P-1\right)\left[f(\hat w^{(s+1)})-f(w_\star)\right]\\
    &\leq \sum_{k=0}^{m-1}\|w_k^{(s)}-w_\star\|_{P^{(s)}_{k}}^2+\frac{2 m \eta^2 \sigma^2}{\rho}.
\end{align*}
Now, taking the total expectation over all inner iterations conditioned on outer iterations $0$ through $s$, yields
\begin{align*}
    & 2\eta m(2\eta \mathcal L_P-1)\left(\E_{0:s}[f(\hat w^{(s+1)})]-f(w_\star)\right)\leq \|\hat w^{(s)}-w_\star\|_{P^{(s)}_0}^2+\frac{2m\eta^2\sigma^2}{\rho}.
\end{align*}
Invoking \cref{lemma:PrecondStrnConvexity}, and rearranging, the preceding display becomes
\begin{align*}
    \E_{0:s}[f(\hat w^{(s+1)})]-f(w_\star)\leq \frac{1}{\hat \gamma_\ell\eta(1-2\eta \mathcal L_P)m}\left(f(\hat w^{(s)})-f(w_\star)\right)+\frac{\eta\sigma^2}{(1-2\eta \mathcal L_P)\rho}.
\end{align*}
Setting $m =  \frac{16\left(\mathcal L_P+2\sigma^2/(\epsilon\rho)\right)}{\hat \gamma_\ell}$ and using
$\eta = \min\{1/4\mathcal L_P,\frac{\varepsilon \rho}{8\sigma^2}\}$, the previous display and a routine computation yield
\[
\E_{0:s}[f(\hat w^{(s+1)})]-f(w_\star)\leq \frac{1}{2}\left(f(\hat w^{(s)})-f(w_\star)\right)+\frac{\varepsilon}{4}.
\]
Taking the total expectation over all outer iterations and recursing, obtains
\[
 \E[f(\hat w^{(s)})]-f(w_\star)\leq \left(\frac{1}{2}\right)^{s}(f(w_0)-f(w_\star))+\frac{\varepsilon}{2}.
\]
Setting $s = 2\log\left(\frac{2(f(w_0)-f(w_\star))}{\epsilon}\right)$, we conclude
\[
\E[f(\hat w^{(s)})]-f(w_\star)\leq \epsilon,
\]
as desired.
\end{proof}
\section{Numerical experiments}
\label{section:Experiments}
In this section, we evaluate the performance of SketchySGD through six sets of experiments. These experiments are presented as follows:

\begin{itemize}
    \item Comparisons to first-order methods (\Cref{subsection:performance_fom}): We compare SketchySGD to SGD, SVRG, minibatch SAGA \cite{gazagnadou2019optimal} (henceforth referred to as SAGA) and loopless Katyusha (L-Katyusha) on ridge regression and $l_2$-regularized logistic regression.
    SketchySGD outperforms the competitor methods, even after they have been tuned.
    \item Comparisons to second-order/quasi-Newton methods (\Cref{subsection:performance_som}): We compare SketchySGD to L-BFGS \cite{liu1989limited}, stochastic LBFGS (SLBFGS) \cite{moritz2016linearly}, randomized subspace Newton (RSN) \cite{gower2019rsn}, and Newton Sketch \cite{pilanci2017newton} on ridge regression and $l_2$-regularized logistic regression.
    SketchySGD either outperforms or performs comparably to these methods. 
    \item Comparisons to preconditioned CG (PCG) (\Cref{subsection:performance_pcg}): We compare SketchySGD to Jacobi PCG \cite{trefethen1997numerical,jambulapati2020fast}, sketch-and-precondition PCG (with Gaussian and sparse embeddings)\cite{avron2010blendenpik,meng2014lsrn,clarkson2017low,martinsson2020randomized}, and Nystr\"{o}m PCG on ridge regression \cite{frangella2021randomized}.\newline
    Again, SketchySGD either outperforms or performs comparably to these methods.
    \item Large-scale logistic regression (\Cref{subsection:large_scale}): We compare SketchySGD to SGD and SAGA on a random features transformation of the HIGGS dataset, where computing full gradients of the objective is computationally prohibitive.
    SketchySGD vastly outperforms the competition in this setting.
    \item Tabular deep learning with multi-layer perceptrons (\Cref{subsection:performance_dl}): We compare SketchySGD to popular first-order (SGD, Adam, Yogi \cite{robbins1951stochastic, kingma2014adam, zaheer2018adaptive}) and second-order methods in deep learning (AdaHessian, Shampoo \cite{yao2021adahessian, gupta2018shampoo, shi2023distributed}).
    SketchySGD outperforms the second-order methods and performs comparably to the first-order methods.
    \item In the supplement, we provide ablation studies for SketchySGD's key hyperparmeters: update frequency (\Cref{subsection:sensitivity_r}), rank parameter (\Cref{subsection:sensitivity_r}), and learning rate (\Cref{subsection:performance_lr_ablation}).
    We also demonstrate how SketchySGD improves problem conditioning in \Cref{section:hessian_precond_appdx}.
\end{itemize}

For convex problems, we run SketchySGD with two different preconditioners: (1) Nystr\"{o}m, which takes a randomized-low rank approximation to the subsampled Hessian and (2) Subsampled Newton (SSN), which uses the subsampled Hessian without approximation. 
Recall, this is a special case of the Nystr\"{o}m preconditioner for which the rank $r_j$ is equal to the Hessian batch size $b_{h_j}$. 
SketchySGD is ran according to the defaults presented in \Cref{section:SketchySGD}, except for in the deep learning experiments (\Cref{subsection:performance_dl}).

We mostly present training loss in the main paper; test loss, training accuracy, and test accuracy appear in \Cref{section:loss_acc_appdx}.
The datasets used in \Cref{subsection:performance_fom,subsection:performance_som,subsection:performance_pcg} are presented in \cref{table:datasets}.
Datasets with ``-rf'' after their names have been transformed using random features \cite{rahimi2007random,mei2022randfeatures}.
The condition number reported in \cref{table:datasets} is a lower bound on the condition number of the corresponding ridge regression/logistic regression problem.\footnote{Details of how this lower bound is computed, are given in \Cref{section:kappa_lwr_bnd}.}
\begin{table}[!htbp]
	\caption{\label{table:datasets} Datasets and summary statistics.} 
	\begin{center}
		\footnotesize
		\begin{tabular}{|c|c|c|c|c|c|c|}
			\hline
			\textbf{Dataset} & $n_{\textrm{tr}}$ & $n_{\textrm{test}}$ & $p$ & nonzeros $\%$ & Condition number & Task \\
			\hline
			E2006-tfidf \cite{kogan2009e2006} & $16087$ & $3308$ & $150360$ & $0.8256$ & $1.051\times 10^{6}$ & Ridge\\
			\hline
			YearPredictionMSD-rf \cite{dua2019uci} & $463715$ & $51630$ & $4367$ & 50.58 & $1.512\times 10^{4}$ & Ridge \\
			\hline
			yolanda-rf \cite{guyon2019yolanda} & $320000$ & $80000$ & $1000$ & $100$ & $1.224\times 10^{4}$ & Ridge \\
			\hline
			ijcnn1-rf \cite{prokhorov2001ijcnn} & $49990$ & $91701$ & $2500$ & $100$ & $3.521\times 10^{5}$ & Logistic \\
			\hline
            real-sim \cite{libsvmrealsim} & $57847$ & $14462$ & $20958$ & 0.2465 & $1.785\times 10^{1}$ & Logistic \\
		    \hline
			susy-rf \cite{baldi2014higgs} & $4500000$ & $500000$ & $1000$ & 100 & $4.626\times 10^{8}$ & Logistic \\
			\hline
		\end{tabular}
	\end{center}
\end{table}


Each method is run for $40$ full gradient evaluations (except \Cref{subsection:large_scale,subsection:performance_dl}, where we use $10$ and $105$, respectively).
Note for SVRG, L-Katyusha, and SLBFGS, this corresponds to 20 epochs, as at the end of each epoch SVRG, they compute full a gradients to perform variance reduction.
All methods use a gradient batch size of $b_g = 256$.

We plot the distribution of results for each dataset and optimizer combination over several random seeds to reduce variability in the outcomes; the solid/dashed lines show the median and shaded regions represent the 10--90th quantile. The figures we show are plotted with respect to both wall-clock time and full gradient evaluations.
We truncate plots with respect to wall-clock time at the time when the second-fastest optimizer terminates.
We place markers at every 10 full data passes for curves corresponding to SketchySGD, allowing us to compare the time efficiency of using the Nystr\"{o}m and SSN preconditioners in our method.

Additional details appear in \Cref{section:exp_details} and code to reproduce our experiments may be found at the git repo \url{https://github.com/udellgroup/SketchySGD/tree/main/simods}. 

\subsection{SketchySGD outperforms first-order methods}
\label{subsection:performance_fom}
In this section, we compare \newline SketchySGD to the first-order methods SGD, SVRG, SAGA, and L-Katyusha. 
In \Cref{subsubsection:performance_auto}, we use the default values for the learning rate/smoothness hyperparameters, based on recommendations made in \cite{johnson2013accelerating,defazio2014saga,kovalev2020lkatyusha} and scikit-learn \cite{pedregosa2011scikit}\footnote{SGD does not have a default learning rate, therefore, we exclude this method from this comparison.} (see \Cref{section:exp_details} for more details).
In \Cref{subsubsection:performance_tuned}, we tune the learning rate/smoothness hyperparameter via grid search.
Across both settings, SketchySGD outperforms the competition.

\subsubsection{First-order methods --- defaults}
\label{subsubsection:performance_auto}
\cref{fig:performance_auto_logistic,fig:performance_auto_least_squares} compare SketchySGD to first-order methods run with their defaults. 
SketchySGD (Nystr{\"o}m) and SketchySGD (SSN) uniformly outperform their first-order counterparts, sometimes dramatically. 
In the case of the E2006-tfidf and ijcnn1-rf datasets, SVRG, SAGA, and L-Katyusha make no progress at all. 
Even for datasets where SVRG, SAGA, and L-Katyusha do make progress, their performance lags significantly behind SketchySGD (Nystr{\"o}m) and SketchySGD (SSN).
Second-order information speeds up SketchySGD without significant computational costs: the plots show both variants of SketchySGD converge faster than their first-order counterparts.

The plots also show that SketchySGD (Nystr{\"o}m) and SketchySGD (SSN) exhibit similar performance, despite SketchySGD (Nystr{\"o}m) using much less information than SketchySGD (SSN).
For the YearPredictionMSD-rf and yolanda-rf datasets, SketchySGD (Nystr{\"o}m) performs better than SketchySGD (SSN).
We expect this gap to become more pronounced as $n$ grows, for SketchySGD (SSN) requires $O(\sqrt{n}p)$ flops to apply the preconditioner, while SketchySGD (Nystr{\"o}m) needs only $O(rp)$ flops.
This hypothesis is validated in \Cref{subsection:large_scale}, where we perform experiments on a large-scale version of the HIGGS dataset.

\begin{figure}[t]
    \centering
    \includegraphics[scale=0.5]{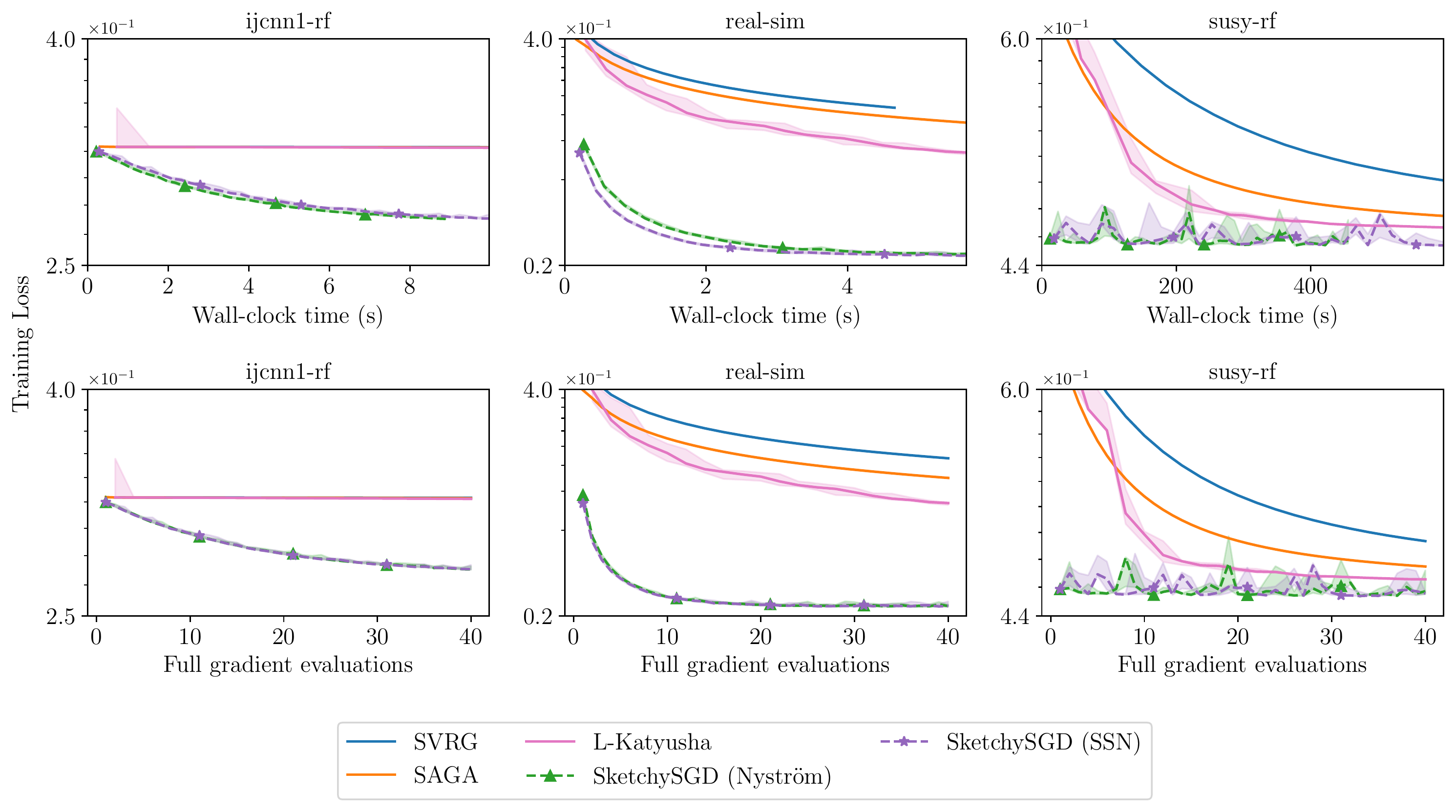}
    \caption{Comparisons to first-order methods with default learning rates (SVRG, SAGA) and smoothness parameters (L-Katyusha) on $l_2$-regularized logistic regression.}
    \label{fig:performance_auto_logistic}
\end{figure}

\begin{figure}[t]
    \centering
    \includegraphics[scale=0.5]{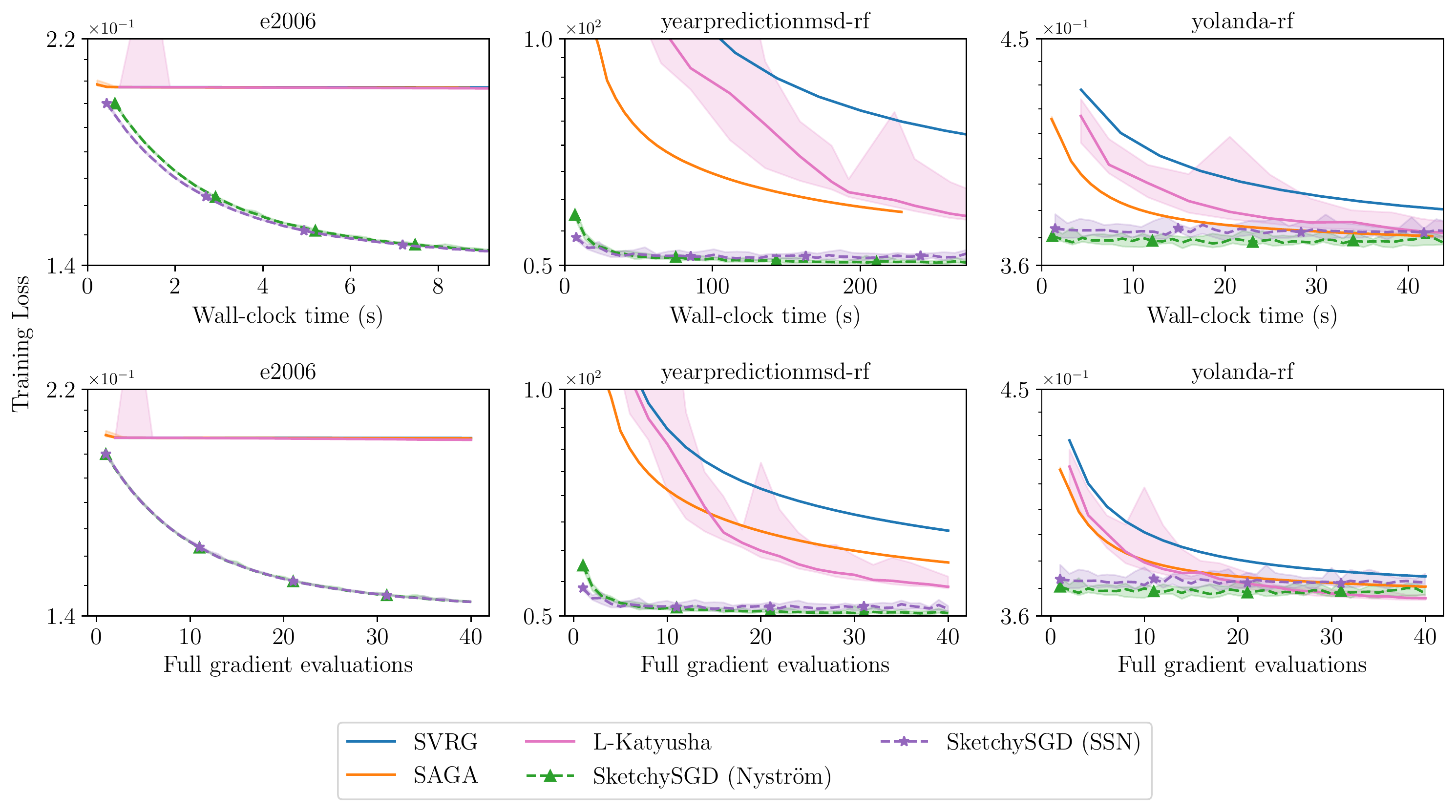}
    \caption{Comparisons to first-order methods with default learning rates (SVRG, SAGA) and smoothness parameters (L-Katyusha) on ridge regression.}
    \label{fig:performance_auto_least_squares}
\end{figure}


\subsubsection{First-order methods --- tuned}
\label{subsubsection:performance_tuned}
\cref{fig:performance_tuned_logistic,fig:performance_tuned_least_squares} show that SketchySGD (Nystr{\"o}m) and SketchySGD (SSN) generally match or outperform tuned first-order methods.
For the tuned first-order methods, we only show the curve corresponding to the lowest attained training loss.
 SketchySGD (Nystr{\"o}m) and SketchySGD (SSN) outperform the competitor methods on E2006-tfidf and YearPredictionMSD-rf, while performing comparably on both yolanda-rf and susy-rf. 
On real-sim, we find SGD and SAGA perform better than SketchySGD (Nystr{\"o}m) and SketchySGD (SSN) on wall-clock time, but perform similarly on gradient evaluations. 

\begin{figure}[htbp]
    \centering
    \includegraphics[scale=0.5]{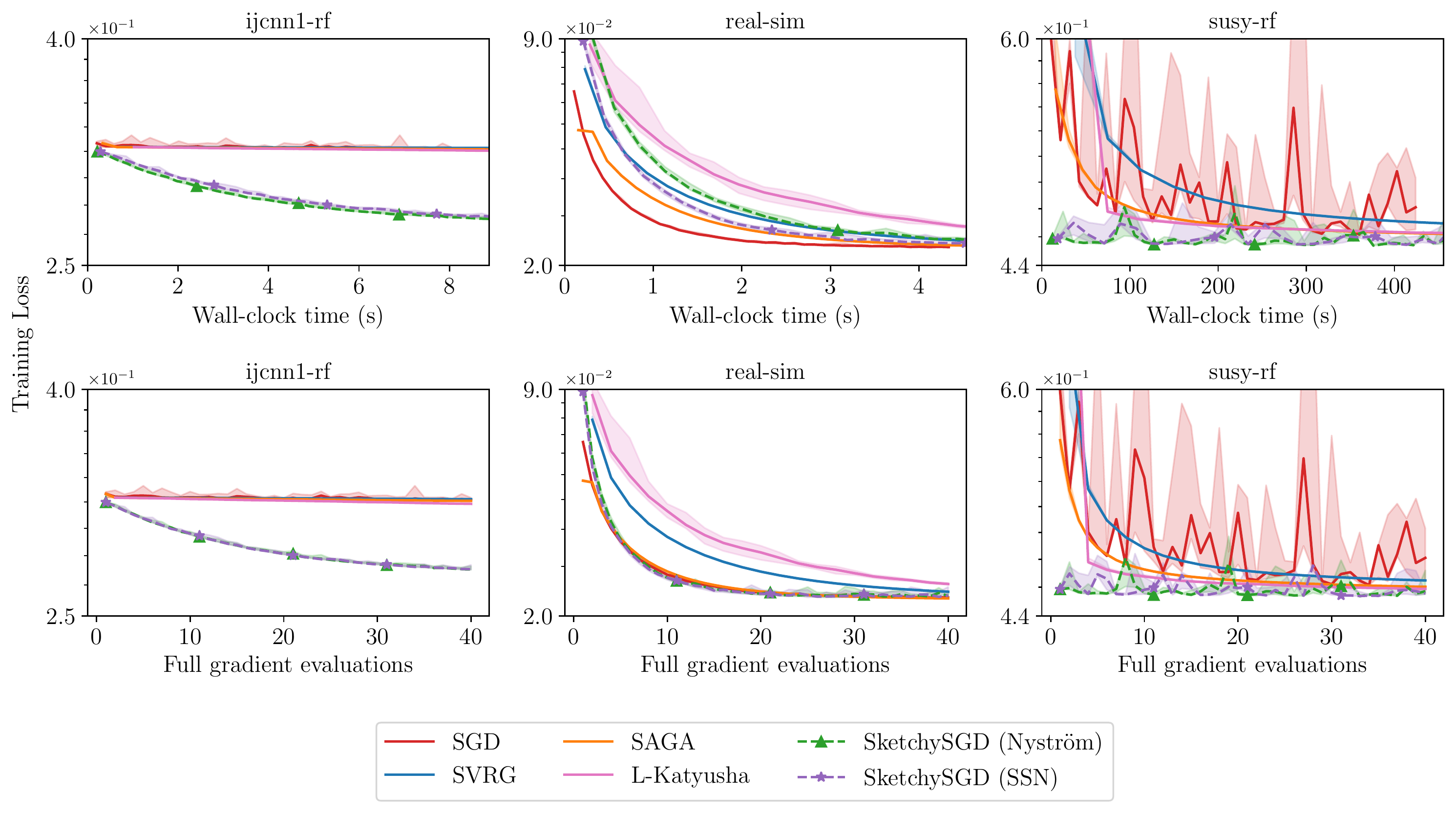}
    \caption{Comparisons to first-order methods with tuned learning rates (SGD, SVRG, SAGA) and smoothness parameters (L-Katyusha) on $l_2$-regularized logistic regression.}
    \label{fig:performance_tuned_logistic}
\end{figure}

\begin{figure}[t]
    \centering
    \includegraphics[scale=0.5]{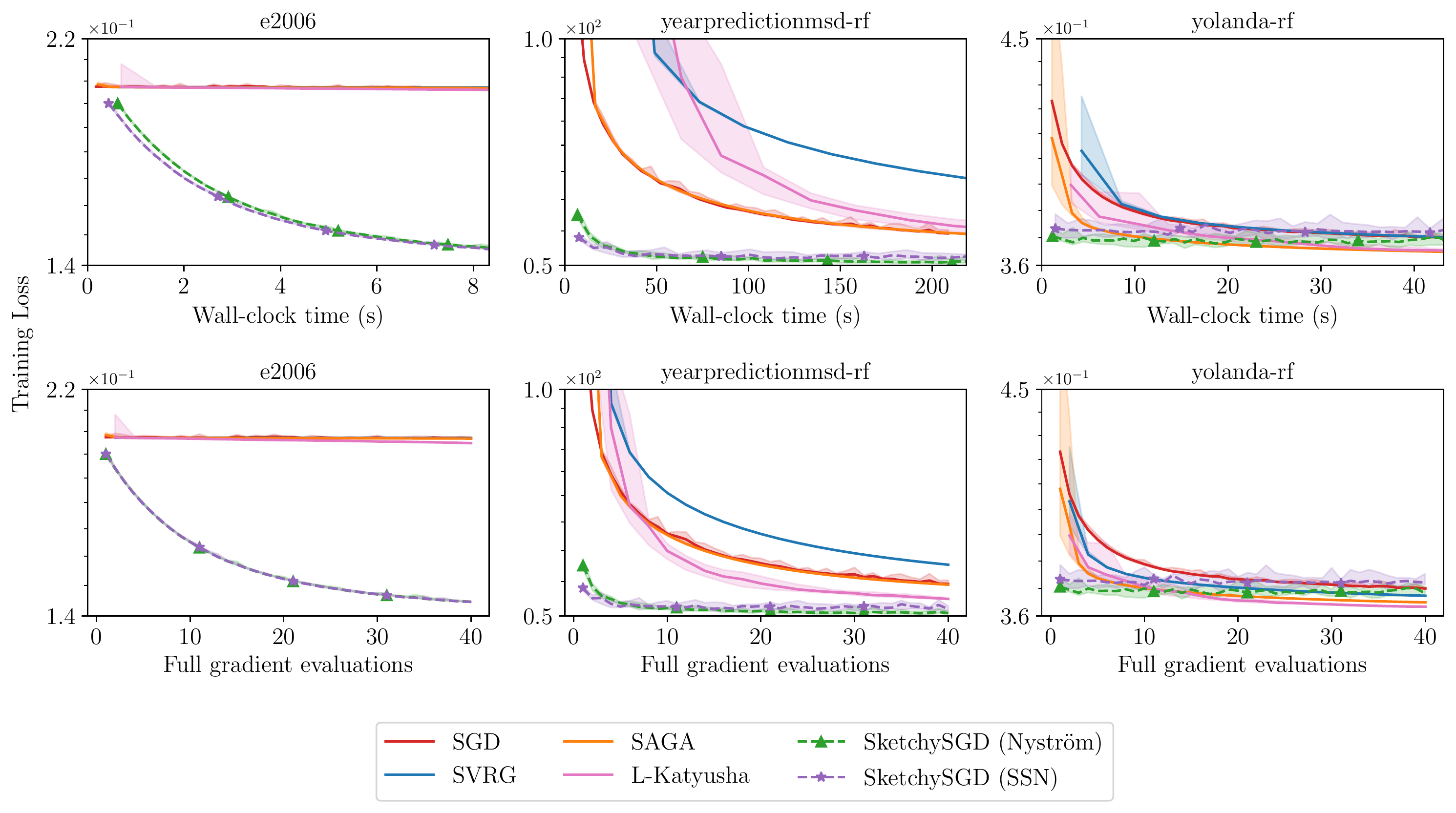}
    \caption{Comparisons to first-order methods with tuned learning rates (SGD, SVRG, SAGA) and smoothness parameters (L-Katyusha) on ridge regression.}
    \label{fig:performance_tuned_least_squares}
\end{figure}


\subsection{SketchySGD (usually) outperforms quasi-Newton methods}
\label{subsection:performance_som}
We compare \\ 
SketchySGD to the second-order methods L-BFGS (using the implementation in SciPy), SLBFGS, RSN, and Newton Sketch. 
For SLBFGS, we tune the learning rate; we do not tune learning rates for L-BFGS, RSN, or Newton Sketch since these methods use line search.
The results for logistic and ridge regression are presented in \cref{fig:performance_som_logistic,fig:performance_som_least_squares}, respectively.
In several plots, SLBFGS cuts off early because it tends to diverge at the best learning rate obtained by tuning.
L-BFGS also terminates early on ijcnn1-rf and real-sim because it reaches a high-accuracy solution in under $40$ iterations.
We are generous to methods that use line search --- we do not account for the number of function evaluations performed by L-BFGS, RSN, and Newton Sketch, nor do we account for the number of additional full gradient evaluations performed by L-BFGS to satisfy the strong Wolfe conditions.

Out of all the methods, SketchySGD provides the most consistent performance.
When considering wall-clock time performance, SketchySGD is only outperformed by SLBFGS (although it eventually diverges) and Newton Sketch on ijcnn1-rf, L-BFGS on real-sim, and RSN on yolanda-rf.
On larger, dense datasets, such as YearPredictionMSD-rf and susy-rf, SketchySGD is the clear winner.
We expect the performance gap between SketchySGD and the quasi-Newton methods to grow as the datasets become larger, and we show this is the case in \Cref{subsection:scaling_qn}.


\begin{figure}[t]
    \centering
    \includegraphics[scale=0.5]{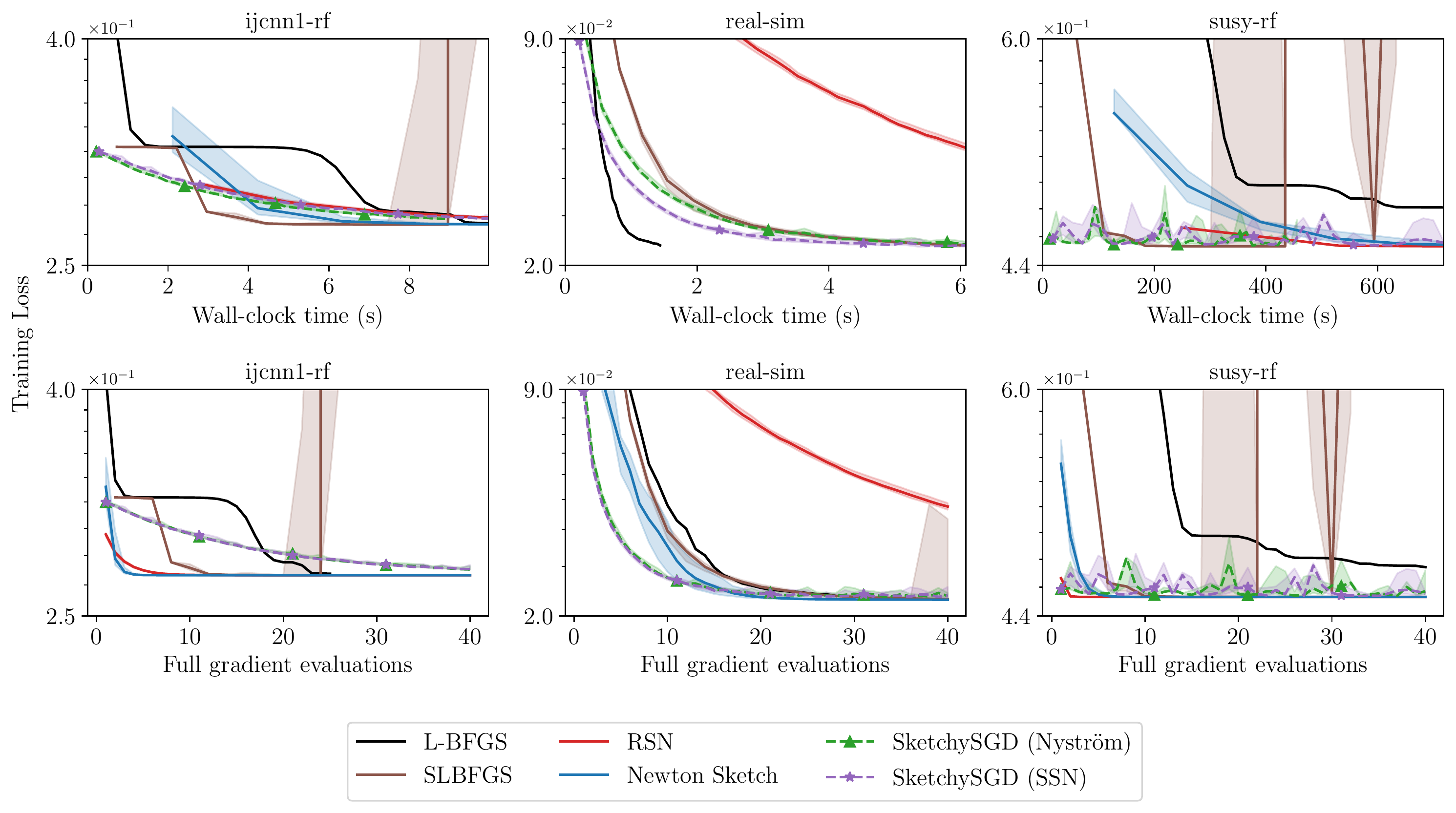}
    \caption{Comparisons to quasi-Newton methods (L-BFGS, SLBFGS, RSN, Newton Sketch) on $l_2$-regularized logistic regression.}
    \label{fig:performance_som_logistic}
\end{figure}

\begin{figure}[t]
    \centering
    \includegraphics[scale=0.5]{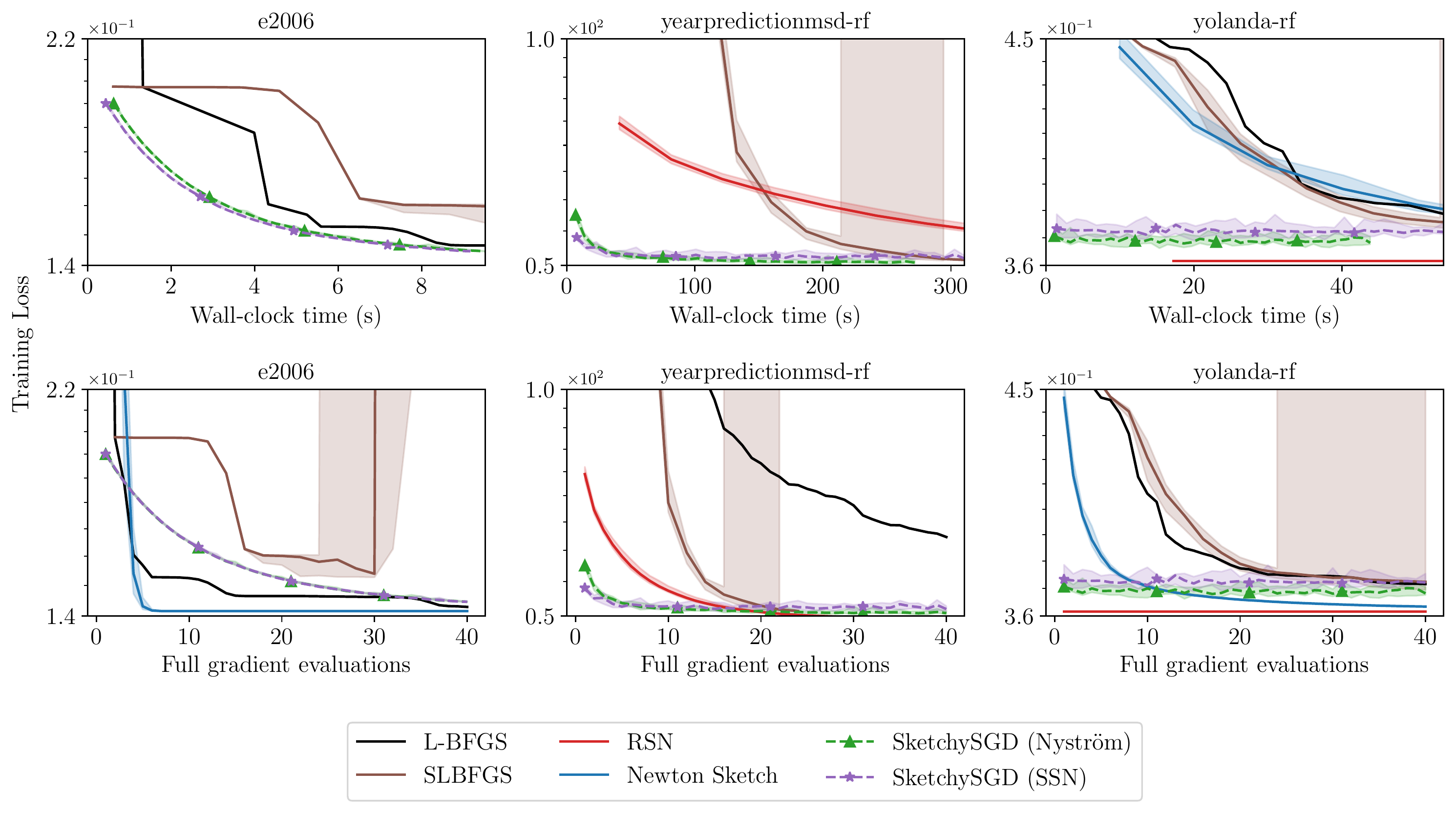}
    \caption{Comparisons to quasi-Newton methods (L-BFGS, SLBFGS, RSN, Newton Sketch) on ridge regression.}
    \label{fig:performance_som_least_squares}
\end{figure}

\subsection{SketchySGD (usually) outperforms PCG}
\label{subsection:performance_pcg}
We compare SketchySGD to PCG with Jacobi, Nystr\"{o}m, and sketch-and-precondition (Gaussian and sparse embeddings) preconditioners.
The results for ridge regression are presented in \Cref{fig:performance_pcg_least_squares}.
For PCG, full gradient evaluations refer to the total number of iterations.

Similar to the results in \Cref{subsection:performance_som}, SketchySGD provides the most consistent performance.
On wall-clock time, SketchySGD is eventually outperformed by JacobiPCG on E2006-tfidf and Nystr\"{o}mPCG on yolanda-rf.
However, SketchySGD (Nystr\"{o}m) already reaches a reasonably low training loss within 10 seconds and 2 seconds on YearPredictionMSD-rf and yolanda-rf, respectively, while the PCG methods take much longer to reach this level of accuracy.
Furthermore, SketchySGD outperforms the sketch-and-precondition methods on all datasets.
We expect the performance gap between SketchySGD and PCG to grow as the datasets become larger, and we show this is the case in \Cref{subsection:scaling_pcg}.

\begin{figure}[t]
    \centering
    \includegraphics[scale=0.5]{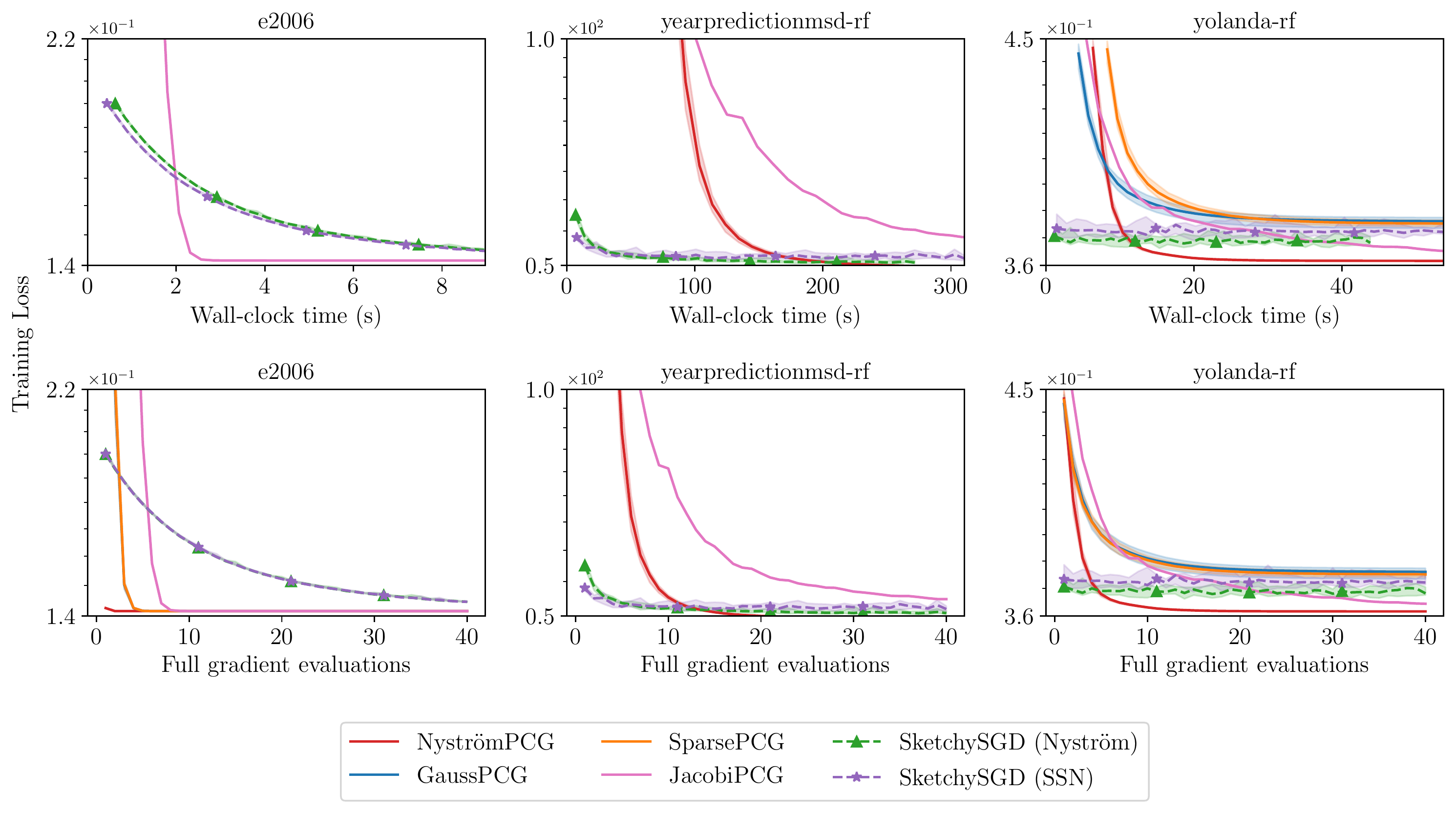}
    \caption{Comparisons to PCG (Jacobi, Nystr\"{o}m, sketch-and-precondition w/ Gaussian and sparse embeddings) on ridge regression.}
    \label{fig:performance_pcg_least_squares}
\end{figure}
\clearpage

\subsection{SketchySGD outperforms competitor methods on large-scale data}
\label{subsection:large_scale}
We apply random Fourier features 
to the HIGGS dataset, for which $(n_{\mathrm{tr}}, p) = (1.05 \cdot 10^7, 28)$, 
to obtain a transformed dataset with size $(n_{\mathrm{tr}}, p) = (1.05 \cdot 10^7, 10^4)$. 
This transformed dataset is $840$ GB, 
larger than the hard drive and RAM capacity of most computers. 
To optimize, we load the original HIGGS dataset in memory 
and at each iteration, form a minibatch of the transformed dataset
by applying the random features transformation to a minibatch of HIGGS.
In this setting, computing a full gradient of the objective is 
computationally prohibitive. 
We exclude SVRG, L-Katyusha, and SLBFGS since they require full gradients. 

We compare SketchySGD to SGD and SAGA with both default learning rates (SAGA only) and tuned learning rates (SGD and SAGA) via grid search. 

\cref{fig:higgs_auto,fig:higgs_tuned} show these two sets of comparisons\footnote{The wall-clock time in \cref{fig:higgs_tuned} cuts off earlier than in \cref{fig:higgs_auto} due to the addition of SGD, which completes $10$ epochs faster than the other methods.}.
We only plot test loss, as computing the training loss is as expensive as computing a full gradient.
The wall-clock time plots only show the time taken in optimization; they do not include the time taken in repeatedly applying the random features transformation. 
We find SGD and SAGA make little to no progress in decreasing the test loss, even after tuning. However, both SketchySGD (Nystr\"{o}m) and SketchySGD (SSN) are able to decrease the test loss significantly. 
Furthermore, SketchySGD (Nystr\"{o}m) is able to achieve a similar test loss to SketchySGD (SSN) while taking less time, confirming that SketchySGD (Nystr\"{o}m) can be more efficient than SketchySGD (SSN) in solving large problems.

\begin{figure}[t]
    \centering
    \includegraphics[scale=0.4]{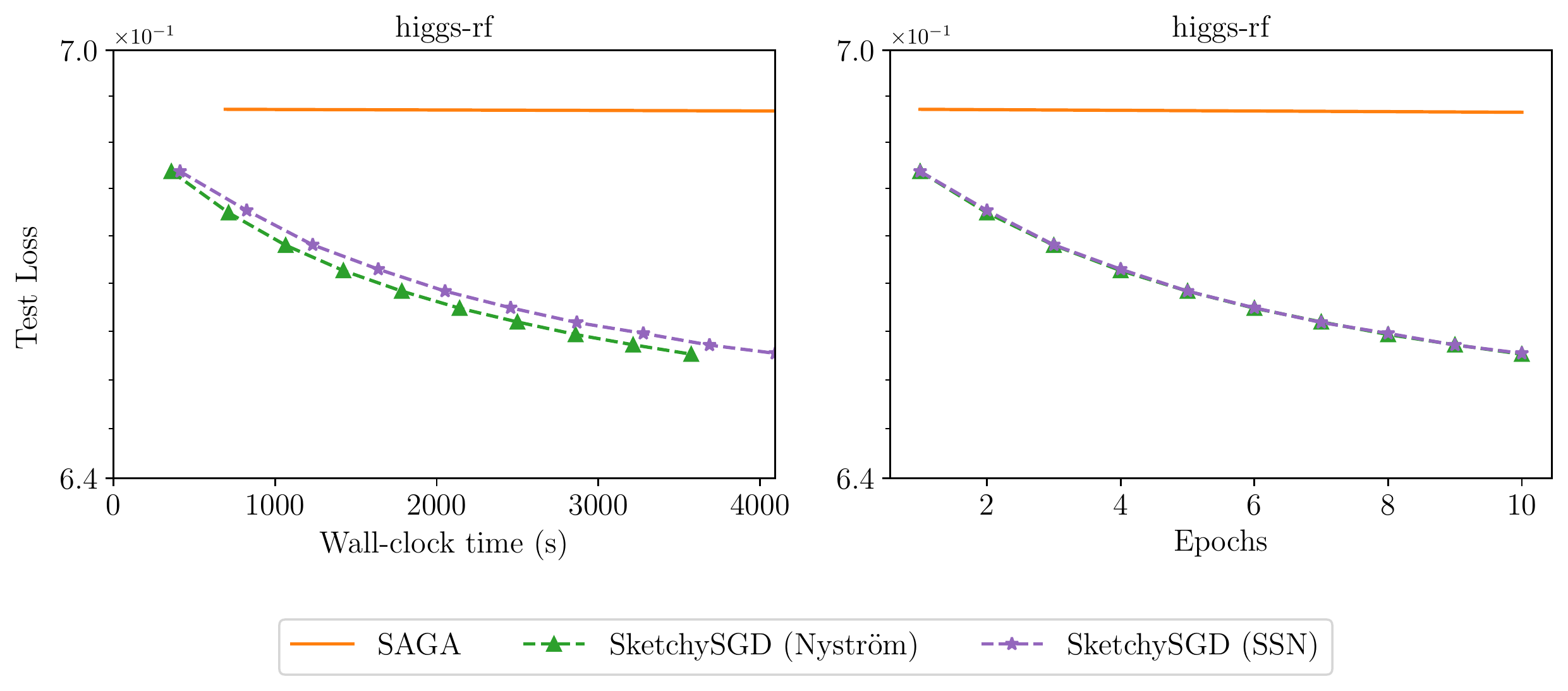}
    \caption{Comparison between SketchySGD and SAGA with default learning rate.}
    \label{fig:higgs_auto}
\end{figure}

\begin{figure}[t]
    \centering
    \includegraphics[scale=0.4]{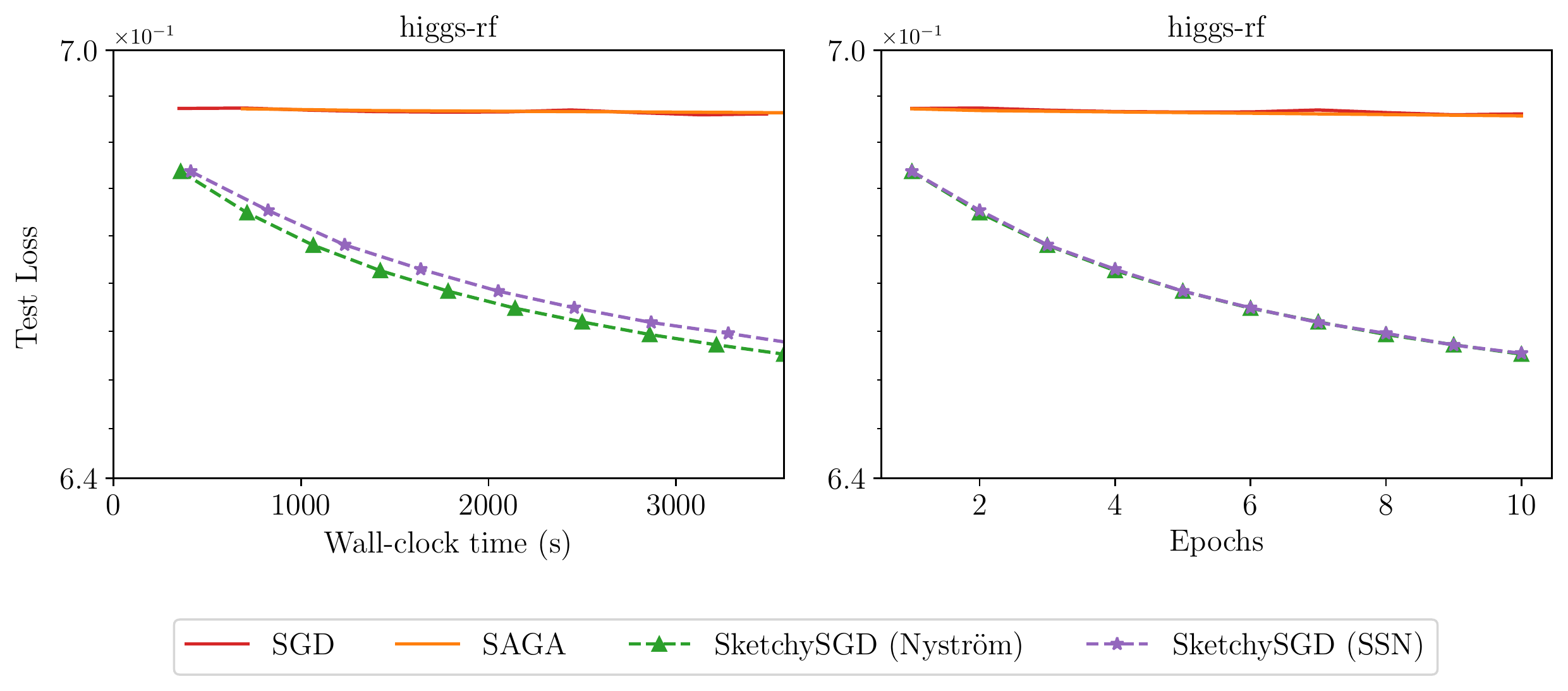}
    \caption{Comparison between SketchySGD, SGD and SAGA with tuned learning rates.}
    \label{fig:higgs_tuned}
\end{figure}

\subsection{Tabular deep learning with multilayer perceptrons}
\label{subsection:performance_dl}
Our experiments closely follow the setting in \cite{kadra2021welltuned}.
We compare SketchySGD (Nystr\"{o}m) to SGD, Adam, AdaHessian, Yogi, and Shampoo, which are popular optimizers for deep learning. 
For experiments in this setting, we modify SketchySGD (Nystr\"{o}m) to use momentum and gradient debiasing, as in Adam (\cref{section:mod_dl_appdx}); we also set $\rho = 10^{-1}$ since it provides better performance.
We use a 9-layer MLP with 512 units in each layer, and cosine annealing for learning rate scheduling.
We run the methods on the Fashion-MNIST, Devnagari-Script, and volkert datasets from OpenML \cite{vanschoren2013openml}.
Throughout, we use the weighted cross-entropy loss and balanced accuracy as evaluation metrics.
We only tune the initial learning rate for the methods and do so via random search.
We form a 60/20/20 training/validation/test split of the data, and select the learning rate with the highest balanced validation accuracy to generate the results reported in this section.

Test accuracy curves for each optimizer are presented in \Cref{fig:deep_learning_test_acc}, and final test accuracies with quantiles are given in \Cref{table:deep_learning_test_acc}.
We see SketchySGD consistently outperforms Shampoo, which is an optimizer designed to approximate full-matrix AdaGrad \cite{duchi2011adaptive}.
Furthermore, SketchySGD tends to be more stable than SGD after tuning.
However, it is unclear whether SketchySGD performs better than SGD, Adam, or Yogi, which are all first-order optimizers.
The reasons for this lack of improvement are unclear, providing an interesting direction for future work.

\begin{figure}[t]
    \centering
    \includegraphics[scale=0.5]{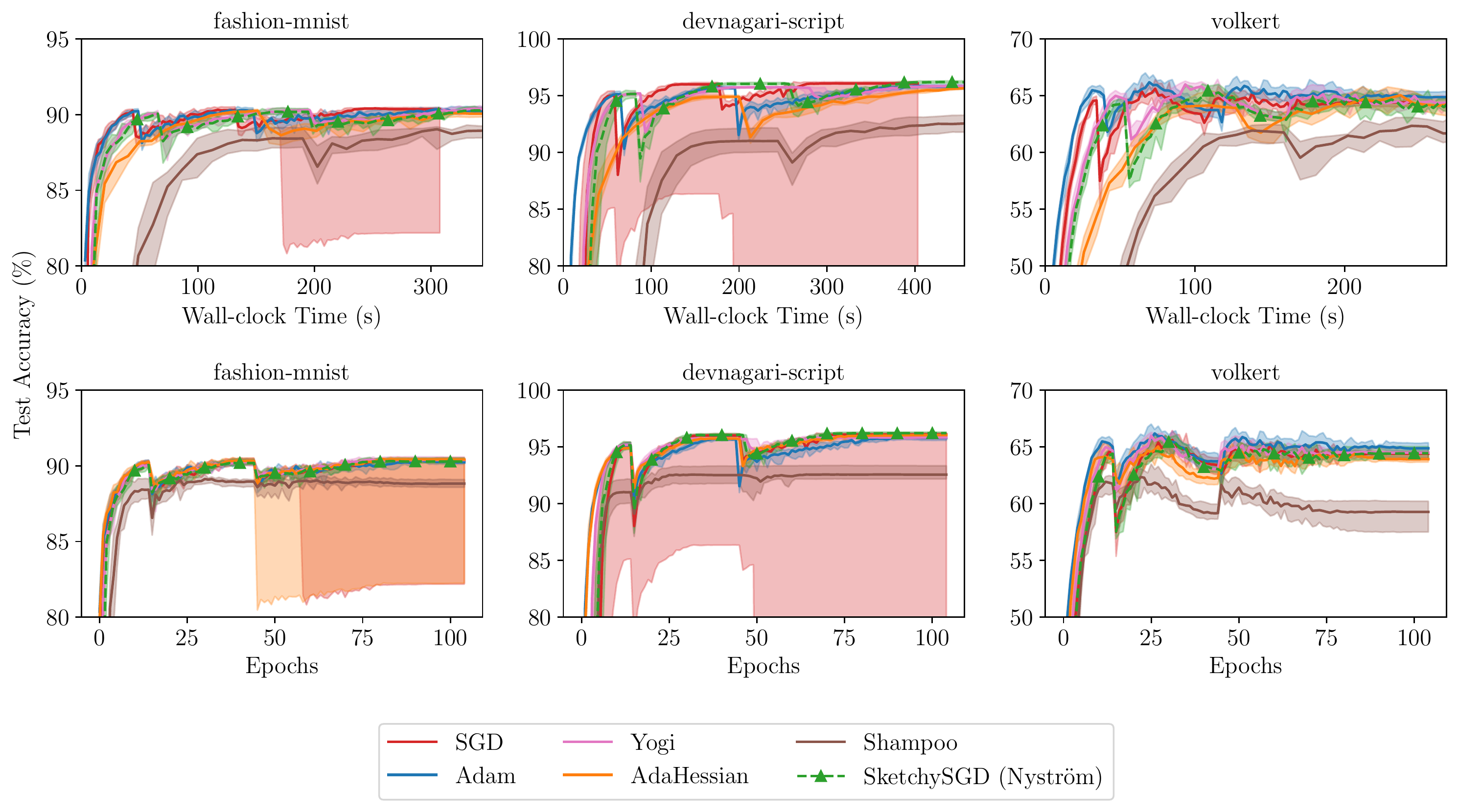}
    \caption{Test accuracies for SketchySGD and competitor methods on tabular deep learning tasks.}
    \label{fig:deep_learning_test_acc}
\end{figure}

\begin{table}[!htbp]
	\caption{\label{table:deep_learning_test_acc} 10th and 90th quantiles for final test accuracies.} 
	\begin{center}
		\scriptsize
        \begin{tabular}{C{1.5cm}C{1.5cm}C{2cm}C{1.5cm}C{2cm}C{1.5cm}C{2cm}C{1.25cm}}
            \hline
            Dataset\textbackslash \newline Optimizer & SGD & Adam & Yogi & AdaHessian & Shampoo & SketchySGD \\
            \hline
            Fashion-MNIST & (82.19, 90.47) & (90.10, 90.51) & (90.26, 90.58) & (82.25, 90.53) & (88.59, 89.10) & (90.19, 90.50) \\
            \hline
            Devnagari-Script & (2.17, 96.23) & (95.57, 95.98) & (95.67, 95.91) & (95.90, 96.17) & (92.18, 93.34) & (96.11, 96.30) \\
            \hline
            volkert & (64.03, 64.58) & (64.61, 65.36) & (64.08, 64.84) & (63.66, 64.52) & (57.50, 60.22) & (63.90, 65.01) \\
            \hline
        \end{tabular}
	\end{center}
\end{table}
\section{Conclusion}
\label{section:conclusion}
In this paper, we have presented SketchySGD, a fast and robust stochastic quasi-Newton method for convex machine learning problems.
SketchySGD uses subsampling and randomized low-rank approximation to improve conditioning by approximating the curvature of the loss.
Furthermore, SketchySGD uses a novel automated learning rate and comes with default hyperparameters that enable it to work out of the box without tuning.

SketchySGD has strong benefits both in theory and in practice.
For quadratic objectives, our theory shows SketchySGD converges to $\epsilon$-suboptimality at a faster rate than SGD, and our experiments validate this improvement in practice. 
SketchySGD with its default hyperparameters outperforms or matches the performance of SGD, SAGA, SVRG, SLBFGS, and L-Katyusha (the last four of which use variance reduction), even when optimizing the learning rate for the competing methods using grid search.

\section*{Acknowledgments}
We would like to acknowledge helpful discussions with John Duchi, Michael Mahoney, Mert Pilanci, and Aaron Sidford.

\bibliographystyle{siamplain}
\bibliography{references}

\appendix
\section{RandNysApprox}
\label{section:RandNysAppx}
In this section we provide the pseudocode for the \texttt{RandNysApprox} algorithm mentioned in \Cref{section:SketchySGD}, which SketchySGD uses to construct the low-rank approximation $\hat{H}_{S_j}$.
\begin{algorithm}[] 
   \caption{\texttt{RandNysApprox}}
   \label{alg:RandNysAppx}
    \begin{algorithmic}
       \STATE {\bfseries Input:} sketch $Y\in \R^{p\times r_j}$ of $H_{S_j}$, orthogonalized test matrix $Q\in \R^{p\times r_j}$,
       \STATE $\nu = \sqrt{p} \text{eps}(\text{norm}(Y, 2))$ \hfill \COMMENT{Compute shift}
       \STATE $Y_{\nu} = Y + \nu Q$ \hfill \COMMENT{Add shift for stability}
       \STATE $C = \text{chol}(Q^TY_\nu)$ \hfill \COMMENT{Cholesky decomposition: $C^{T}C = Q^{T}Y_\nu$}
       \STATE $B = YC^{-1}$ \hfill \COMMENT{Triangular solve}
       \STATE $[\hat V, \Sigma, \sim] = \text{svd}(B, 0)$ \hfill \COMMENT{Thin SVD}
       \STATE $\hat \Lambda = \text{max}\{0, \Sigma^2 - \nu I\}$ \hfill \COMMENT{Compute eigs, and remove shift with element-wise max}
       \STATE {\bfseries Return:} $\hat V, \hat \Lambda$
    \end{algorithmic}
\end{algorithm} 

\cref{alg:RandNysAppx} follows Algorithm 3 from \cite{tropp2017fixed}. 
$\text{eps}(x)$ is defined as the positive distance between $x$ and the next largest floating point number of the same precision as $x$. 
The test matrix $Q$ is the same test matrix used to generate the sketch $Y$ of $H_{S_k}$. 
The resulting Nystr\"{o}m approximation $\hat{H}_{S_k}$ is given by $\hat V \hat \Lambda \hat V^T$.
The resulting Nystr{\"o}m approximation is psd but may have eigenvalues that are equal to $0$.
In our algorithms, this approximation is always used in conjunction with a regularizer to ensure positive definiteness.

\section{Modifications for deep learning}
\label{section:mod_dl_appdx}
We make modifications to the randomized Nystr\"{o}m approximation and SketchySGD for deep learning.
\cref{alg:RandNysAppxMod} adapts \cref{alg:RandNysAppx} to the non-convex (e.g., deep learning) setting, and ensures that the Randomized Nystr\"{o}m approximation remains positive definite.
The main difference between \cref{alg:RandNysAppxMod} and \cref{alg:RandNysAppx} is that \cref{alg:RandNysAppxMod} comes with a fail-safe step (colored red in the algorithm block) in case the subsampled Hessian is indefinite and the resulting Cholesky decomposition fails.
When this failure occurs, the fail-safe step shifts the spectrum of $C$ by its smallest eigenvalue to ensure it is positive definite.
When there is no failure, it is easy to see that \cref{alg:RandNysAppxMod} gives the exact same output as \cref{alg:RandNysAppx}.

\begin{algorithm}[] 
   \caption{\texttt{RandNysApproxMod}}
   \label{alg:RandNysAppxMod}
    \begin{algorithmic}
       \STATE {\bfseries Input:} sketch $Y\in \R^{p\times r_j}$ of $H_{S_j}$, orthogonalized test matrix $Q\in \R^{p\times r_j}$,
       \STATE $\nu = \sqrt{p} \text{eps}(\text{norm}(Y, 2))$ \hfill \COMMENT{Compute shift}
       \STATE $Y_{\nu} = Y + \nu Q$ \hfill \COMMENT{Add shift for stability}
       \STATE $\lambda = 0$ \hfill \COMMENT{Additional shift may be required for positive definiteness}
       \STATE $C = \text{chol}(Q^TY_\nu)$ \hfill \COMMENT{Cholesky decomposition: $C^{T}C = Q^{T}Y_\nu$}
       \textcolor{red}{
       \IF {chol fails}
        \STATE Compute $[W, \Gamma] = \mathrm{eig}(Q^T Y_\nu)$ \hfill \COMMENT{$Q^T Y_\nu$ is small and square}
        \STATE Set $\lambda = \lambda_{\min}(Q^T Y_\nu)$
        \STATE $R = W(\Gamma + |\lambda| I)^{-1/2} W^T$
        \STATE $B = YR$ \hfill \COMMENT{$R$ is psd}
       \ELSE
        \STATE $B = YC^{-1}$ \hfill \COMMENT{Triangular solve}
       \ENDIF
       }
       \STATE $[\hat V, \Sigma, \sim] = \text{svd}(B, 0)$ \hfill \COMMENT{Thin SVD}
       \STATE $\hat \Lambda = \text{max}\{0, \Sigma^2 - (\nu + |\lambda| I)\}$ \hfill \COMMENT{Compute eigs, and remove shift with element-wise max}
       \STATE {\bfseries Return:} $\hat V, \hat \Lambda$
    \end{algorithmic}
\end{algorithm} 

\cref{alg:SketchySGD_dl} adds momentum (which is a hyperparameter $\beta$) and gradient debiasing (similar to Adam) to the practical version of SketchySGD \cref{alg:SketchySGD_pract}.
These changes to the algorithm are shown in red.
In addition, \cref{alg:SketchySGD_dl} does not use the automated learning rate; the learning rate is now an input to the algorithm.

\begin{algorithm}[tb]
   \caption{SketchySGD (Deep learning version)}
   \label{alg:SketchySGD_dl}
    \begin{algorithmic}
       \STATE {\bfseries Input:} initialization $w_0$, learning rate $\eta$, momentum parameter $\beta$, hvp oracle $\mathcal{O}_{H}$, ranks $\{r_j\}$, regularization $\rho$, preconditioner update frequency $u$, stochastic gradient batch size $b_g$, stochastic Hessian batch sizes $\{b_{h_j}\}$
        \STATE Initialize $z_{-1} = 0$ \hfill \COMMENT{Initialize momentum vector}
        \FOR {$k = 0,1,\dots, m-1$}
                \STATE Sample a batch $B_k$
                \STATE Compute stochastic gradient $g_{B_{k}}(w_k)$
                \IF{$k \equiv 0 \pmod u$}
                    \STATE Set $j = j+1$ 
                    \STATE Sample a batch $S_j$ \hfill \COMMENT{$|S_j| = b_{h_j}$}
                    \STATE $\Phi = \text{randn}(p, r_j)$ \hfill \COMMENT{Gaussian test matrix}
                    \STATE $Q = $ \texttt{qr\_econ} $(\Phi)$
                    \STATE Compute sketch $ Y = H_{S_j}(w_k)Q$ \hfill \COMMENT{$r$ calls to $\mathcal{O}_{H_{S_j}}$}
                    \STATE $[\hat{V},\hat{\Lambda}] = $ \texttt{RandNysApproxMod}$(Y,Q,r_j)$
                \ENDIF
            \textcolor{red}{
            \STATE $z_k = \beta z_{k - 1} + (1 - \beta) g_{B_k}(w_k)$ \hfill \COMMENT{Update biased first moment estimate}}
            \textcolor{red}{
            \STATE $\hat{z}_k = z_k / (1 - \beta^{k + 1})$ \hfill \COMMENT{Compute bias-corrected first moment estimate}
            }
            \STATE Compute $v_k = (\hat{H}_{S_j}+\rho I)^{-1}\hat{z}_k$ via \eqref{eq:NysSMWSolve} 
            \STATE $w_{k+1} = w_{k}-\eta v_{k}$ \hfill \COMMENT{Update parameters}
        \ENDFOR
\end{algorithmic}
\end{algorithm} 


\section{Proofs not appearing in the main paper}
In this section, we give the proofs of claims that are not present in the main paper. 
\subsection{Proof that SketchySGD is SGD in preconditioned space}
\label{subsection:PSGD}
Here we give the proof of \eqref{eq:SketchySGDIter} from \Cref{section:Introduction}.
As in the prequel, we set $P_j = \hat{H}^\rho_{S_j}$ in order to avoid notational clutter. 
Recall the SketchySGD update is given by
\[w_{k+1}=w_k-\eta_j P_j^{-1}g_{B_k},\]
where $\mathbb{E}_{B_k}[g_{B_k}] = g_k$.
We start by making the following observation about the SketchySGD update.
\begin{lemma}[SketchySGD is SGD in preconditioned space]
\label{lemma:UpdateLemma}
At outer iteration $j$ define $f_{P_j}(z) = f(P_j^{-1/2}z)$, that is define the change of variable $w = P_j^{-1/2}z$.
Then,
\begin{align*}
    & g_{P_j}(z) = P_j^{-1/2}g(P_j^{-1/2}z) = P^{-1/2}_jg(w) \\
    & H_{P_j}(z) = P_j^{-1/2}H(P_j^{-1/2}z)P_j^{-1/2} = P_j^{-1/2}H(w)P^{-1/2}_j. 
\end{align*}
Hence the SketchySGD update may be realized as
\begin{align*}
    & z_{k+1} = z_k-\eta_j \hat{g}_{P_j}(z_k)\\
    & w_{k+1} = P_j^{-1/2}z_{k+1},
\end{align*}
where $\hat{g}_{P_j}(z_k) = P_j^{-1/2}g_{B_k}(P_j^{-1/2}z_{k})$ is the stochastic gradient in preconditioned space.
\end{lemma}
\begin{proof}
The first display of equations follow from the definition of the change of variable and the chain rule, while the last display follows from definition of the SketchySGD update and the first display. 
\end{proof}

\subsection{Proof of \Cref{lemma:rel_quad}}
\label{subsection:quad_reg_pf}
\begin{proof}
    Let $w,w',w''\in \mathcal C$ and set $v = w'-w$. Then by Taylor's theorem and smoothness and convexity of $h$, it holds that
    \begin{align*}
        & h(w') =  h(w)+\langle \nabla f(w),v\rangle+\left(\int_{0}^{1}(1-t)\frac{\|v\|^2_{\nabla^2 h(w+tv)}}{\|v\|_{A(w'')}^2}dt\right)\|v\|_{A(w'')}^2 \\
        & =  h(w)+\langle \nabla h(w),v\rangle+\frac{\mathcal I}{2}\|v\|_{A(w'')}^2, \\
        & \text{where} ~\mathcal I = \int_{0}^{1}(1-t)\frac{\|v\|^2_{\nabla^2 h(w+tv)}}{\|v\|_{A(w'')}^2}.
    \end{align*}
    Now, if $h$ is just convex and $A(w'') = \nabla^2 h(w'')+\rho I$, a routine calculation shows that $\mathcal I \leq L/\rho$, which by definition 
    implies $\gamma_u(\mathcal C)\leq L/\rho$. 
    This gives the first statement. 
    For the second statement, note that when $h$ is $\mu$-strongly convex and $A(w'') = \nabla^2 h(w'')$, we have $\mu/L\leq \mathcal I\leq L/\mu$, which implies the second claim with $\mu/L \leq \gamma_\ell(\mathcal C) \gamma_u(\mathcal C)\leq L/\mu$.
\end{proof}
\subsection{Proof of \Cref{lemma:tau_rho_bnd}}
Below we provide the proof of \Cref{lemma:tau_rho_bnd}.
\label{subsection:tau_rho_lessn_pf}
\begin{proof}
By convexity of the $f_i$'s and the finite sum structure of $f$, it is easy to see that 
\[
\nabla^2 f_i(w)+\rho I \preceq n\left(H(w)+\rho I\right)~\forall i\in [n].
\]
Conjugating both sides by $(H^{\rho})^{-1/2}$, we reach
\[
\left(H(w)+\rho I\right)^{-1/2}\left(\nabla^2 f_i(w)+\rho I\right)\left(H(w)+\rho I\right)^{-1/2} \preceq n I.
\]
It now immediately follows that $\tau^{\rho}(H(w)) \leq n $.
On the other hand, for all $i\in\{1,\dots n\}$ we have 
\[
\nabla^2 f_i(w)+\rho I \preceq \left[\lambda_1 \left(\nabla^2 f_i(w)\right)+\rho \right] I \preceq \left(M(w)+\rho \right) I,  
\]
where the last relation follows by definition of $M(w)$.
So, conjugating by $\left(H^{\rho}(w)\right)^{-1/2}$ we reach
\[
\left(H^\rho(w)\right)^{-1/2}\left(\nabla^2 f_i(w)+\rho I\right)\left(H^\rho(w)\right)^{-1/2} \preceq \left(M(w)+\rho\right) \left(H(w)+\rho I\right)^{-1} \preceq \frac{M(w)+\rho}{\mu+\rho} I,
\]
which immediately yields $\tau^{\rho}(H(w)) \leq \frac{M(w)+\rho}{\mu+\rho}$.
Combining both bounds, we conclude
\[
\tau^\rho\left(H(w)\right) \leq \min \left\{n,\frac{M(w)+\rho}{\mu+\rho}\right \}.
\]
\end{proof}
\subsection{Proof of \cref{prop:tau_rho}}
\label{subsection:tau_rho_pf}
 The proof of \cref{prop:tau_rho} is the culmination of several lemmas.
 We begin with a \emph{truncated} intrinsic dimension Matrix Bernstein Inequality discussed, only requires bounds on the first and second moments that hold with some specified probability.
 It refines \cite{hopkins2016fast}, who established a similar result for the vanilla Matrix Bernstein Inequality. 
\begin{lemma}[Truncated Matrix Bernstein with intrinsic dimension]
\label{lemma:trunc_bern}
    Let $\{X_i\}_{i\in [n]}$ be a sequence of independent mean zero random matrices of the same size. Let $\beta \geq 0$ and $\{V_{1,i}\}_{i\in[n]},\{V_{2,i}\}_{i\in[n]}$ be sequences of matrices with $V_{1,i},V_{2,i}\succeq 0$ for all $i$. 
    Consider the event $\mathcal E_i = \left\{\|X_i\|\leq \beta,  X_iX_i^{T}\preceq V_{1,i}, X_i^{T}X_i\preceq V_{2,i}\right\}$. 
    Define $Y_i = X_i1_{\mathcal E_i}$, $Y = \sum^{n}_{i=1}Y_i$.
    Suppose that the following conditions hold
    \[
     \P\left(\mathcal E_i\right) \geq 1-\delta \quad \text{for all}~i \in[n], 
    \]
    \[
    \left\|\E[Y_i]\right\|\leq q.
    \]
    Set $V_1 = \sum^{n}_{i=1}V_{1,i}, V_2 = \sum^{n}_{i=1}V_{2,i}$, and define 
    \[\mathcal V = 
    \begin{bmatrix}
        V_1 & 0 \\
        0   &  V_2
    \end{bmatrix},~ \varsigma^2 = \max\{\|V_1\|,\|V_2\|.\}
    \]
    \newline
    Then for all $t\geq q+\varsigma+\frac{\beta}{3}$, $X = \sum^{n}_{i=1}X_i$ satisfies 
    \[
    \P\left(\|X\|\geq t\right) \leq n\delta+4\frac{\textup{trace}(\mathcal V)}{\|\mathcal V\|}\exp\left(\frac{-(t-q)^2/2}{\varsigma^2+\beta(t-q)/3}\right).
    \]
\end{lemma}
\begin{proof}
    The argument consists of relating $\P\left(\|X\|\geq t\right)$ to $\P\left(\|Y\|\geq t\right)$, the latter of which is easily bounded. 
    Indeed, from the definition of the $\mathcal E_i$'s, it is easily seen that
    \[
    \|Y_i\|\leq \beta,~\E\left[(Y-\E[Y])(Y-\E[Y])^{T}\right]\preceq V_1,~\E\left[(Y-\E[Y])^{T}(Y-\E[Y])\right]\preceq V_2. 
    \]
    Consequently, the intrinsic dimension Matrix Bernstein inequality \cite[Theorem~7.3.1]{tropp2015introduction} implies for any $s\geq \varsigma+\beta/3$, that
    \begin{equation}
    \label{eq:trunc_bern_bound}
    \P\left(\|Y-\E[Y]\|\geq s\right) \leq 4\frac{\textup{trace}(\mathcal V)}{\|\mathcal V\|}\exp\left(\frac{-s^2/2}{\varsigma^2+\beta s/3}\right).
    \end{equation}
     
    We now relate the tail probability of $\|X\|$ to the tail probability of $\|Y\|$. 
    To this end, the law of total probability implies
    \begin{align*}
        &\P\left(\|X\|\geq t\right) = \P\left(\|Y\|\geq t|X = Y\right)\mathbb{P}(X = Y)\\
        &+\P\left(\|X\|\geq t|X \neq Y\right)\mathbb{P}(X \neq Y) \\
        &\leq \P\left(\|Y\|\geq t|X = Y\right)+\P\left(\bigcup_{i=1}^{n}\mathcal E^{\complement}_i\right)\\ 
        &\leq \P\left(\|X\|\geq t|X = Y\right)+n\delta, 
    \end{align*}
    where the third inequality follows from $\{X\neq Y\}\subset \bigcup_{i=1}^{n}\mathcal E^{\complement}_i$, and the last inequality uses $\P\left(\bigcup_{i=1}^{n}\mathcal E^{\complement}_i\right)\leq \sum_{i=1}^{n}\left(1-\P(\mathcal E_i)\right)\leq n \delta$. 
    To bound $\P\left(\|X\|\geq t|X = Y\right)$, observe that
    \[
    \|X\| \leq \|X-\E[Y]\|+\|\E[Y]\| \leq \|X-\E[Y]\|+q,
    \]
    which implies
    \begin{align*}
    &\P\left(\|X\|\geq t|X = Y\right) \leq \P\left(\|X-\E[Y]\|+q\geq t|X = Y\right)\\
        & = \P\left(\|Y-\E[ Y]\|\geq t-q|X = Y\right).
    \end{align*}
    Inserting this last display into our bound for $\P\left(\|X\|\geq t\right)$ , we find
    \begin{align*}
        \P\left(\|X\|\geq t\right) &\leq n\delta+\P\left(\|Y-\E[ Y]\|\geq t-q|X = Y\right). \\
    \end{align*}
    To conclude, we apply $\eqref{eq:trunc_bern_bound}$ with $s = t-q$ to obtain
    \begin{align*}
        \P\left(\|X\|\geq t\right)\leq n\delta +4\frac{\textup{trace}(\mathcal V)}{\|\mathcal V\|}\exp\left(\frac{-(t-q)^2/2}{\varsigma^2+\beta(t-q)/3}\right),
    \end{align*}
    for all $t\geq q+\varsigma+\beta/3$.
\end{proof}

\begin{lemma}[Bounded statistical leverage]
\label{lemma:stat_lev}
    Let $D_{\infty}^{\rho} = H_{\infty}(w)^{1/2}H^{\rho}_{\infty}(w)^{-1}H_{\infty}(w)^{1/2}$, and set $\bar d^{\rho}_{\textup{eff}}(H_{\infty}^{\rho}(w)) = \max\{d^{\rho}_{\textup{eff}}(H_{\infty}^{\rho}(w)),1\}$. Then for some absolute constant $C>0$, the random vector $z$ satisfies
    \[
    \mathbb{P}\left\{\left\|\sqrt{\ell''(x^{T}w)}x\right\|_{H^{\rho}_{\infty}(w)^{-1}}^2>C\bar d^{\rho}_{\textup{eff}}(H_{\infty}(w))\log\left(\frac{1}{\delta}\right)\right\}\leq \delta.
    \]
\end{lemma}
\begin{proof}
Recall that $\sqrt{\ell''(x^{T}w)}x = H^{1/2}_{\infty}(w)z$, so that $\|\sqrt{\ell''(x^{T}w)}x\|_{H^{\rho}_{\infty}(w)^{-1}}^2 = \|z\|_{D^{\rho}_{\infty}}^2$.
As $z$ is $\nu$-sub-Gaussian and $\textrm{trace}(D_{\infty}^{\rho}) = d^{\rho}_{\textup{eff}}(H_{\infty}(w))$, Theorem 2.1 of \cite{hsu2012tail} with $\Sigma = D_{\infty}^{\rho}$ implies that
    \[
    \mathbb{P}\left\{\|z\|_{D_{\infty}^{\rho}}^2>\nu^2\left(d^{\rho}_{\textup{eff}}(H_{\infty}(w))+2\sqrt{d^{\rho}_{\textup{eff}}(H(w))t}+2t\right)\right\}\leq \exp(-t).
    \]
Setting $t = \bar d^{\rho}_{\textup{eff}}(H_{\infty}^{\rho}(w))\log(1/\delta)$, we obtain the desired claim with $C = 5\nu^2$.
\end{proof}

\begin{lemma}[Empirical Hessian concentration]
\label{lemma:AppxPopHess}
    Suppose $n = \tilde{\Omega}\left[\bar d^{\rho}_{\textup{eff}}(H_{\infty}(w))\log(n/\delta)\right]$, then 
    \[
    \left\|H_{\infty}^{\rho}(w)^{-1/2}\left[H(w)-H_{\infty}(w)\right]H_{\infty}^{\rho}(w)^{-1/2}\right\|\leq 1/2,
    \]
    with probability at least $1-\delta/n$.
\end{lemma}

\begin{proof}
We begin by writing
\begin{align*}
    H_{\infty}^{\rho}(w)^{-1/2}\left[H(w)-H_{\infty}(w)\right]H_{\infty}^{\rho}(w)^{-1/2} = \frac{1}{n}\sum_{i=1}^{n}\left(Z_iZ_i^{T}-D_{\infty}^{\rho}\right).
\end{align*}
where $Z_i = H_{\infty}^{\rho}(w)^{-1/2}\sqrt{\ell''(x^{T}w)}x_i$ and $D_{\infty}^{\rho} =  H_{\infty}^{\rho}(w)^{-1/2} H_{\infty}(w)H_{\infty}^{\rho}(w)^{-1/2}$.
Set $X_i = \frac{1}{n}\left(Z_iZ_i^{T}-D_{\infty}^{\rho}\right)$, and observe that $\E[X_i] = 0$.
We seek to apply \cref{lemma:trunc_bern}, to this end observe that
\cref{lemma:stat_lev} implies
\[
    \max_{i\in [n]} \mathbb{P}\left(\left\|\sqrt{\ell''(x_i^{T}w)}x_i\right\|_{H^{\rho}_{\infty}(w)^{-1}}^2>C\bar d^{\rho}_{\textup{eff}}(H_{\infty}(w))\log\left(\frac{2n}{\delta}\right)\right)\leq \frac{\delta}{2n^2}.
\]
Consequently, we obtain the following bounds on $\|X_i\|$ and $\E[X_i^2]$:
 \begin{align*}
 & \|X_i\| = \frac{1}{n}\max\left\{\lambda_{\textrm{max}}\left(Z_iZ_i^{T}-D^{\rho}_{\infty}\right),-\lambda_{\textrm{min}}\left(Z_iZ_i^{T}-D^{\rho}_{\infty}\right)\right\}\\ 
 &\leq \frac{1}{n}\max\left\{\|Z_i\|^2, \lambda_{\textrm{max}}(D_\infty^{\rho})\right\}\leq \frac{1}{n}\max\{\|Z_i\|^2,1\} \\
 &= \frac{1}{n}\max\left\{\left\|\sqrt{\ell''(x^{T}w)}x_i\right\|^2_{H^{\rho}_{\infty}(w)^{-1}},1\right\}\leq \frac{C\bar d^{\rho}_{\textup{eff}}(H_{\infty}(w))\log\left(\frac{2n}{\delta}\right)}{n}, 
 \end{align*}
 and 
 \begin{align*}
 \E[X_i^2] = \frac{1}{n^2}\E[\|Z_i\|^2Z_iZ_i^{T}]\preceq \frac{C\bar d^{\rho}_{\textup{eff}}(H_{\infty}(w))\log\left(\frac{2n}{\delta}\right)}{n^2}D^{\rho}_{\infty} ,    
 \end{align*}
 Hence setting $\beta = C\bar d^{\rho}_{\textup{eff}}(H_{\infty}(w))\log\left(\frac{2n}{\delta}\right)$ and $V_{i} = \beta D^{\infty}_{\rho}$, it follows immediately from the preceding considerations that
\[
\max_{i\in [n]} \P\left(\|X_i\|\leq \beta/n, \E[X_i^{2}]\preceq \frac{1}{n^2}V_i \right) \geq 1-\delta/(2n^2).
\]
As $V_1 = V_2 = V = \beta D^{\rho}_{\infty}/n$, it follows that $\|\mathcal V\|\leq \beta/n$. 
Moreover,
\begin{align*}
& \textrm{trace}(\mathcal V)/\|\mathcal V\| = \textrm{trace}(V)/\|V\| = \textrm{trace}(D_{\infty}^{\rho})/\|D_{\infty}^{\rho}\| \\
&= d^{\rho}_{\textrm{eff}}(H_{\infty}(w))\left[1+\rho/\lambda_{1}(H_{\infty}(w))\right]\leq 2 d^{\rho}_{\textrm{eff}}(H_{\infty}(w)),
\end{align*}
where the last inequality follows as $\rho\leq \lambda_{1}(H_{\infty}(w))$.
Thus, we can invoke \cref{lemma:trunc_bern} with
\[
t =  C\left(\sqrt{\frac{\beta\log\left(\frac{d^{\rho}_{\textrm{eff}}(H_{\infty}(w))}{\delta}\right)}{n}}+\frac{\beta\log\left(\frac{d^{\rho}_{\textrm{eff}}(H_{\infty}(w))}{\delta}\right)}{n}\right),
\]
to reach with probability at least $1-\delta/(2n)$ that
\begin{align*}
&\left\|H_{\infty}^{\rho}(w)^{-1/2}\left[H(w)-H_{\infty}(w)\right]H_{\infty}^{\rho}(w)^{-1/2}\right\| \\
&\leq C\left(\sqrt{\frac{\beta\log\left(\frac{d^{\rho}_{\textrm{eff}}(H_{\infty}(w))}{\delta}\right)}{n}}+\frac{\beta\log\left(\frac{d^{\rho}_{\textrm{eff}}(H_{\infty}(w))}{\delta}\right)}{n}\right).
\end{align*}
Recalling $\beta = \bigO\left(\bar d^{{\rho}}_{\textrm{eff}}(H_\infty(w))\log\left(\frac{n}{\delta}\right)\right)$ and $n = \Omega\left(\bar d^{{\rho}}_{\textrm{eff}}(H_\infty(w))\log\left(\frac{d^{\rho}_{\textrm{eff}}(H_{\infty}(w))}{\delta}\right)\log\left(\frac{n}{\delta}\right)\right)$, we conclude from the last display that 
\[
\P\left(\left\|H_{\infty}^{\rho}(w)^{-1/2}\left[H(w)-H_{\infty}(w)\right]H_{\infty}^{\rho}(w)^{-1/2}\right\|\leq 1/2\right)\geq 1-\delta/(2n).
\]
\end{proof}

\paragraph{Proof of \cref{prop:tau_rho}}
\begin{proof}
    Let $Z_i = H^{\rho}_{\infty}(w)^{-1/2}\sqrt{\ell''(x_i^{T}w)}x_i$ and observe
    the hypotheses on $n$, combined with \cref{lemma:stat_lev} and \cref{lemma:AppxPopHess} imply that
    \begin{align*}
    &\P\left(\max_{i\in [n]}\|Z_i\|^2 \leq C\bar d^{{\rho}}_{\textrm{eff}}(H_\infty(w))\log\left(\frac{n}{\delta}\right), \quad \frac{1}{2}H_{\infty}^{\rho}(w)\preceq H^{\rho}(w)\preceq \frac{3}{2}H_{\infty}^{\rho}(w)\right)\\
    &\geq 1-\delta/n.
    \end{align*}
    Combining the previous relation with matrix similarity we find,
    \begin{align*}
    & \lambda_1\left(H^{\rho}(w)^{-1/2}\nabla^2f^{\rho}_i(w)H^{\rho}(w)^{-1/2}\right) = \left(\nabla^2f^{\rho}_i(w)^{1/2}H^{\rho}(w)^{-1}\nabla^2f^{\rho}_i(w)^{1/2}\right) \\
    & \leq 2 \left(\nabla^2f^{\rho}_i(w)^{1/2}H_{\infty}^{\rho}(w)^{-1}\nabla^2f^{\rho}_i(w)^{1/2}\right) = 2\lambda_1\left(H_{\infty}^{\rho}(w)^{-1/2}\nabla^2f^{\rho}_i(w)H_{\infty}^{\rho}(w)^{-1/2}\right) \\
    &\leq 2+2\lambda_1\left(H_{\infty}^{\rho}(w)^{-1/2}\nabla^2f_i(w)H_{\infty}^{\rho}(w)^{-1/2}\right) = 2+2\lambda_1(Z_iZ_i^{T}) = 2+2\left\|Z_i\right\|^2\\
    &\leq C\bar d^{{\rho}}_{\textrm{eff}}(H_\infty(w))\log\left(\frac{n}{\delta}\right),
    \end{align*}
    Recalling that $\tau^{\rho}(H(w)) = \max_{i\in[n]} \lambda_1\left(H^{\rho}(w)^{-1/2}\nabla^2f^{\rho}_i(w)H^{\rho}(w)^{-1/2}\right)$, the last display and a union bound yield
    \[
    \P\left(\tau^{\rho}(H(w))\leq C\bar d^{{\rho}}_{\textrm{eff}}(H_\infty(w))\log\left(\frac{n}{\delta}\right)\right)\geq 1-\delta,
    \]
    as desired.
\end{proof}

\subsection{Proof of \cref{lemma:SubsampAppx}}
\label{subsection:subsamp_appx_pf}
\begin{proof}
The result is a consequence of a standard application of the intrinsic dimension Matrix Bernstein inequality.
Indeed, let $D^{\rho}= (H(w)+\rho I)^{-1/2}H(w)(H(w)+\rho I)^{-1/2}$ and $X_i = \frac{1}{b_h}\left(Z_{i}Z_{i}^{T}-D^{\rho}\right)$, where $Z_i = (H(w)+\rho I)^{-1/2}\nabla^2 f_i(w)^{1/2}$. 
Observe that $\E[X_i] = 0$, and set $X = \sum_{i} X_i$. 
To see that the conditions of the intrinsic dimension Matrix Bernstein inequality are met, note that $X_i$ and $\E[X^2]$ satisfy 
\begin{align*}
    & \|X_i\| = \frac{1}{b_h}\|Z_i\|^2 \leq \frac{\tau^{\rho}(H(w))}{b_h}, \\
    &  \E[X^2] \preceq \frac{1}{b^2_h}\sum_i\E[\|z_i\|^2 z_iz_{i}^{T}] \preceq \frac{\tau^{\rho}(H(w))}{b_h}D^{\rho} \coloneqq \mathcal V. 
\end{align*}
Moreover as $\rho \leq \lambda_1(H(w))$,
\[
\textrm{tr}(\mathcal V)/\|\mathcal V\|\leq 2 d_{\textrm{eff}}^{\rho}(H(w)).
\]
Thus, the intrinsic dimension Matrix Bernstein inequality \cite[Theorem~7.3.1]{tropp2015introduction} implies
\[
\P\left\{\|X\|\geq t \right\}\leq 8d_{\textrm{eff}}^{\rho}(H(w))\exp\left(-\frac{b_ht^2/2}{\tau^{\rho}(H(w))\left(1+t/3\right)}\right),
\]
for all $t\geq \sqrt{\tau^{\rho}(H(w))/b_h}+\tau^{\rho}(H(w))/(3b_h)$.
So, setting  
\[
t = \sqrt{\frac{4\tau^{\rho}(H(w))\log\left(\frac{8d_{\textrm{eff}}^{\rho}(H(w))}{\delta}\right)}{b_h}}+\frac{4}{3b_h}\tau^{\rho}(H(w))\log\left(\frac{8d_{\textrm{eff}}^{\rho}(H(w))}{\delta}\right),
\]
and $b_h = \bigO\left(\frac{\tau^{\rho}(H(w))\log\left(\frac{d_{\textrm{eff}}^{\rho}(H(w))}{\delta}\right)}{\zeta^2}\right)$, it holds that 
\[
\P\left(\|X\|\leq \frac{\zeta}{1+\zeta}\right)\geq 1-\delta.
\]
This last display immediately implies that
\[
\left(1-\frac{\zeta}{1+\zeta}\right)H^{\rho}(w)\preceq H^{\rho}_S(w) \preceq \left(1+\frac{\zeta}{1+\zeta}\right) H^{\rho}(w),
\]
which is equivalent to
    \[
    \left(1+\frac{\zeta}{1+\zeta}\right)^{-1}H_S^{\rho}\preceq H^{\rho}\preceq \left(1-\frac{\zeta}{1+\zeta}\right)^{-1}H^{\rho}_{S}. 
    \]
    The desired claim now follows from the last display, upon observing that 
    \[
    1-\zeta\leq \left(1+\frac{\zeta}{1+\zeta}\right)^{-1}\leq \left(1-\frac{\zeta}{1+\zeta}\right)^{-1}=1+\zeta.
    \]
\end{proof}
\subsection{Proof of \cref{prop:NysPrecondLem} and \cref{corr:NysLoewnUnionBnd}}
\label{subsection:NysPrecondPf}
In order to relate $\hat H^\rho_{S_k}$ to $H_{S_k}$, we need to control the error resulting from the low-rank approximation.  
To accomplish this, we recall the following result from \cite{zhao2022nysadmm}.
\begin{lemma}[Controlling low-rank approximation error, Lemma A.7 Zhao et al. (2022) \cite{zhao2022nysadmm}]
	\label{lemma:NysErrLemma}
	Let $\tau>0$ and $E = H_{S}-\hat{H}_{S}$. Construct a randomized Nystr{\"o}m approximation from a standard Gaussian random matrix $\Omega$ with rank $r = \bigO(d^{\tau}_{\textup{eff}}(H_{S})+\log(\frac{1}{\delta}))$. Then 
	\begin{equation}
		\mathbb P \left(\|E\|\leq \tau\right)\geq 1-\delta.
	\end{equation}
\end{lemma}

\paragraph{Proof of \cref{prop:NysPrecondLem}}
\begin{proof}
Let $E =  H_{S}-\hat{H}_{S}$, and note by the properties of the Nystr{\"o}m approximation that $E\succeq 0$ \cite{tropp2017fixed,frangella2021randomized}. 
Now by our hypothesis on $r$, it follows from \cref{lemma:NysErrLemma}, that  
\[\mathbb P\left(\|E\|\leq \zeta \rho/4\right)\geq 1-\delta/2.\]
Using the decomposition $H_{S}^{\rho} = P+E$, we apply Weyl's inequalities to find 
\begin{align*}
&\lambda_{1}(P^{-1/2} H_{S}^{\rho} P^{-1/2}) \leq  \lambda_{1}\left(P^{-1/2}\hat{H}^{\rho}_{S}P^{-1/2}\right)+\lambda_{1}\left(P^{-1/2}EP^{-1/2}\right) = \\
& 1+\|P^{-1/2}E P^{-1/2}\| \leq 1+\|P^{-1}\|\|E\|\leq 1+\frac{\|E\|}{\rho}\leq 1+\zeta/4.
\end{align*}
Moreover as $E\succeq 0$ we have $H^{\rho}_S\succeq P$, so conjugation yields $P^{-1/2} H_{S}^{\rho} P^{-1/2}\succeq I_{p}$.
The preceding inequality immediately yields $\lambda_{p}(P^{-1/2} H_{S}^{\rho} P^{-1/2})\geq 1.$
Hence
\[1\leq\lambda_{p}(P^{-1/2} H_{S}^{\rho} P^{-1/2})\leq \lambda_{1}(P^{-1/2} H_{S}^{\rho} P^{-1/2})\leq 1+\zeta/4.\]
As an immediate consequence of this last display, we obtain the Loewner ordering relation
\[P\preceq H_{S}^{\rho} \preceq (1+\zeta/4)P.\]
Now, \cref{lemma:SubsampAppx} and a union bound implies that
\[
\P\left(P\preceq H_{S}^{\rho} \preceq (1+\zeta/4)P,~(1-\zeta/4)H_S^{\rho}\preceq H^{\rho}\preceq (1+\zeta/4)H^{\rho}_{S}\right)\geq 1-\delta.
\]
Combining these relations and using that $(1+\zeta/4)^2\leq 1+\zeta$ for $\zeta\in(0,1)$, we find
\[
(1-\zeta)P\preceq H^{\rho}\preceq (1+\zeta)P.
\]
To conclude, note that $(1+\rho/\mu)H^{\rho}\preceq H\preceq H^{\rho}$, which combined with the last display implies
\[
(1-\zeta)\frac{1}{1+\rho/\mu}P\preceq H\preceq (1+\zeta) P.
\]
\end{proof}


\subsection{Proof of \cref{lemma:PrecondStrnConvexity}}
\label{subsection:PreStrnCvxPf}
\begin{proof}
The function $f$ is smooth and strongly convex, so it is quadratically regular.
Consequently, 
\[
f(w)\geq f(w_\star)+\frac{\hat \gamma_\ell}{2}\|w-w_\star\|_{H(w_P)}^2.
\]
Hence we have
\begin{align*}
    &f(w)-f(w_\star) \geq \frac{\gamma_\ell}{2}(1-\zeta)\frac{1}{1+\rho/\mu}\|w-w_\star\|^2_{P},
\end{align*}
where in the last inequality we have used the hypothesis that the conclusion of \cref{prop:NysPrecondLem} holds.
The claim now follows by recalling that $\hat \gamma_{\ell} = \frac{(1-\zeta)\mu}{\mu+\rho}\gamma_\ell$.
\end{proof}



\subsection{Proof of \cref{prop:PrecondSmoothGrad}}
\label{subsection:PrecondSmoothGradPf}
In this subsection we prove \cref{prop:PrecondSmoothGrad}, which controls the variance of the preconditioned minibatch gradient.
We start by proving the following more general result, from which \cref{prop:PrecondSmoothGrad} follows immediately. 
\begin{proposition}[Expected smoothness in the dual-norm]
\label{prop:dual_ES}
Suppose that each $f_i$ is convex and satisfies
\[
f_i(w+h)\leq f_i(w)+\langle g_i(w),h\rangle +\frac{L_i}{2}\|h\|_{M_i}^2, \quad \forall w, h \in \R^{p},
\]
for some symmetric positive definite matrix $M_i$.
Moreover, let $f$ satisfy
\[
f(w+h)\leq f(w)+\langle g(w),h\rangle +\frac{L}{2}\|h\|_{M}^2, \quad \forall w, h \in \R^{p},
\]
where $M = \frac{1}{n}\sum_{i=1}^{n}M_i$.
Define $\tau(M) \coloneqq \max_{1\leq i\leq n}\lambda_1\left(M^{-1/2}M_iM^{-1/2}\right)$.
Further, suppose we construct the gradient sample $g_{B}(w)$ with batch-size $b_g$. 
Then for every $w\in \R^p$
\begin{equation*}
\mathbb{E}_{B} \|g_{B}(w)-g_{B}(w')\|_{M^{-1}}^2 \le 2\mathcal{L}_M (f(w) - f(w')-\langle g(w'),w-w'\rangle),
\end{equation*}
\begin{equation*}
    \mathbb{E}_{B} \|g_{B}(w)\|_{M^{-1}}^2 \le 4\mathcal{L}_M (f(w) - f(w_{\star}))+2\sigma_M^2,
\end{equation*}
where
\[
\mathcal{L}_{M} = \frac{n(b_g-1)}{b_g(n-1)}L+\tau(M)\frac{n-b_g}{b_g(n-1)}\textup{max}_{1\leq i\leq n}L_i,\]
and 
\[
\sigma_M^2 = \frac{n-b_g}{b_{g}(n-1)}\frac{1}{n}\sum_{i=1}^{n}\|\nabla f_i(w_\star)\|_{M^{-1}}^2.
\]
\end{proposition}
\begin{proof}
Introduce the change of variable $w = M^{-1/2}z$, so that $f_i(w) = f_i(M^{-1/2}z) = f_{i,M}(z)$ and $f(w) = f(M^{-1/2}z)=f_{M}(z)$. Then by our hypotheses on the $f_i$'s, $f$, and the definition of $\tau(M)$, we have that for all $z, h\in \R^{p}$
\[
f_{i,M}(z+h)\leq f_{i,M}(z)+\langle g_{i,M}(z),h\rangle+\frac{\tau(M) L_i}{2}\|h\|^2,
\]
\[
f_{M}(z+h)\leq f_{M}(z)+\langle g_{M}(z),h\rangle+\frac{L}{2}\|h\|^2.
\]
Consequently, Proposition 3.8 of \cite{gower2019sgd} implies
\[
\E\|g_{B,M}(z)-g_{B,M}(z')\|^2\leq 2\mathcal L_M\left(f_{M}(z)-f_{M}(z)-\langle g_{M}(z'),z-z'\rangle\right),
\]
\[
\E\|g_{B,M}(z_\star)\|^{2}=\frac{n-b_g}{b_g(n-1)}\frac{1}{n}\sum^{n}_{i=1}\|g_{M,i}(z_\star)\|^2,
\]
where $\mathcal L_{M} = \frac{n(b_g-1)}{b_g(n-1)}L+\tau(M)\frac{n-b_g}{b_g(n-1)}\textup{max}_{1\leq i\leq n}L_i$.
Invoking \cref{lemma:UpdateLemma}, the above displays become
\[
\E\|g_{B}(w)-g_{B}(w')\|_{M^{-1}}^2\leq 2\mathcal L_{M}\left(f(w)-f(w')-\langle g(w'),w-w'\rangle\right),
\]
\[
\E\|g_{B}(w_\star)\|_{M^{-1}}^{2}=\frac{n-b_g}{b_g(n-1)}\frac{1}{n}\sum^{n}_{i=1}\|g_{i}(w_\star)\|_{M^{-1}}^2 = \sigma^2.
\]
The last portion of the desired claim now follows by combining the preceding displays with $w' = w_\star$, and the identity $\E\|a\|_{M^{-1}}^2\leq 2\E\|a-b\|_{M^{-1}}^2+2\E\|b\|_{M^{-1}}^2$.

\paragraph{Proof of \cref{prop:PrecondSmoothGrad}}
\begin{proof}
    Set $M_i = P$ so that $M = H^{\rho}(w_P)$. 
    Now, by the assumption that $H(w_P)\preceq (1+\zeta)P$ and item 1 of \cref{lemma:rel_quad} it holds that each $f_i$ is smooth with respect to $M_i$ with $L_i = (1+\zeta)\tau^{\rho}H(w_P)\gamma_{i,u}^{\rho}$, while $f$ is smooth with respect to $M$ with $L = (1+\zeta)\gamma^{\rho}_{u}$. 
    Noting that $\tau(M) = 1$, the claim follows from \cref{prop:dual_ES}.
\end{proof}

\section{Lower bound on condition number in \cref{table:datasets}}
\label{section:kappa_lwr_bnd}
Recall the condition number $\kappa$ is given by
\[
\kappa = \frac{\sup_{w\in \R^p}\lambda_{1}(H(w))}{\inf_{w\in \R^p}\lambda_{p}(H(w))}.
\]
Consequently when $f$ is a the least-squares or logistic loss with data matrix $A$, and $l^2$-regularization $\mu\geq 0$, it holds for any $r\geq p$ that
\[
\kappa \geq \frac{\lambda_{1}(H(0))}{\lambda_p(H(0))}\geq \frac{\lambda_{1}(H(0))}{\lambda_r(H(0))} = \frac{\sigma_{1}^2(A)/n+\mu}{\sigma_r^2(A)/n+\mu}.
\]
Hence $\kappa$ is lower bounded by $(\sigma_{1}^2(A)/n+\mu)/(\sigma_{r}^2(A)/n+\mu)$. 
\cref{table:datasets} gives the numerical value for this lower bound for $r = 100$. 

\end{proof}

\section{Experimental details}
\label{section:exp_details}
Here we provide more details regarding the experiments in \Cref{section:Experiments}.

\paragraph{Regularization}
For convex problems (\Cref{subsubsection:performance_auto,subsubsection:performance_tuned,subsection:performance_som,subsection:performance_pcg,subsection:large_scale,subsection:performance_lr_ablation}), we set the $l_2$-regularization to $10^{-2}/n_{\mathrm{tr}}$, where $n_{\mathrm{tr}}$ is the size of the training set.
For deep learning (\Cref{subsection:performance_dl}), we use no regularization or weight decay, as the experiments are proof-of-concept.

\paragraph{Ridge regression datasets} 
The ridge regression experiments are run on the datasets described in the main text.
E2006-tfidf and YearPredictionMSD's rows are normalized to have unit-norm, while we standardize the features of yolanda.
For YearPredictionMSD we use a ReLU random features transformation that gives us $4367$ features in total. 
For yolanda we use a random features transformation with bandwidth $1$ that gives us $1000$ features in total, and perform a random 80-20 split to form a training and test set. 
In \Cref{table:ridgedata}, we provide the dimensions of the datasets, where $n_{\textrm{tr}}$ is the number of training samples, $n_{\textrm{test}}$ is the number of testing samples, and $p$ is the number of features.

\begin{table}[!htbp]
	\caption{\label{table:ridgedata} Dimensions of Ridge Regression Datasets} 
	\begin{center}
		\footnotesize
		\begin{tabular}{|c|c|c|c|}
			\hline
			\textbf{Dataset} & $n_{\textrm{tr}}$ & $n_{\textrm{test}}$ & $p$\\
			\hline
			E2006-tfidf & $16087$ & $3308$ & $150360$\\
			\hline
			YearPredictionMSD & $463715$ & $51630$ & $4367$\\
			\hline
			yolanda & $320000$ & $80000$ & $1000$\\
			\hline
		\end{tabular}
	\end{center}
\end{table}

\paragraph{Logistic regression datasets}
The logistic regression experiments are run on the datasets described in the main text. 
All datasets' rows are normalized so that they have unit norm. 
For ijcnn1 and susy we use a random features transformation with bandwidth $1$ that gives us $2500$ and $1000$ features, respectively.
For real-sim, we use a random 80-20 split to form a training and test set. 
For HIGGS, we repeatedly apply a random features transformation with bandwidth $1$ to obtain $10000$ features, as described in \Cref{subsection:large_scale}.
In \Cref{table:logdata}, we provide the dimensions of the datasets, where $n_{\textrm{tr}}$ is the number of training samples, $n_{\textrm{test}}$ is the number of testing samples, and $p$ is the number of features.

\begin{table}[!htbp]
	\caption{\label{table:logdata} Dimensions of Logistic Regression Datasets} 
	\begin{center}
		\footnotesize
		\begin{tabular}{|c|c|c|c|}
			\hline
			\textbf{Dataset} & $n_{\textrm{tr}}$ & $n_{\textrm{test}}$ & $p$\\
			\hline
			ijcnn1 & $49990$ & $91701$ & $2500$\\
			\hline
			susy & $4500000$ & $500000$ & $1000$\\
			\hline
			real-sim & $57847$ & $14462$ & $20958$\\
			\hline
			HIGGS & $10500000$ & $500000$ & $10000$\\
			\hline
		\end{tabular}
	\end{center}
\end{table}

\paragraph{Deep learning datasets}
The deep experiments are run on the datasets described in the main text. 
We download the datasets using the OpenML-Python connector \cite{vanschoren2019openmlpython}.
Each dataset is standardized to have zero mean and unit variance, and the statistics for standardization are calculated using only the training split.
In \Cref{table:dldata}, we provide the dimensions of the datasets, where $n$ is the number of samples (before splitting into training, validation, and test sets), $p$ is the number of features, and ID is the unique identifier of the dataset on OpenML.

\paragraph{Dataset augmentation for scaling (\Cref{subsection:performance_som,subsection:performance_pcg})}
We perform data augmentation before any additional preprocessing steps (e.g., normalization, standardization, random features). 
To increase the samples by a factor of $k$, we duplicate the dataset a total of $k - 1$ times. 
For each duplicate, we generate a random Gaussian matrix, where each element has variance $0.02$.
For sparse datasets, this Gaussian matrix is generated to have the same number of nonzeros as the original dataset.
Each duplicate and Gaussian matrix is summed; the resulting sums are stacked to form the augmented dataset.

\begin{table}[!htbp]
	\caption{\label{table:dldata} Dimensions of Deep Learning Datasets} 
	\begin{center}
		\footnotesize
		\begin{tabular}{|c|c|c|c|}
			\hline
			\textbf{Dataset} & $n$ & $p$ & ID\\
			\hline
            Fashion-MNIST & $70000$ & $784$ & $40996$\\
            \hline
            Devnagari-Script & $92000$ & $1024$ & $40923$\\
            \hline
            volkert & $58310$ & $180$ & $41166$\\
			\hline
		\end{tabular}
	\end{center}
\end{table}

\paragraph{Random seeds}
In \Cref{subsection:performance_fom,subsection:performance_lr_ablation,section:sensitivity_appdx} we run all experiments with $10$ random seeds, with the exception of susy, for which we use $3$ random seeds.

We use the same number of random seeds in \Cref{subsection:performance_som,subsection:performance_som}, except for the scaling experiments.
For the scaling experiments, we only use $3$ random seeds.

In \Cref{subsection:large_scale} we use only $1$ random seed due to the sheer size of the problem.

In \Cref{subsection:performance_dl} we use only $1$ random seed for each learning rate given by random search.
However, we use $10$ random seeds to generate the results with the tuned learning rate.

\paragraph{Additional hyperparameters (\Cref{subsection:performance_fom,subsection:large_scale})} 
For SVRG we perform a full gradient computation at every epoch. 

For L-Katyusha, we initialize the update probability $p_{\mathrm{upd}} = b_g / n_{\mathrm{tr}}$ to ensure the average number of iterations between full gradient computations is equal to one epoch. 
We follow \cite{kovalev2020lkatyusha} and set $\mu$ equal to the $\ell^2$-regularization parameter, $\sigma = \frac{\mu}{L}, \theta_1 = \min\{\sqrt{2\sigma n_{\mathrm{tr}}/3}, \frac{1}{2} \}$, and $\theta_2 = \frac{1}{2}$. 

All algorithms use a batch size of $256$ for computing stochastic gradients, except on the HIGGS dataset. For the HIGGS dataset, SGD, SAGA, and SketchySGD are all run with a batch size of $4096$.

\paragraph{Additional hyperparameters (\Cref{subsection:performance_som})}
For SLBFGS we perform a full gradient computation at every epoch. 
Furthermore, we update the inverse Hessian approximation every epoch and set the Hessian batch size to $\sqrt{n_{\mathrm{tr}}}$, which matches the Hessian batch size hyperparameter in SketchySGD.
In addition, we follow \cite{moritz2016linearly} and set the memory size of SLFBGS to $10$. 
We use a batch size of $256$ for computing stochastic gradients.

We use the defaults for L-BFGS provided in the \href{https://docs.scipy.org/doc/scipy/reference/generated/scipy.optimize.fmin_l_bfgs_b.html}{SciPy implementation}, only tuning the ``factr'' parameter when necessary to avoid early termination.

For RSN, we fix the sketch size to $250$.

For Newton Sketch, we set the sketch size to $2 \cdot 10^{-3} \cdot \min(n_{\mathrm{tr}}, p)$.

For RSN and Newton Sketch, we follow the original publications' suggestions \cite{gower2019rsn,pilanci2017newton} for setting the line search parameters.

\paragraph{Additional hyperparameters (\Cref{subsection:performance_pcg})}
We set the sketch size to $2 \cdot 10^{-3} \cdot \min(n_{\mathrm{tr}}, p)$ for all three of Nystr\"{o}mPCG, GaussPCG, and SparsePCG.

\paragraph{Additional hyperparameters (\Cref{subsection:performance_dl})}
For all of the competitor methods (except Shampoo), we use the default hyperparameters.
For Shampoo, we modify the preconditioner update frequency to occur every epoch, similar to SketchySGD.
For SketchySGD, we set the momentum parameter $\beta$ to $0.9$, just as in Adam.
We compute stochastic gradients using a batch size of $128$.
For learning rate scheduling, we use cosine annealing with restarts.
For the restarts, we use an initial budget of 15 epochs, with a budget multiplier of 2.

\paragraph{Default hyperparameters for SAGA/SVRG/L-Katyusha} 
The theoretical analysis of SVRG, SAGA, and L-Katyusha all yield recommended learning rates that lead to linear convergence.
In practical implementations such as scikit-learn \cite{pedregosa2011scikit}, these recommendations are often taken as the default learning rate. 
For SAGA, we follow the scikit-learn implementation, which uses the following learning rate:
\[
\eta = \max\left\{\frac{1}{3L}, \frac{1}{2\left(L+n_{\textrm{tr}}\mu\right)}\right\},
\]
where $L$ is the smoothness constant of $f$ and $\mu$ is the strong convexity constant.
The theoretical analysis of SVRG suggests a step-size of $\eta = \frac{1}{10\mathcal L}$, where $\mathcal L$ is the expected-smoothness constant. 
We have found this setting to pessimistic relative to the SAGA default, so we use the same default for SVRG as we do for SAGA. 
For L-Katyusha the hyperparameters $\theta_1$ and $\theta_2$ are controlled by how we specify $L^{-1}$, the reciprocal of the smoothness constant. Thus, the default hyperparameters for all methods are controlled by how $L$ is specified. 

Now, standard computations show that the smoothness constants for least-squares and logistic regression satisfy 
\[
L_{\textrm{least-squares}} \leq \frac{1}{n_{\textrm{tr}}}\sum_{i=1}^{n_{\textrm{tr}}}\|a_i\|^2,  
\]
\[
L_{\textrm{logistic}} \leq \frac{1}{4n_{\textrm{tr}}}\sum_{i=1}^{n_{\textrm{tr}}}\|a_i\|^2.
\]
The scikit-learn software package uses the preceding upper-bounds in place of $L$ to set $\eta$ in their implementation of SAGA.
We adopt this convention for setting the hyperparameters of SAGA, SVRG and L-Katyusha. 
We display the defaults for the three methods in \cref{table:param_auto}.

\begin{table}[!htbp]
    \caption{\label{table:param_auto} Default hyperparameters for SVRG/SAGA/L-Katyusha} 
    \centering
    \resizebox{\columnwidth}{!}{
        \scriptsize
		\begin{tabular}{|c|c|c|c|c|c|c|c|}
			\hline
			\textbf{Method\textbackslash Dataset} & E2006-tfidf & YearPredictionMSD & yolanda & ijcnn1 & real-sim & susy & HIGGS\\
			\hline
            SVRG/SAGA & $4.95 \cdot 10^{-1}$ & $1.01 \cdot 10^0$ & $4.96 \cdot 10^{-1}$ & $1.91 \cdot 10^0$ & $1.93 \cdot 10^0$ & $1.87 \cdot 10^0$ & $1.93 \cdot 10^0$\\
			\hline
            L-Katyusha & $1.00 \cdot 10^0$ & $4.85 \cdot 10^{-1}$ & $9.98 \cdot 10^{-1}$ & $2.52 \cdot 10^{-1}$ & $2.50 \cdot 10^{-1}$ & $2.57 \cdot 10^{-1}$ & N/A\\
            \hline
		\end{tabular}
    }
\end{table}

\paragraph{Grid search parameters (\Cref{subsubsection:performance_tuned,subsection:performance_som})} We choose the grid search ranges for SVRG, SAGA, and L-Katyusha to (approximately) include the default hyperparameters across the tested datasets (\cref{table:param_auto}). 
For ridge regression, we set $[10^{-3}, 10^{2}]$ as the search range for the learning rate in SVRG and SAGA, and $[10^{-2}, 10^{0}]$ as the search range for the smoothness parameter $L$ in L-Katyusha. 
Similar to SVRG and SAGA, we set $[10^{-3}, 10^{2}]$ as the search range for SGD. 
For SLBFGS, we set the search range to be $[10^{-5}, 10^{0}]$ in order to have the same log-width as the search range for SGD, SVRG, and SAGA. 
In logistic regression, the search ranges for SGD/SVRG/SAGA, L-Katyusha, and SLBFGS become $[4 \cdot 10^{-3}, 4 \cdot 10^{2}], [2.5 \cdot 10^{-3}, 2.5 \cdot 10^{-1}]$, and $[4 \cdot 10^{-5}, 4 \cdot 10^{0}]$, respectively. 
The grid corresponding to each range samples $10$ equally spaced values in log space.
The tuned hyperparmeters for all methods across each dataset are presented in \Cref{table:param_tuned}.
\begin{table}[!htbp]
    \caption{\label{table:param_tuned} Tuned hyperparameters for competitor methods} 
	\begin{center}
    \scriptsize
		\begin{tabular}{|c|c|c|c|c|c|c|}
			\hline
			\textbf{Method\textbackslash Dataset} & E2006-tfidf & YearPredictionMSD-rf & yolanda-rf & ijcnn1-rf & real-sim & susy-rf\\
			\hline
			SGD & $5.99 \cdot 10^{-1}$ & $2.15 \cdot 10^0$ & $5.99 \cdot 10^{-1}$ & $8.62 \cdot 10^0$ & $4 \cdot 10^2$ & $8.62 \cdot 10^0$\\
			\hline
            SVRG & $5.99 \cdot 10^{-1}$ & $2.15 \cdot 10^0$ & $2.15 \cdot 10^0$ & $8.62 \cdot 10^0$ & $4 \cdot 10^2$ & $8.62 \cdot 10^0$\\
			\hline
            SAGA & $5.99 \cdot 10^{-1}$ & $2.15 \cdot 10^0$ & $2.15 \cdot 10^0$ & $8.62 \cdot 10^0$ & $4 \cdot 10^2$ & $8.62 \cdot 10^0$\\
			\hline
            L-Katyusha & $2.15 \cdot 10^{-1}$ & $1.29 \cdot 10^{-1}$ & $2.15 \cdot 10^{-1}$ & $1.94 \cdot 10^{-2}$ & $2.5 \cdot 10^{-3}$ & $3.32 \cdot 10^{-2}$\\
			\hline
            SLBFGS & $7.74 \cdot 10^{-2}$ & $2.15 \cdot 10^{-2}$ & $1.67 \cdot 10^{-3}$ & $1.11 \cdot 10^0$ & $8.62 \cdot 10^{-2}$ & $8.62 \cdot 10^{-2}$\\
            \hline
		\end{tabular}
	\end{center}
\end{table}

\paragraph{Grid search parameters (\Cref{subsection:large_scale})} Instead of using a search range of $[4 \cdot 10^{-3}, 4 \cdot 10^{2}]$ for SGD/SAGA, we narrow the range to $[4 \cdot 10^{-2}, 4 \cdot 10^{1}]$ and sample $4$ equally spaced values in log space. 
The reason for reducing the search range and grid size is to reduce the total computational cost of running the experiments on the HIGGS dataset. 
Furthermore, we find that $4 \cdot 10^{0}$ is the best learning rate for HIGGS, while $4 \cdot 10^{1}$ leads to non-convergent behavior, meaning these search ranges are appropriate.

\paragraph{Random search parameters (\Cref{subsection:performance_dl})} We tune the learning rate for each optimizer using 30 random search trials with log-uniform sampling in the range $[10^{-3}, 10^{-1}]$. The tuning is performed with Optuna \cite{akiba2019optuna}.

\section{Sensitivity experiments}
\label{section:sensitivity_appdx}
In this section, we investigate the sensitivity of SketchySGD to the rank hyperparameter $r$ (\Cref{subsection:sensitivity_r}) and update frequency hyperparameter $u$ (\Cref{subsection:sensitivity_u}). 
In the first set of sensitivity experiments, we select ranks $r \in \{1, 2,5, 10, 20, 50\}$ while holding the update frequency fixed at $u = \left\lceil \frac{n_{\textrm{tr}}}{b_g} \right\rceil$ (1 epoch)\footnote{If we set $u = \infty$ in ridge regression, which fixes the preconditioner throughout the run of SketchySGD, the potential gain from a larger rank $r$ may not be realized due to a poor initial Hessian approximation.}.
In the second set of sensitivity experiments, we select update frequencies $u \in \left\{0.5\left\lceil \frac{n_{\textrm{tr}}}{b_g} \right\rceil, 
\left\lceil \frac{n_{\textrm{tr}}}{b_g} \right\rceil, 
2\left\lceil \frac{n_{\textrm{tr}}}{b_g} \right\rceil, 
5\left\lceil \frac{n_{\textrm{tr}}}{b_g} \right\rceil,
\infty \right\}$, while holding the rank fixed at $r = 10$. 
We use the datasets from \Cref{table:datasets}. 
Each curve is the median performance of a given $(r, u)$ pair across $10$ random seeds (except for susy, which uses $3$ seeds), run for $40$ epochs.

\subsection{Effects of changing the rank}
\label{subsection:sensitivity_r}
Looking at \cref{fig:sensitivity_r}, we see two distinct patterns: either increasing the rank has no noticeable impact on performance (E2006-tfidf, real-sim), or increasing the rank leads to faster convergence to a ball of noise around the optimum (YearPredictionMSD-rf). 
We empirically observe that these patterns are related to the spectrum of each dataset, as shown in \cref{fig:spectrums}. 
For example, the spectrum of E2006-tfidf is highly concentrated in the first singular value, and decays rapidly, increased rank does not improve convergence. 
On the other hand, the spectrum of YearPredictionMSD-rf is not as concentrated in the first singular value, but still decays rapidly, so convergence improves as we increase the rank  from $r = 1$  to $r = 10$, after which performance no longer improves, in fact it  slightly degrades.
The spectrum of real-sim decays quite slowly in comparison to E2006-tfidf or YearPredictionMSD-rf, 
so increasing the rank up to $50$ does not capture enough of the spectrum to improve convergence.
One downside in increasing the rank is that the quantity $\eta_{\textrm{SketchySGD}}$ \eqref{eq:SSG_LR} can become large, leading to SketchySGD taking a larger step size.
As a result, SketchySGD oscillates more about the optimum, as seen in YearPredictionMSD-rf (\cref{fig:sensitivity_r}).
Last, \cref{fig:sensitivity_r} shows $r=10$ delivers great performance across all datasets, supporting its position as the recommended default rank. 
Rank sensitivity plots for yolanda-rf, ijcnn1-rf, and susy-rf appear in \Cref{section:sensitivity_appdx}.

\begin{figure}[htbp]
    \centering
    \includegraphics[scale=0.5]{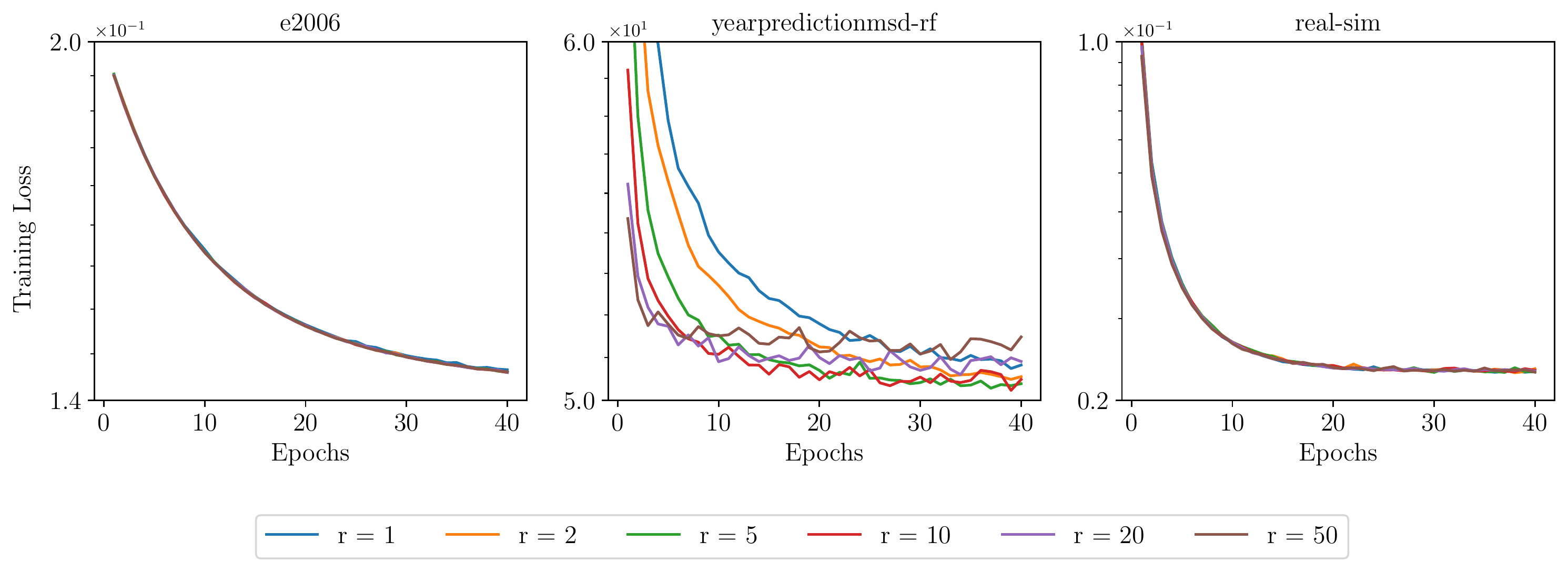}
    \caption{Sensitivity of SketchySGD to rank $r$.}
    \label{fig:sensitivity_r}
\end{figure}

\begin{figure}[htbp]
    \centering
    \includegraphics[scale=0.5]{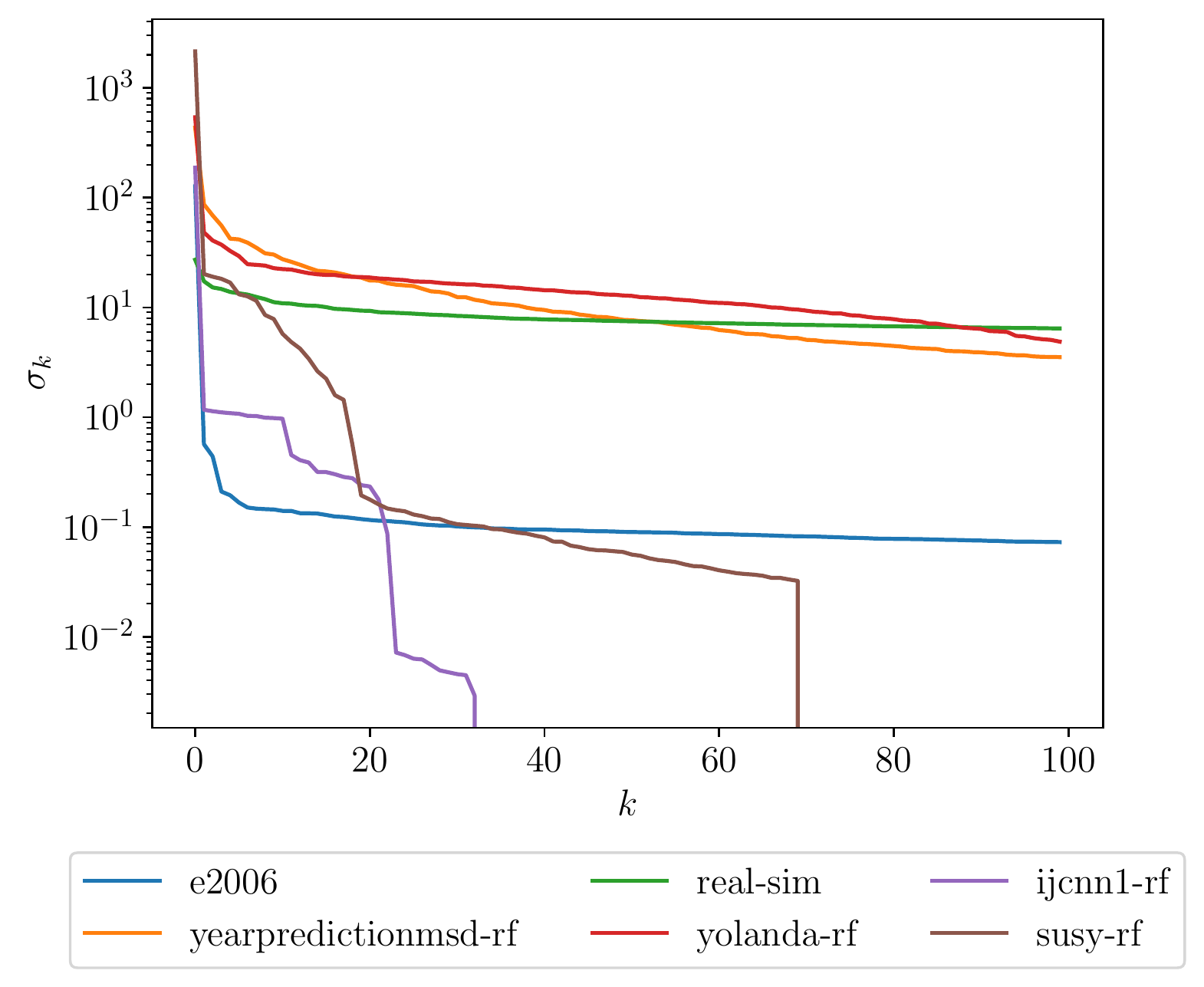}
    \caption{Top $100$ singular values of datasets after preprocessing.}
    \label{fig:spectrums}
\end{figure}

\subsection{Effects of changing the update frequency}
\label{subsection:sensitivity_u}
In this section, we display results only for logistic regression (\Cref{fig:sensitivity_u}), since there is no benefit to updating the preconditioner for a quadratic problem such as ridge regression (\Cref{section:sensitivity_appdx}):
the Hessian in ridge regression is constant for all $w \in \R^p$.
The impact of the update frequency depends on the spectrum of each dataset.
The spectra of ijcnn1-rf and susy-rf are highly concentrated in the top $r = 10$ singular values and decay rapidly (\cref{fig:spectrums}), 
so even the initial preconditioner approximates the curvature of the loss well throughout optimization. 
On the other hand, the spectrum of real-sim decays quite slowly, 
and the initial preconditioner does not capture most of the curvature information in the Hessian.
Hence for real-sim it is beneficial to update the preconditioner, however only infrequent updating is required, as an update frequency of 5 epochs yields almost identical performance to updating every half epoch.
So, increasing the update frequency of the preconditioner past a certain threshold does not improve performance, it just increases the computational cost of the algorithm.
Last, $u = \lceil \frac{n_{\textrm{tr}}}{b_g}\rceil$ exhibits consistent excellent performance across all datasets, supporting the recommendation that it be the default update frequency.

\begin{figure}[htbp]
    \centering
    \includegraphics[scale=0.5]{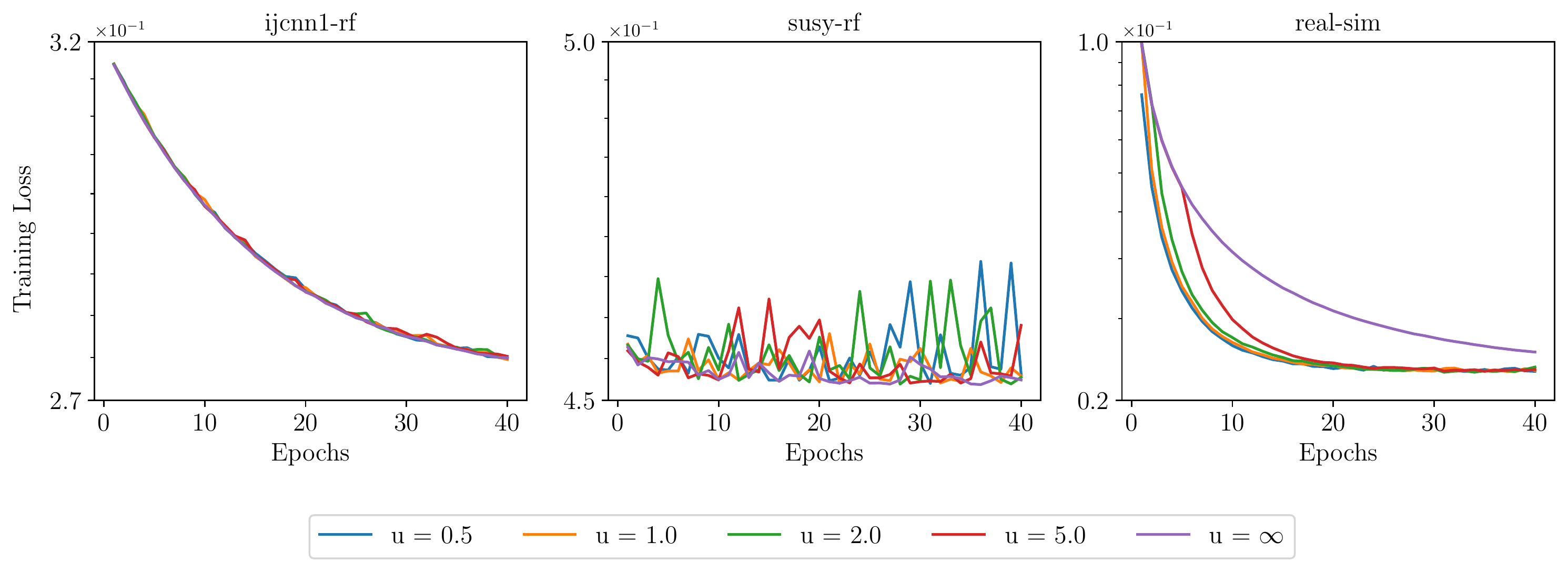}
    \caption{Sensitivity of SketchySGD to update frequency $u$.}
    \label{fig:sensitivity_u}
\end{figure}

\subsection{SketchySGD default learning rate ablation}
\label{subsection:performance_lr_ablation}
It is natural to ask how much of SketchySGD's improved performance relative to SGD stems from preconditioning. 
Indeed, it may be the case that SketchySGD's gains arise from how it sets the learning rate, and not from using approximate second-order information.
To test this, we employ SGD with the same learning rate selection strategy as SketchySGD, but with the preconditioned minibatch Hessian replaced by the minibatch Hessian. 
We refer to this algorithm as \emph{Adaptive SGD} (AdaSGD).

\cref{fig:lr_ablation} shows the results of AdaSGD and SketchySGD on the E2006-tfidf and ijcnn1-rf datasets.
Adaptive SGD performs significantly worse then SketchySGD on these two problems, which shows that SketchySGD's superior performance over SGD  is due to employing preconditioning.
This result is not too surprising.
To see why, let us consider the case of the least-squares loss. 
In this setting, if the subsampled Hessian is representative of the true Hessian, then $\eta_{\textup{AdaSGD}}\approx \bigO(1/L)$. 
Hence when the problem is ill-conditioned, the resulting stepsize will result in poor progress,  
which is precisely what is observed in \cref{fig:lr_ablation}.

\begin{figure}[t]
    \centering
    \begin{subfigure}[b]{0.4\textwidth}
        \centering
        \includegraphics[width=\textwidth]{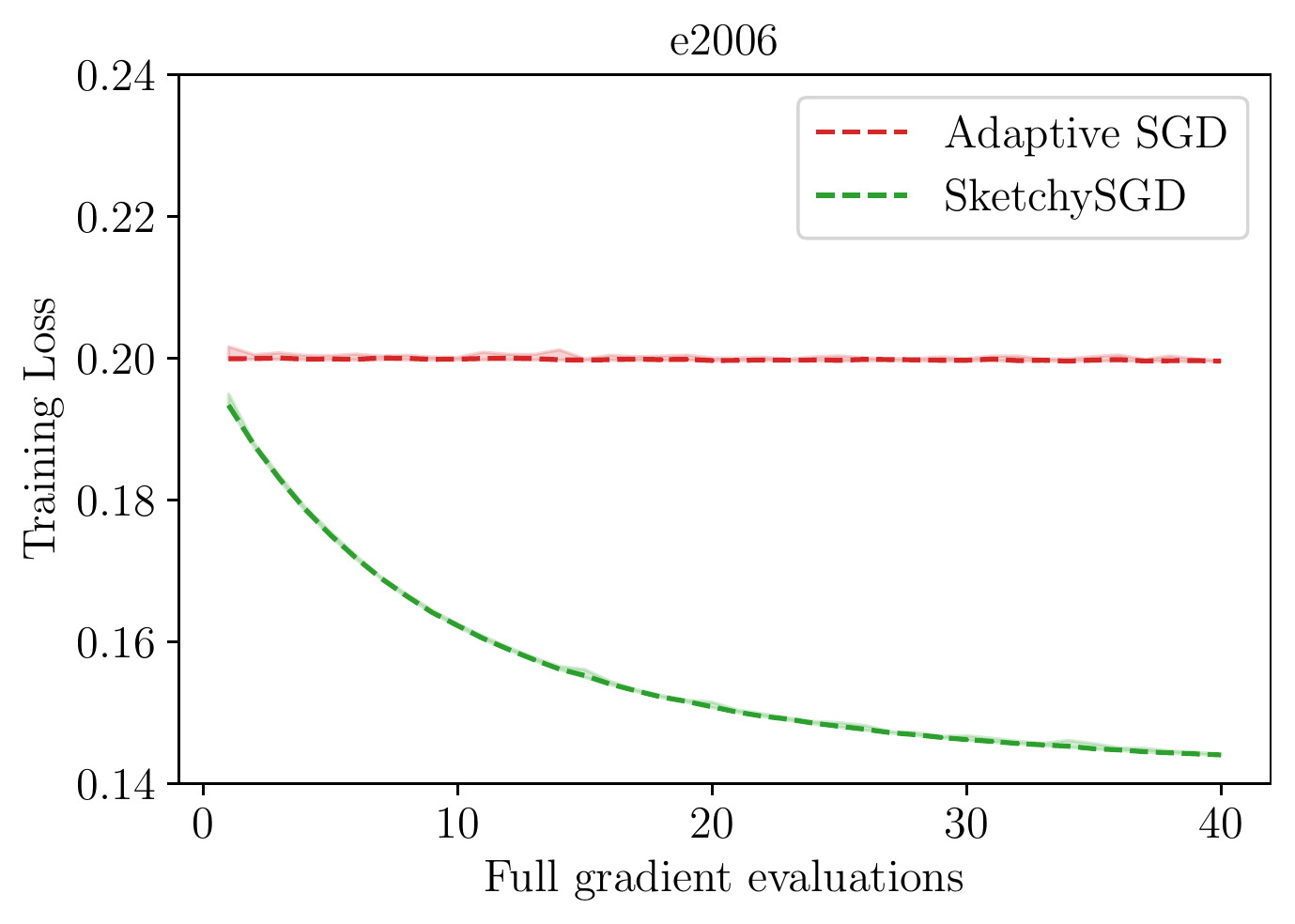}
    \end{subfigure}
    ~
    \begin{subfigure}[b]{0.4\textwidth}
        \centering
        \includegraphics[width=\textwidth]{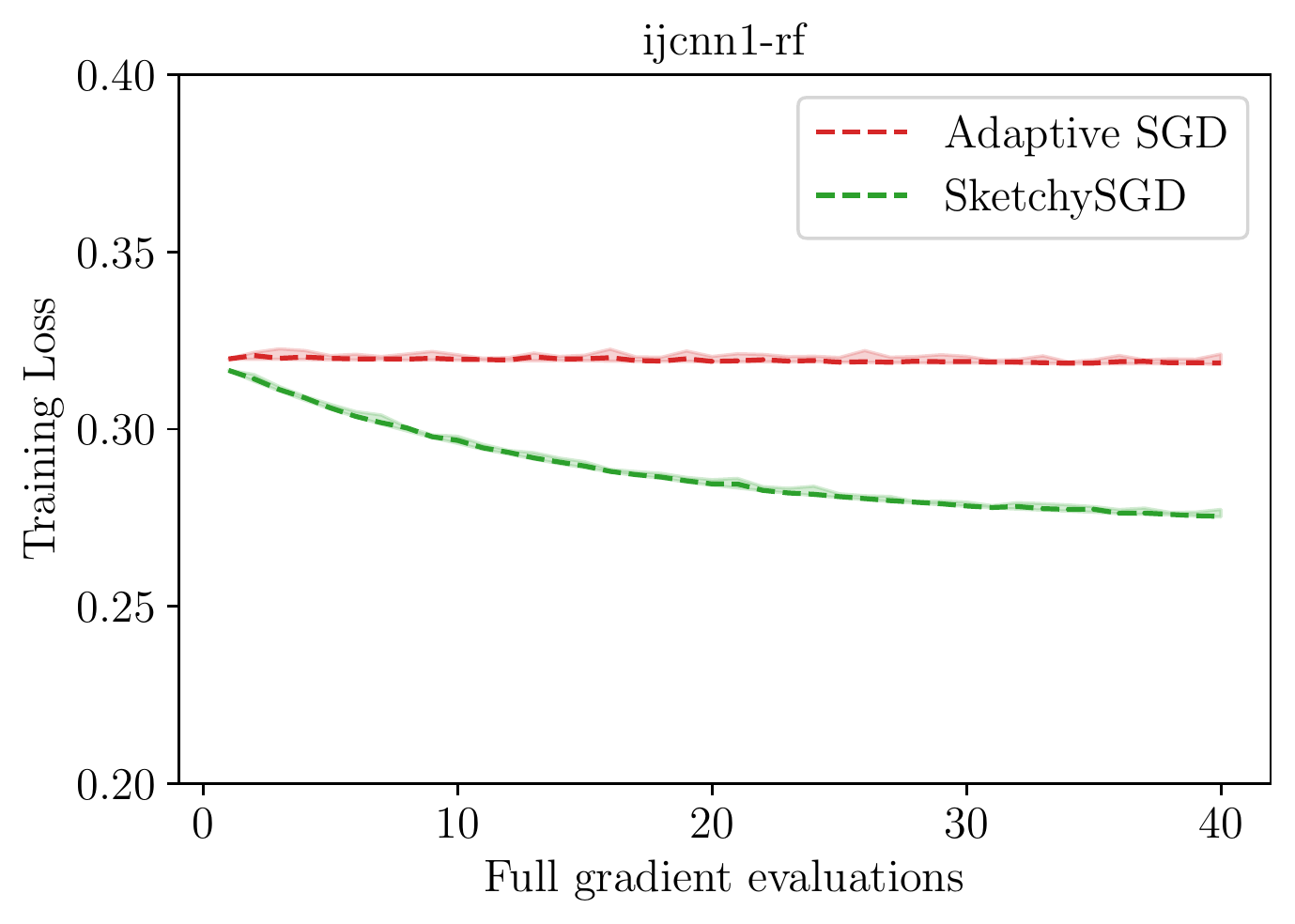}
    \end{subfigure}
    \caption{Results for Adaptive SGD vs. SketchySGD. Adaptive SGD performs much worse then SketchySGD on these two problems, despite employing the same learning rate strategy as SketchySGD. Thus, SketchySGD's improved performance over SGD comes from incorporating preconditioning, and not from how it sets the learning rate.}
    \label{fig:lr_ablation}
\end{figure}

\section{SketchySGD improves the conditioning of the Hessian}
\label{section:hessian_precond_appdx}
In \Cref{subsection:performance_fom}, we showed that SketchySGD generally converges faster than other first-order stochastic optimization methods. 
In this section, we examine the conditioning of the Hessian before and after preconditioning to understand why SketchySGD displays these improvements.

Recall from \Cref{subsection:sketchysgd} that SketchySGD is equivalent to performing SGD in a preconditioned space induced by $P_j = \hat{H}_{S_j} + \rho I$. 
Within this preconditioned space, the Hessian is given by $P_j^{-1/2} H P_j^{-1/2}$, where $H$ is the Hessian in the original space. 
Thus, if $\kappa(P_j^{-1/2} H P_j^{-1/2}) \ll \kappa(H)$, we know that SketchySGD is improving the conditioning of the Hessian, which allows SketchySGD to converge faster.

\Cref{fig:spectrum_logistic_0,fig:spectrum_least_squares_0} display the top $500$ eigenvalues (normalized by the largest eigenvalue) of the Hessian $H$ and the preconditioned Hessian $P_j^{-1/2} H P_j^{-1/2}$ at the initialization of SketchySGD (Nystr\"{o}m) for both logistic and ridge regression.
With the exception of real-sim, SketchySGD (Nystr\"{o}m) improves the conditioning of the Hessian by several orders of magnitude. 
This improved conditioning aligns with the improved convergence that is observed on the ijcnn1-rf, susy-rf, E2006-tfidf, YearPredictionMSD-rf, and yolanda-rf datasets in \Cref{subsection:performance_fom}. 

\begin{figure}[htbp]
    \begin{subfigure}[b]{\textwidth}
    \centering
    \includegraphics[width=\textwidth]{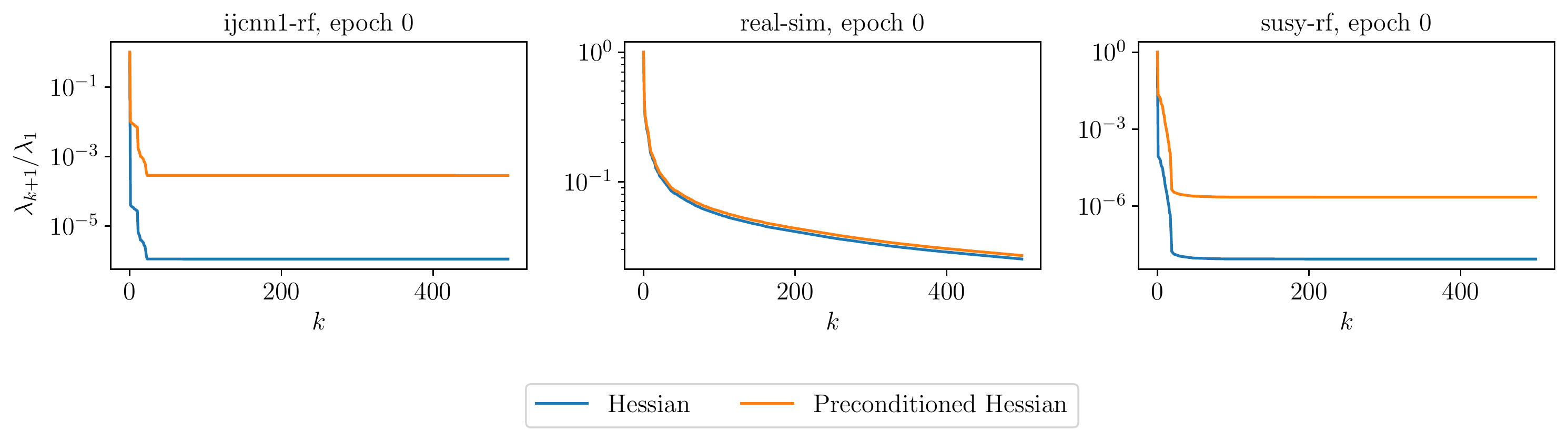}
    \end{subfigure}
    
    \begin{subfigure}{\textwidth}
     \includegraphics[width=\textwidth]{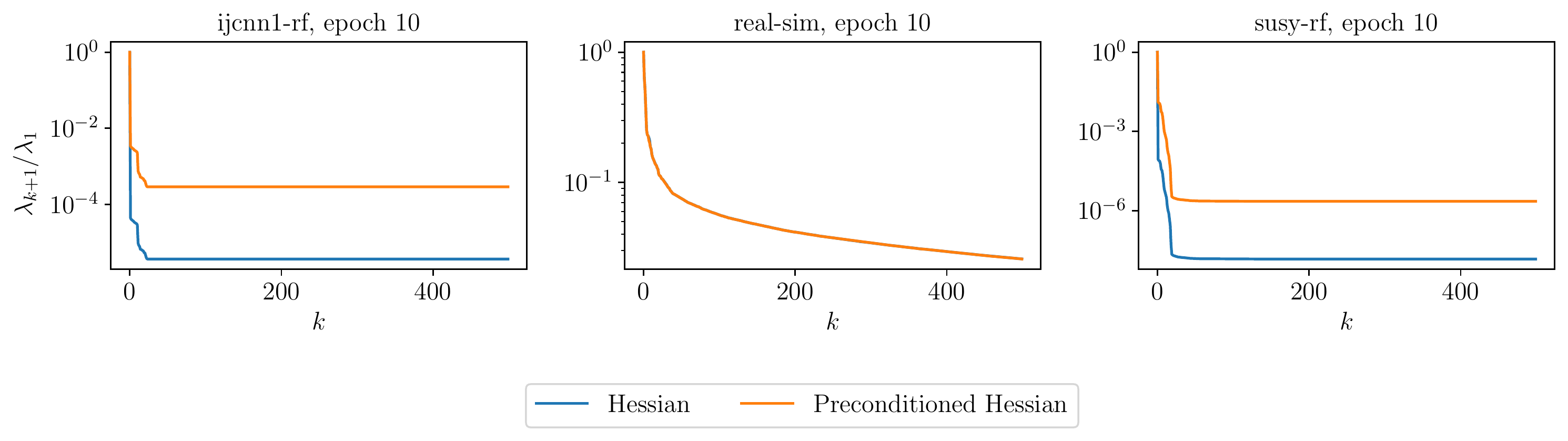} 
    \end{subfigure}
    
    \begin{subfigure}{\textwidth}
    \centering
       \includegraphics[width=\textwidth]{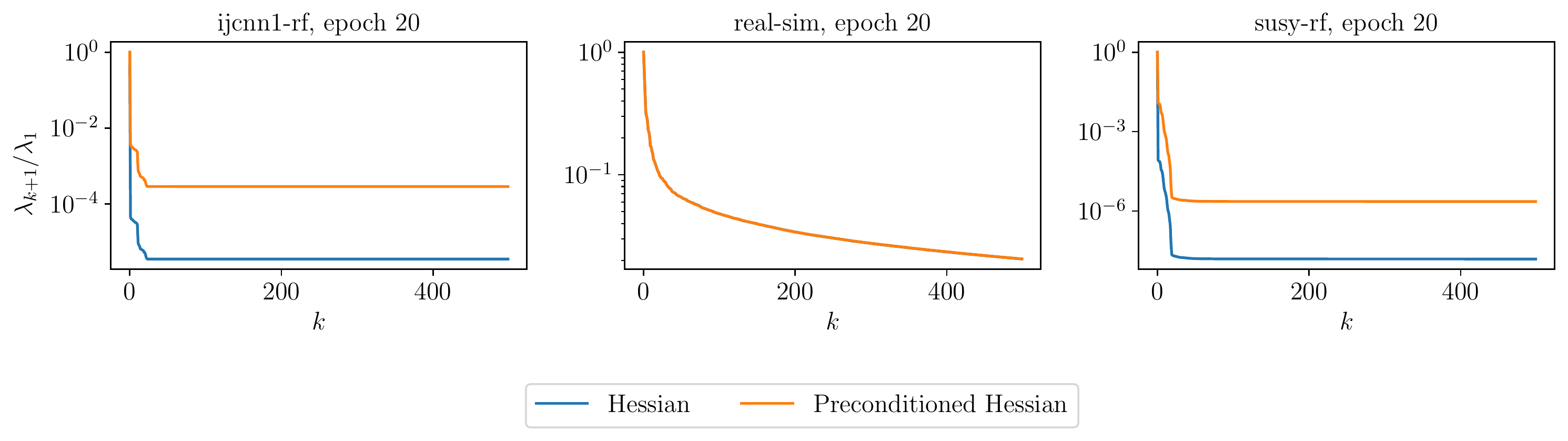}  
    \end{subfigure}
   
    \begin{subfigure}{\textwidth}
    \centering
       \includegraphics[width=\textwidth]{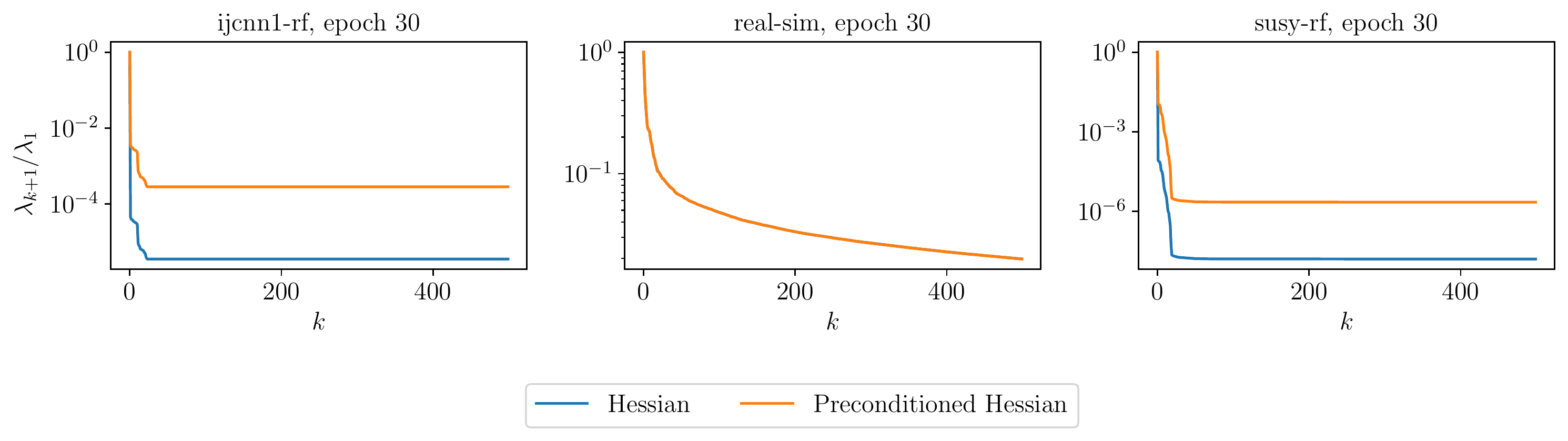}  
    \end{subfigure}
    \caption{Spectrum of the Hessian at epochs $0,10,20,30$  before and after preconditioning in $l_2$-regularized logistic regression.}
    \label{fig:spectrum_logistic_0}
\end{figure}

\begin{figure}[htbp]
    \centering
    \includegraphics[scale=0.5]{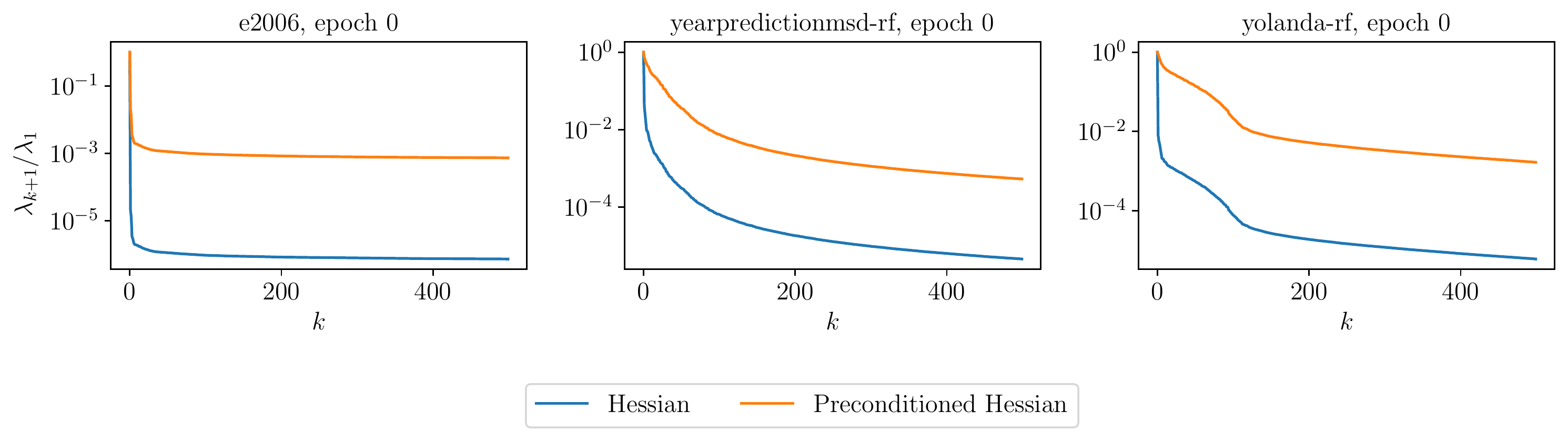}
    \caption{Normalized spectrum of the Hessian before and after preconditioning in ridge regression.}
    \label{fig:spectrum_least_squares_0}
\end{figure}




\section{Scaling experiments}
\subsection{Quasi-Newton}
\label{subsection:scaling_qn}
For larger datasets, we expect the performance gap between SketchySGD and the selected quasi-Newton methods to grow even larger, since these methods require full-gradient computations.
We increase the number of samples for each dataset in \Cref{table:datasets} (with the exception of susy-rf) by a factor of 3 using data augmentation with Gaussian random noise, i.e., a dataset of size $n_{\mathrm{tr}} \times p$ now has size $3 n_{\mathrm{tr}} \times p$.
The results are shown in \cref{fig:performance_som_logistic_scaled,fig:performance_som_least_squares_scaled}.
When looking at performance with respect to wall-clock time, SketchySGD is outperformed less often by the quasi-Newton methods; it is only outperformed by SLBFGS (before it diverges) on ijcnn1-rf and RSN on yolanda-rf.
Again, SketchySGD performs much better than the competition on YearPredictionMSD-rf, which is a larger, dense dataset.

\begin{figure}[t]
    \centering
    \includegraphics[scale=0.5]{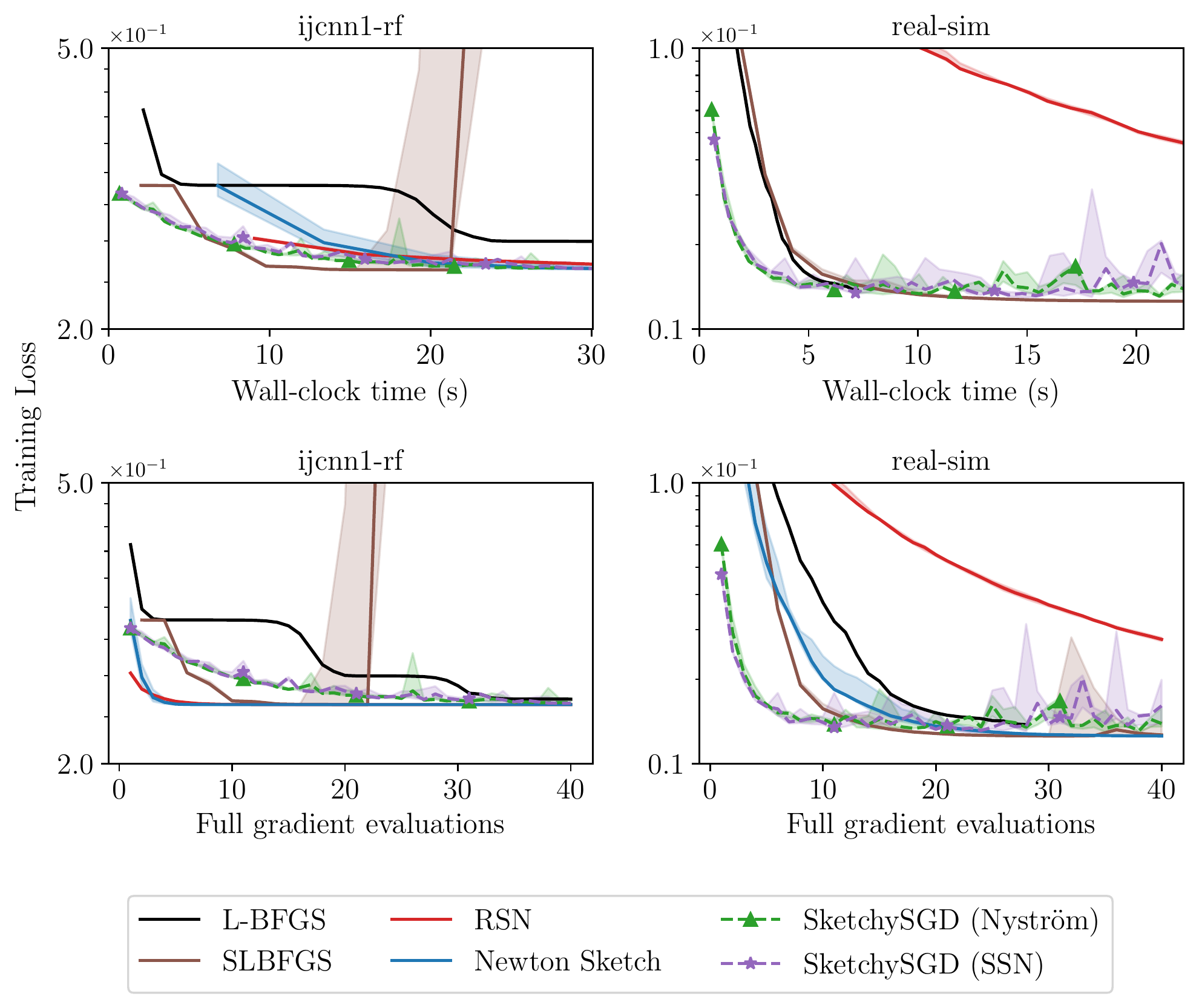}
    \caption{Comparisons to quasi-Newton methods (L-BFGS, SLBFGS, RSN, Newton Sketch) on $l_2$-regularized logistic regression with augmented datasets.}
    \label{fig:performance_som_logistic_scaled}
\end{figure}

\begin{figure}[t]
    \centering
    \includegraphics[scale=0.5]{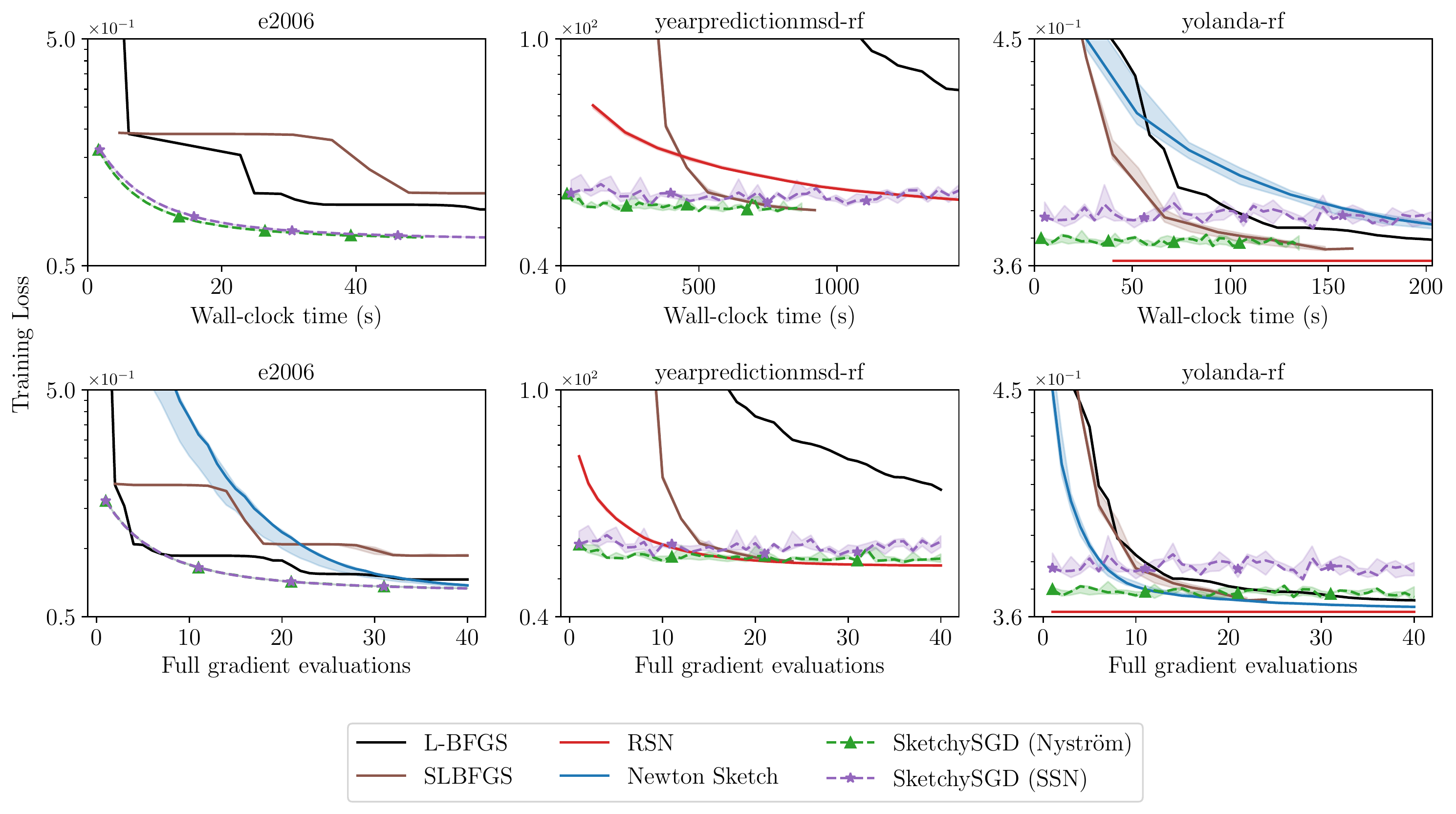}
    \caption{Comparisons to quasi-Newton methods (L-BFGS, SLBFGS, RSN, Newton Sketch) on ridge regression with augmented datasets.}
    \label{fig:performance_som_least_squares_scaled}
\end{figure}

\subsection{PCG}
\label{subsection:scaling_pcg}
The main costs of PCG are generally in (i) computing the preconditioner and (ii) performing matrix-vector products with the data matrix. 
For larger datasets, we expect both of these costs to increase, which should close the performance gap between SketchySGD and PCG.
We increase the number of samples for each ridge regression dataset by a factor of 3 as in \Cref{subsection:performance_som}; the results on these augmented datasets are presented in \Cref{fig:performance_pcg_least_squares_scaled}.
We see that the performance gap closes slightly --- on E2006-tfidf, SketchySGD now performs comparably to JacobiPCG.
In addition, the PCG methods now take significantly more wall-clock time to reach the training loss attained by SketchySGD.

\begin{figure}[t]
    \centering
    \includegraphics[scale=0.5]{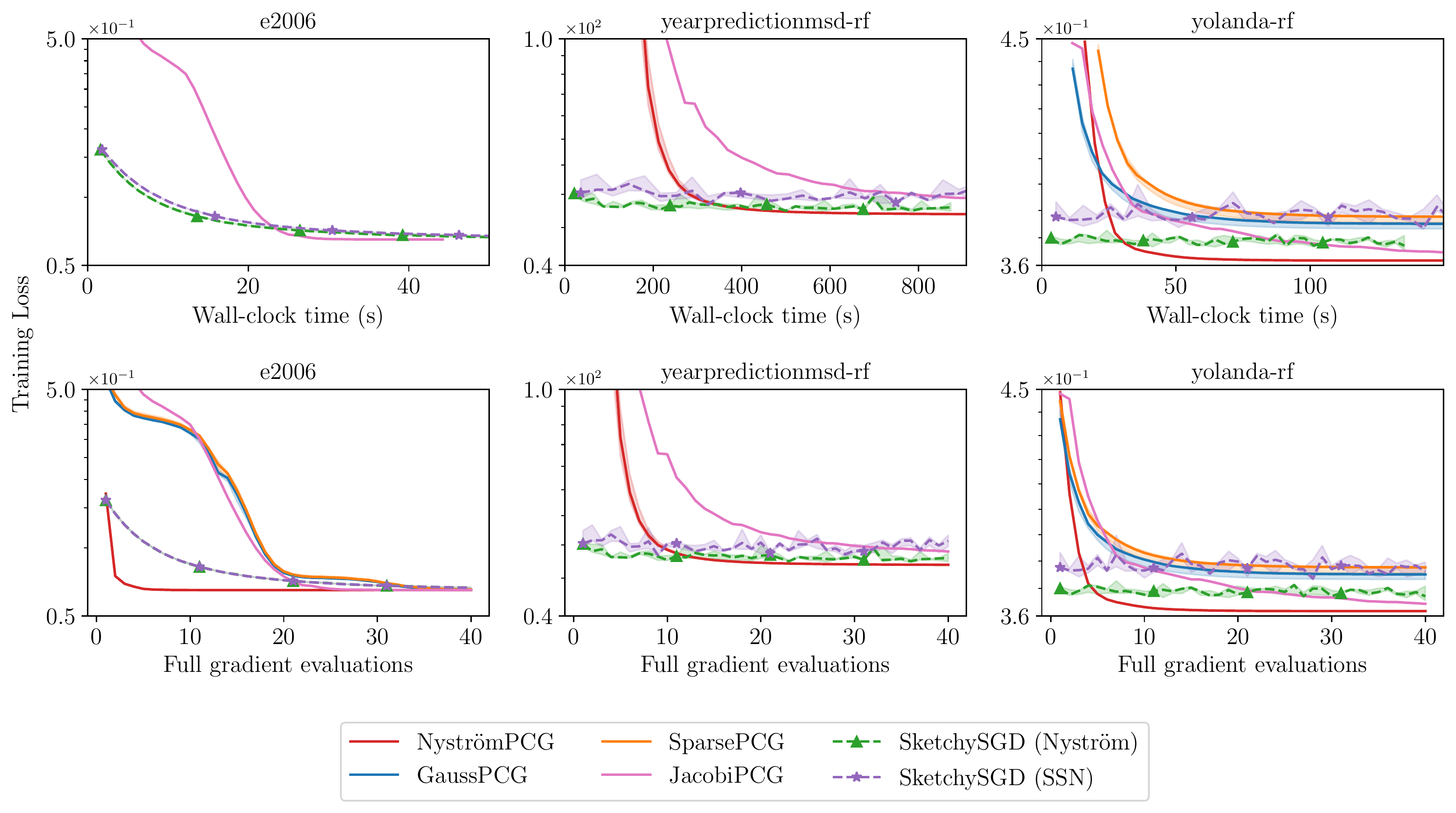}
    \caption{Comparisons to PCG (Jacobi, Nystr\"{o}m, sketch-and-precondition w/ Gaussian and sparse embeddings) on ridge regression with augmented datasets.}
    \label{fig:performance_pcg_least_squares_scaled}
\end{figure}

\section{Comparisons omitted from the main paper}
\label{section:loss_acc_appdx}
We present training/test loss and training/test accuracy comparisons that were omitted from \Cref{subsubsection:performance_auto,subsubsection:performance_tuned,subsection:performance_som,subsection:performance_pcg,subsection:large_scale,subsection:performance_dl}. 

\paragraph{Comparison to first-order methods with defaults} \Cref{fig:performance_auto_logistic_test_loss,fig:performance_auto_least_squares_test_loss,fig:performance_auto_logistic_train_acc,fig:performance_auto_logistic_test_acc} display test loss/training accuracy/test accuracy plots corresponding to \Cref{subsubsection:performance_auto}, where we compare SketchySGD to first-order methods with their default hyperparameters.

\begin{figure}[htbp]
    \centering
    \includegraphics[scale=0.5]{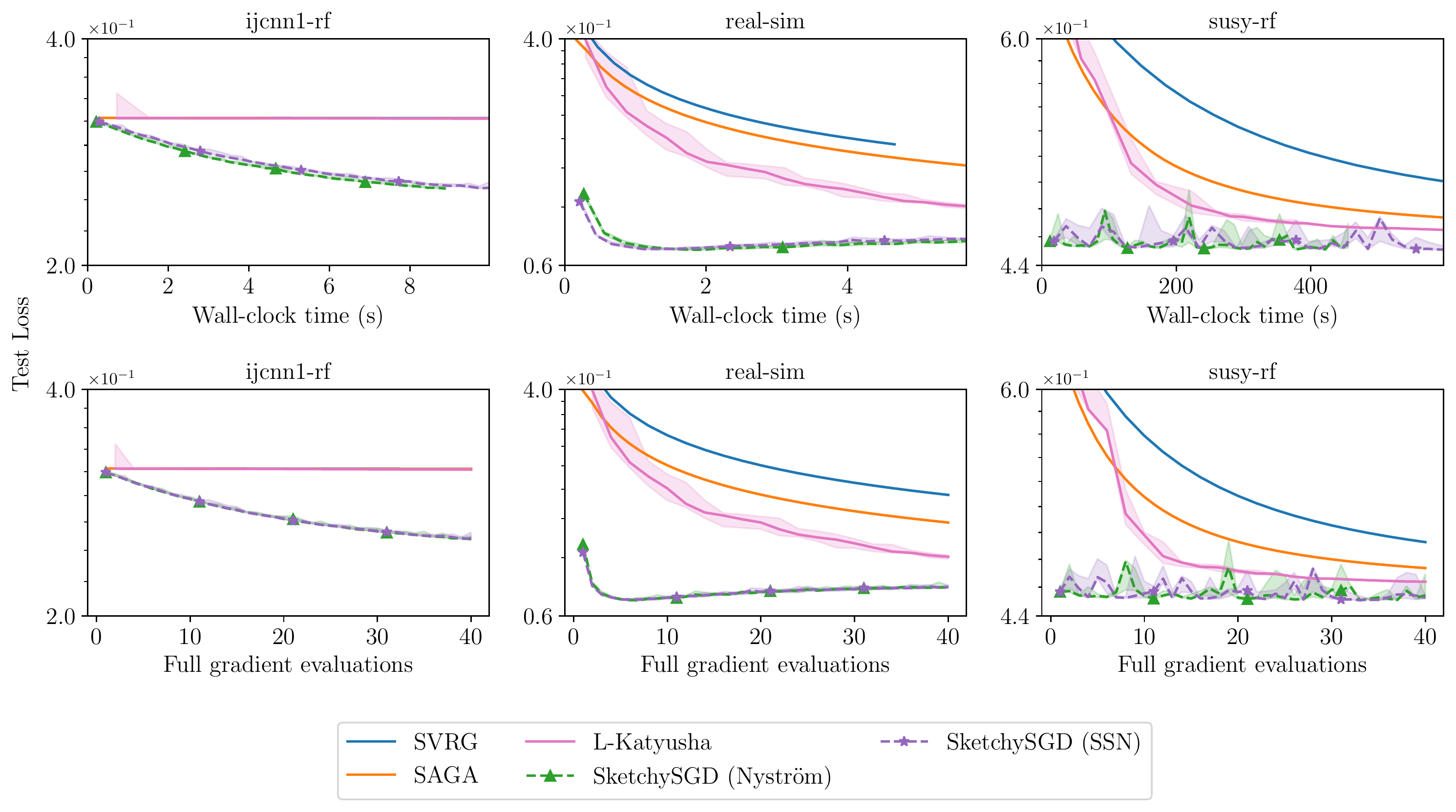}
    \caption{Test losses of first-order methods with default learning rates (SVRG, SAGA) and smoothness parameters (L-Katyusha) on $l_2$-regularized logistic regression.}
    \label{fig:performance_auto_logistic_test_loss}
\end{figure}

\begin{figure}[htbp]
    \centering
    \includegraphics[scale=0.5]{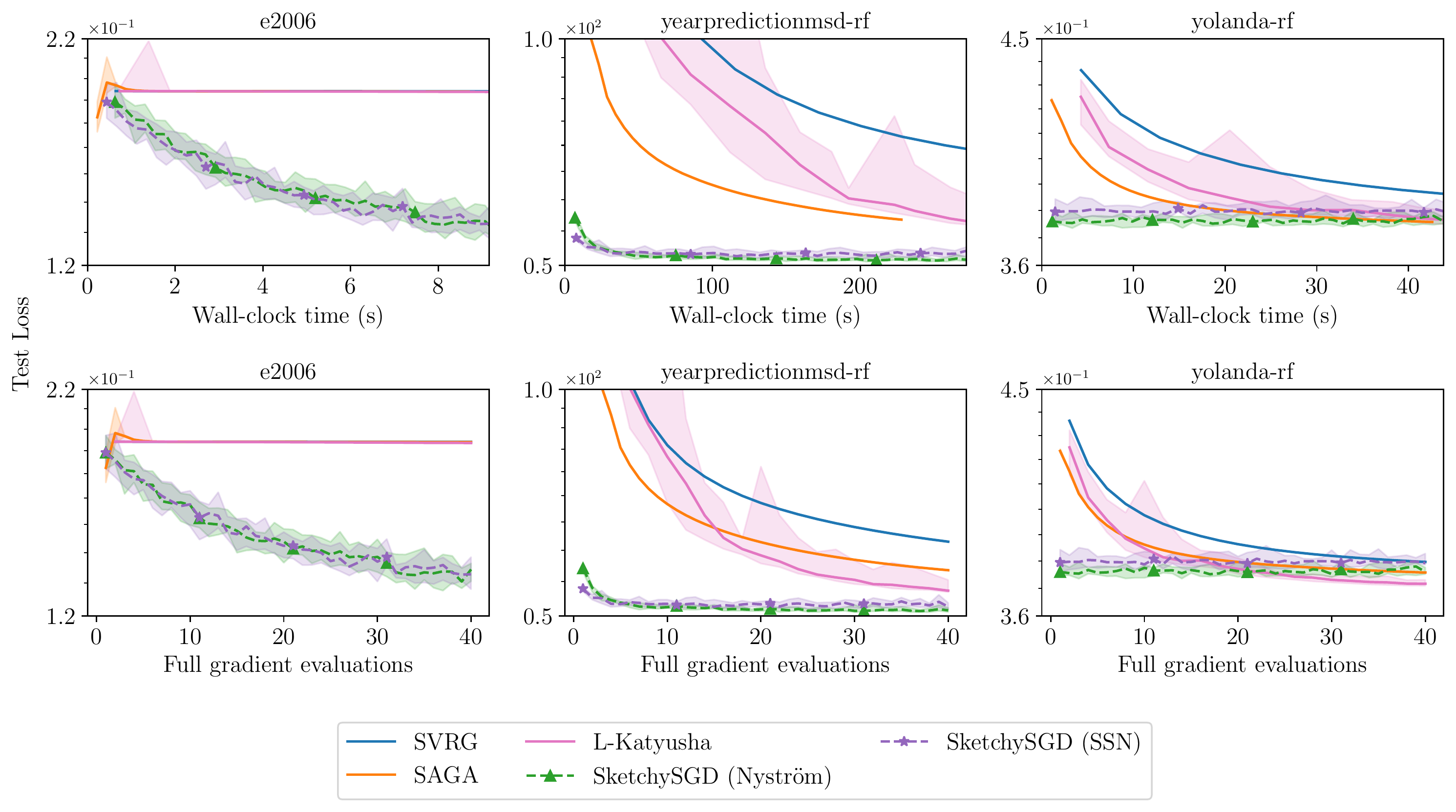}
    \caption{Test losses of first-order methods with default learning rates (SVRG, SAGA) and smoothness parameters (L-Katyusha) on ridge regression.}
    \label{fig:performance_auto_least_squares_test_loss}
\end{figure}

\begin{figure}[htbp]
    \centering
    \includegraphics[scale=0.5]{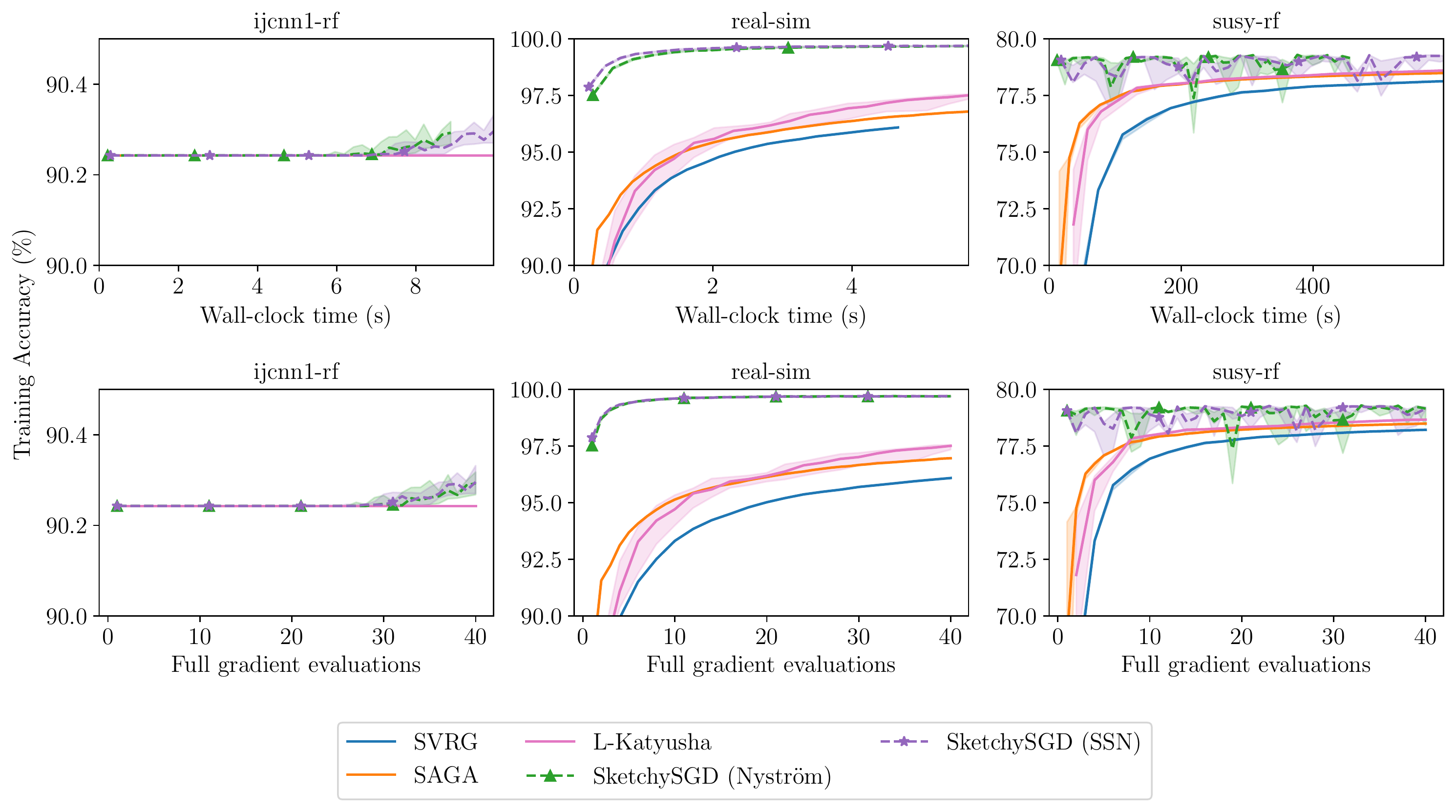}
    \caption{Training accuracies of first-order methods with default learning rates (SVRG, SAGA) and smoothness parameters (L-Katyusha) on $l_2$-regularized logistic regression.}
    \label{fig:performance_auto_logistic_train_acc}
\end{figure}

\begin{figure}[htbp]
    \centering
    \includegraphics[scale=0.5]{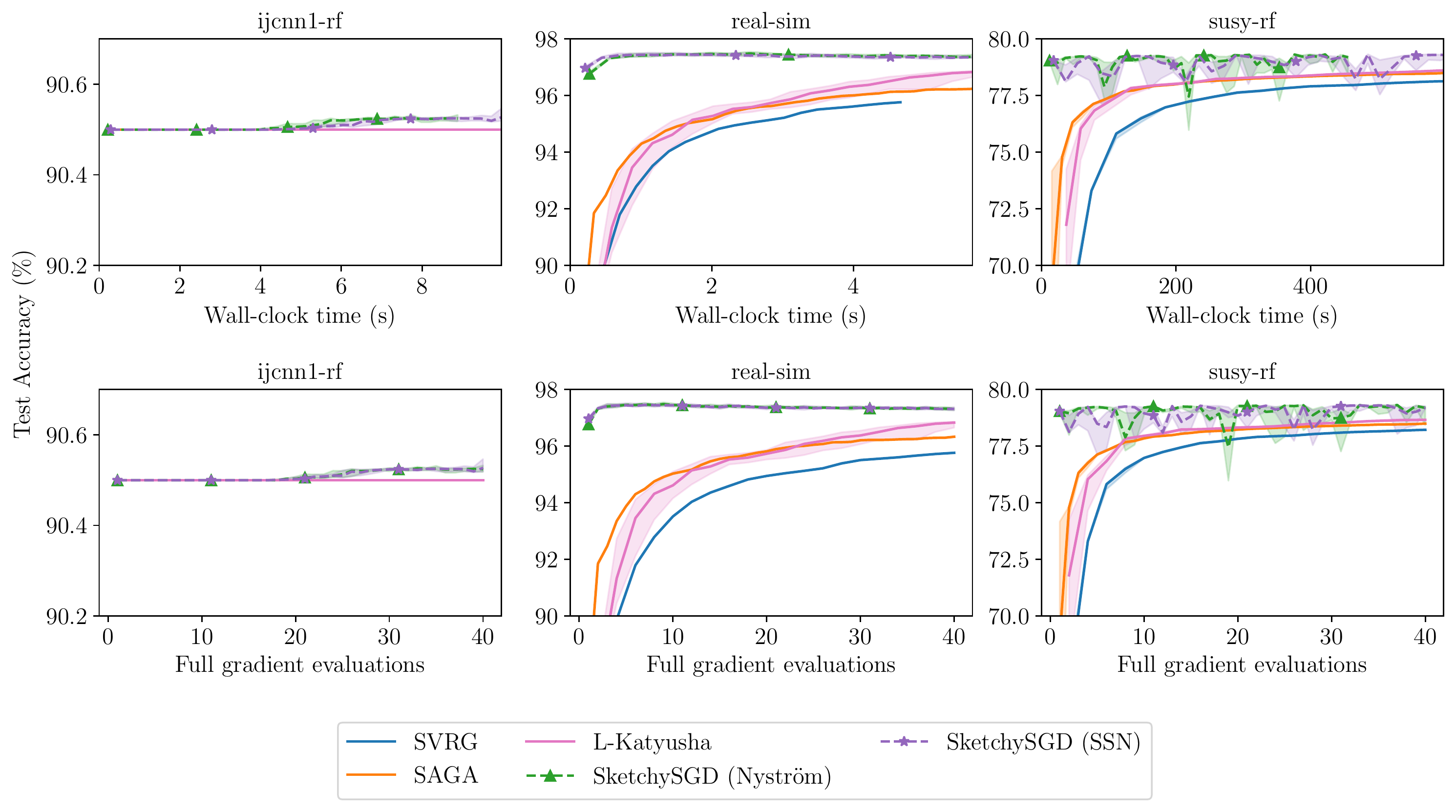}
    \caption{Test accuracies of first-order methods with default learning rates (SVRG, SAGA) and smoothness parameters (L-Katyusha) on $l_2$-regularized logistic regression.}
    \label{fig:performance_auto_logistic_test_acc}
\end{figure}

\paragraph{Comparisons to first-order methods with tuning} 
\Cref{fig:performance_tuned_logistic_test_loss,fig:performance_tuned_least_squares_test_loss,fig:performance_tuned_logistic_train_acc,fig:performance_tuned_logistic_test_acc} display test loss/training accuracy/test accuracy plots corresponding to \Cref{subsubsection:performance_tuned}, where we compare SketchySGD to first-order methods with tuned hyperparameters. 

\begin{figure}[htbp]
    \centering
    \includegraphics[scale=0.5]{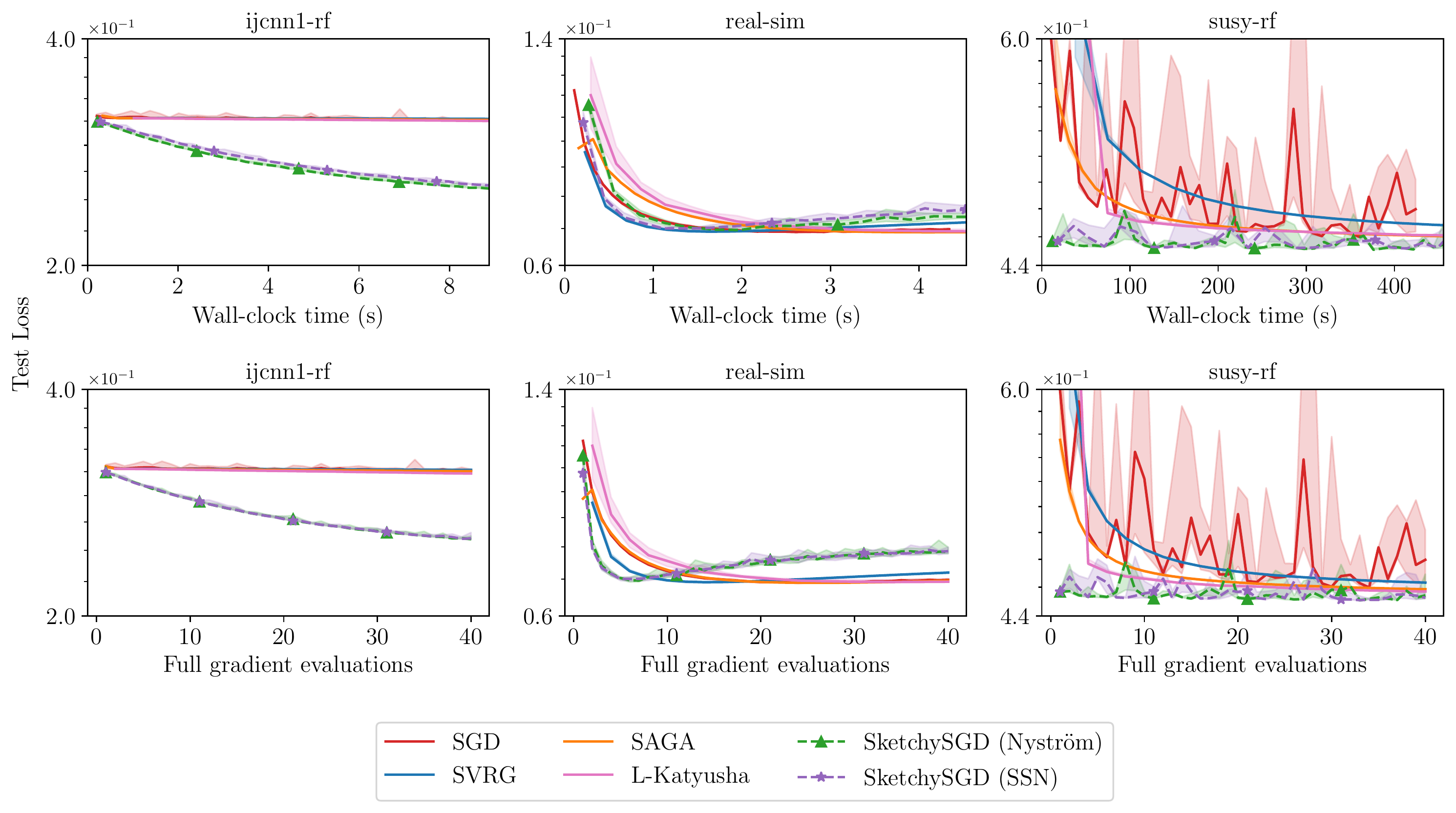}
    \caption{Test losses of first-order methods with tuned learning rates (SGD, SVRG, SAGA) and smoothness parameters (L-Katyusha) on $l_2$-regularized logistic regression.}
    \label{fig:performance_tuned_logistic_test_loss}
\end{figure}

\begin{figure}[htbp]
    \centering
    \includegraphics[scale=0.5]{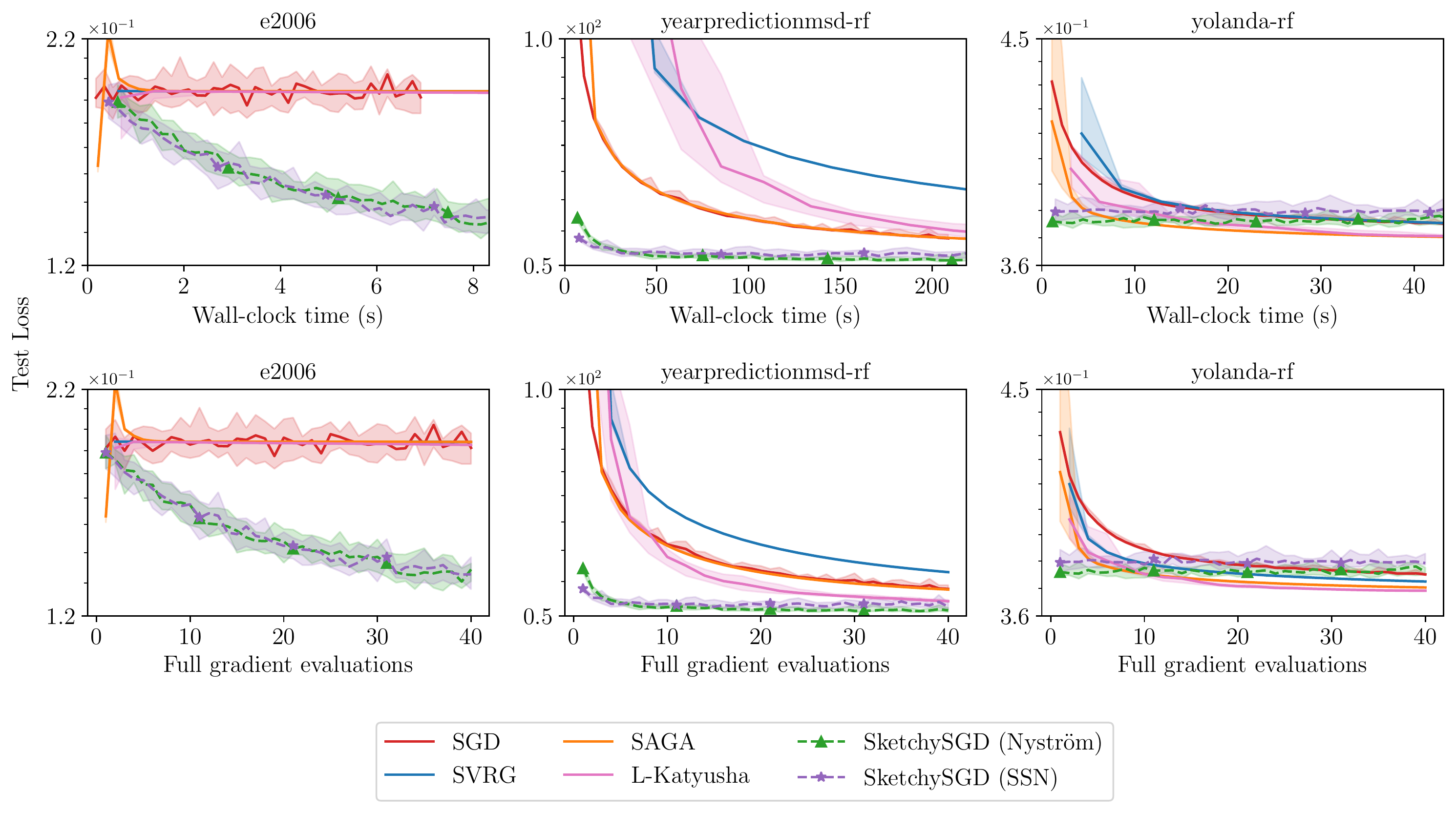}
    \caption{Test losses of first-order  methods with tuned learning rates (SGD, SVRG, SAGA) and smoothness parameters (L-Katyusha) on  ridge regression.}
    \label{fig:performance_tuned_least_squares_test_loss}
\end{figure}

\begin{figure}[htbp]
    \centering
    \includegraphics[scale=0.5]{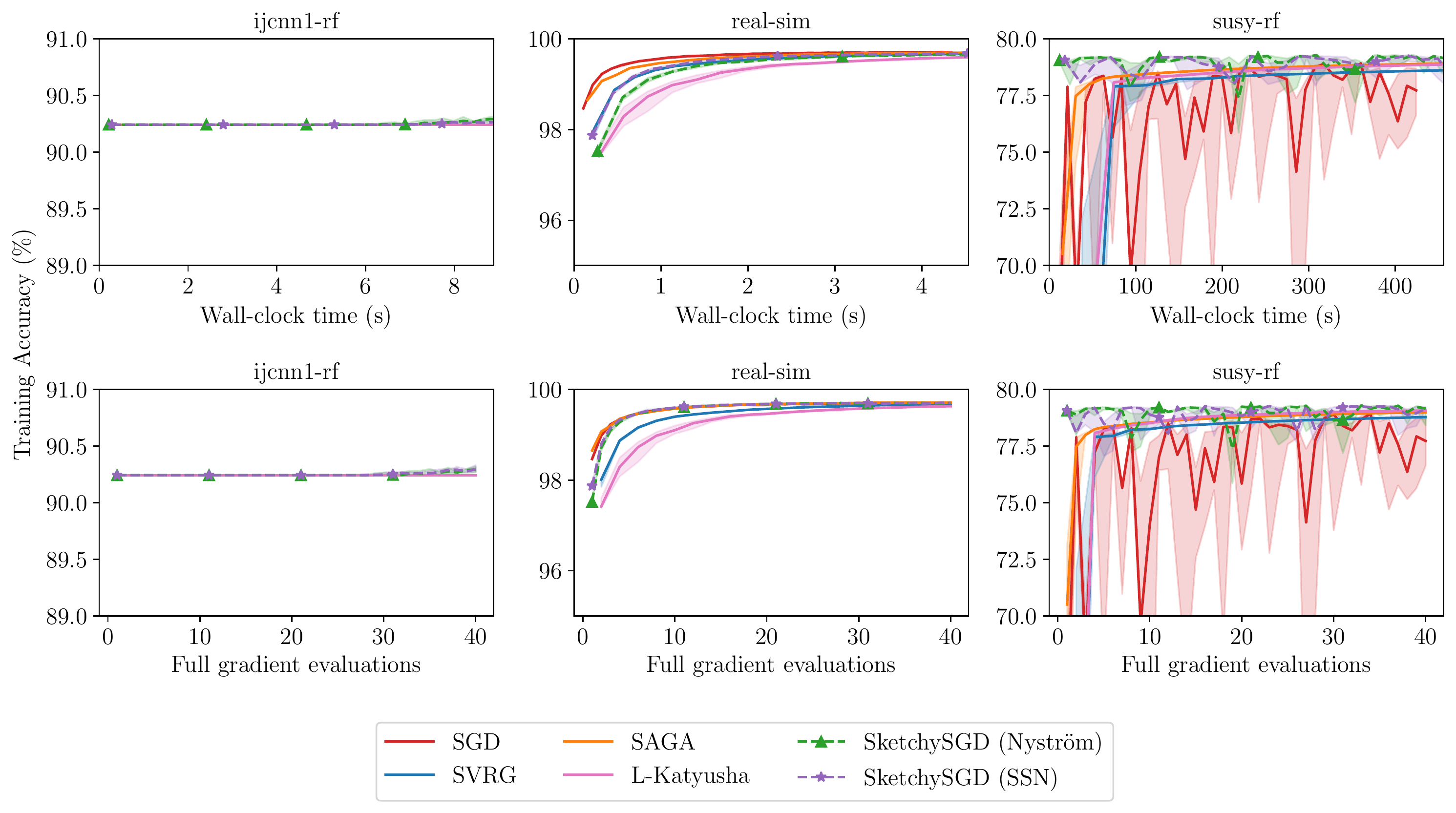}
    \caption{Training accuracies of first-order methods with tuned learning rates (SGD, SVRG, SAGA) and smoothness parameters (L-Katyusha) on $l_2$-regularized logistic regression.}
    \label{fig:performance_tuned_logistic_train_acc}
\end{figure}

\begin{figure}[htbp]
    \centering
    \includegraphics[scale=0.5]{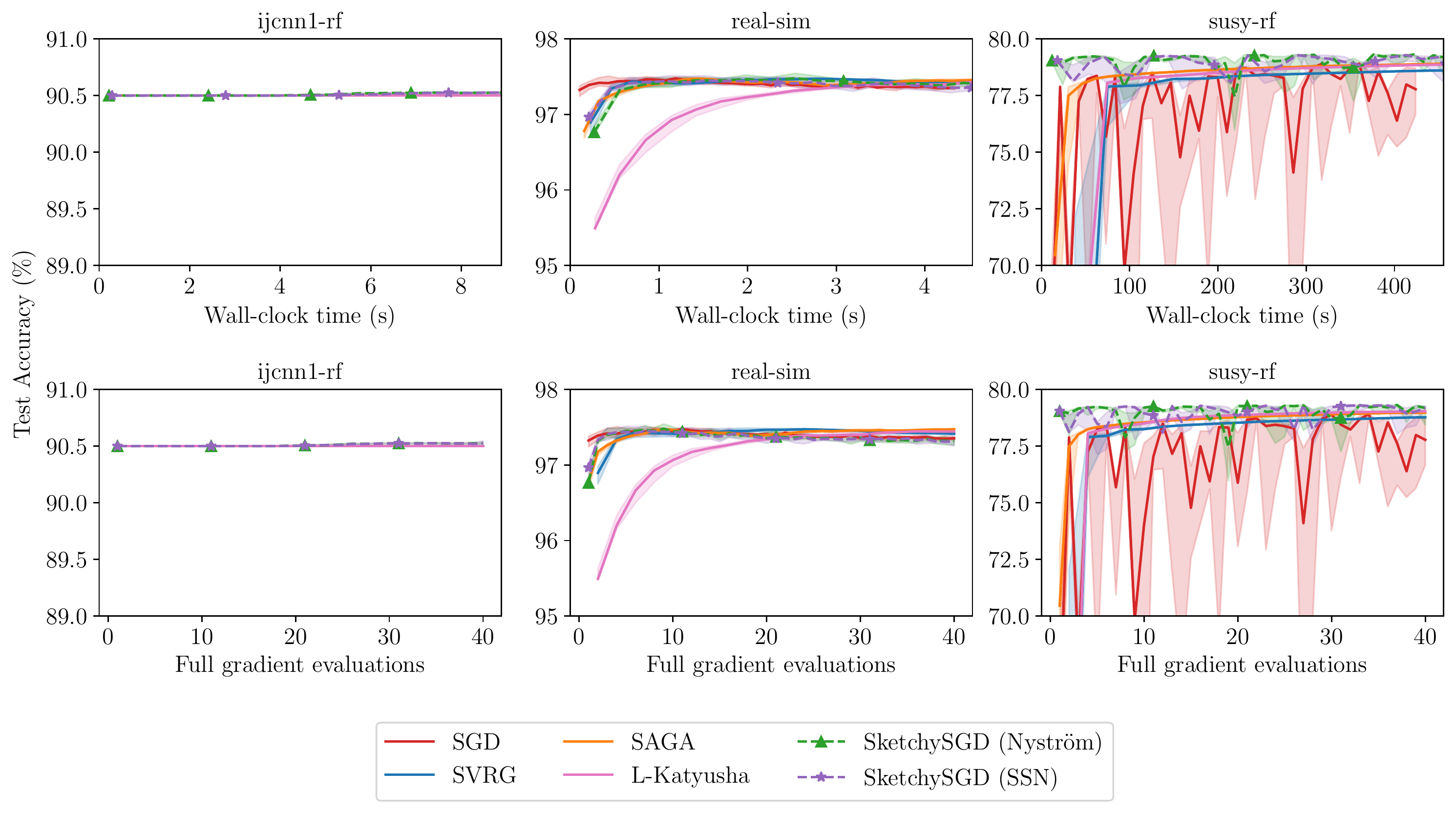}
    \caption{Test accuracies of first-order methods with tuned learning rates (SGD, SVRG, SAGA) and smoothness parameters (L-Katyusha) on $l_2$-regularized logistic regression.}
    \label{fig:performance_tuned_logistic_test_acc}
\end{figure}

\paragraph{Comparisons to quasi-Newton methods}
\Cref{fig:performance_som_logistic_test_loss,fig:performance_som_least_squares_test_loss,fig:performance_som_logistic_train_acc,fig:performance_som_logistic_test_acc,fig:performance_som_logistic_scaled_test_loss,fig:performance_som_least_squares_scaled_test_loss,fig:performance_som_logistic_scaled_train_acc,fig:performance_som_logistic_scaled_test_acc} display test loss/training accuracy/test accuracy plots corresponding to \Cref{subsection:performance_som}, where we compare SketchySGD to quasi-Newton methods.

\begin{figure}[htbp]
    \centering
    \includegraphics[scale=0.5]{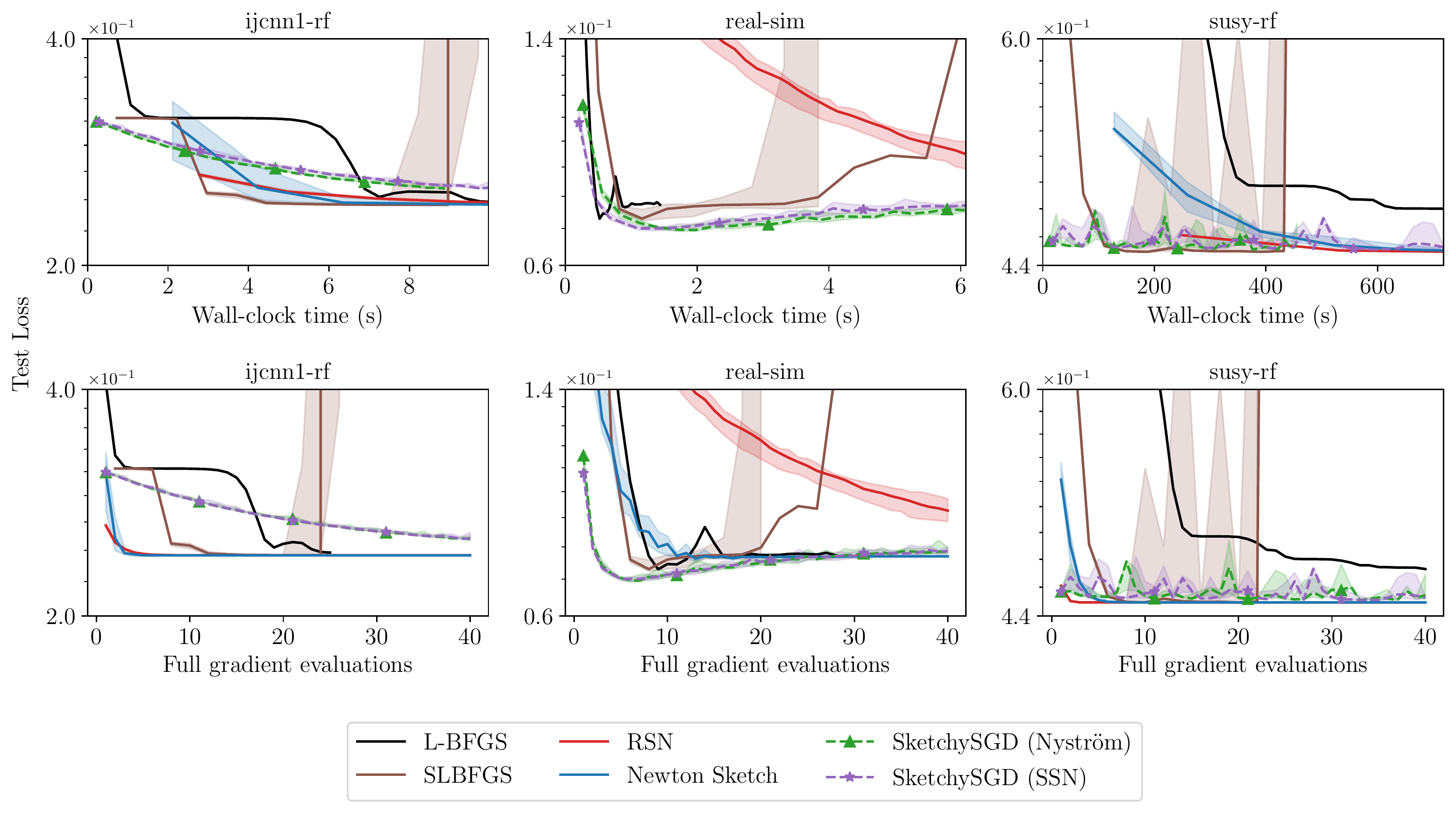}
    \caption{Test losses of quasi-Newton methods (L-BFGS, SLBFGS, RSN, Newton Sketch) on $l_2$-regularized logistic regression.}
    \label{fig:performance_som_logistic_test_loss}
\end{figure}

\begin{figure}[htbp]
    \centering
    \includegraphics[scale=0.5]{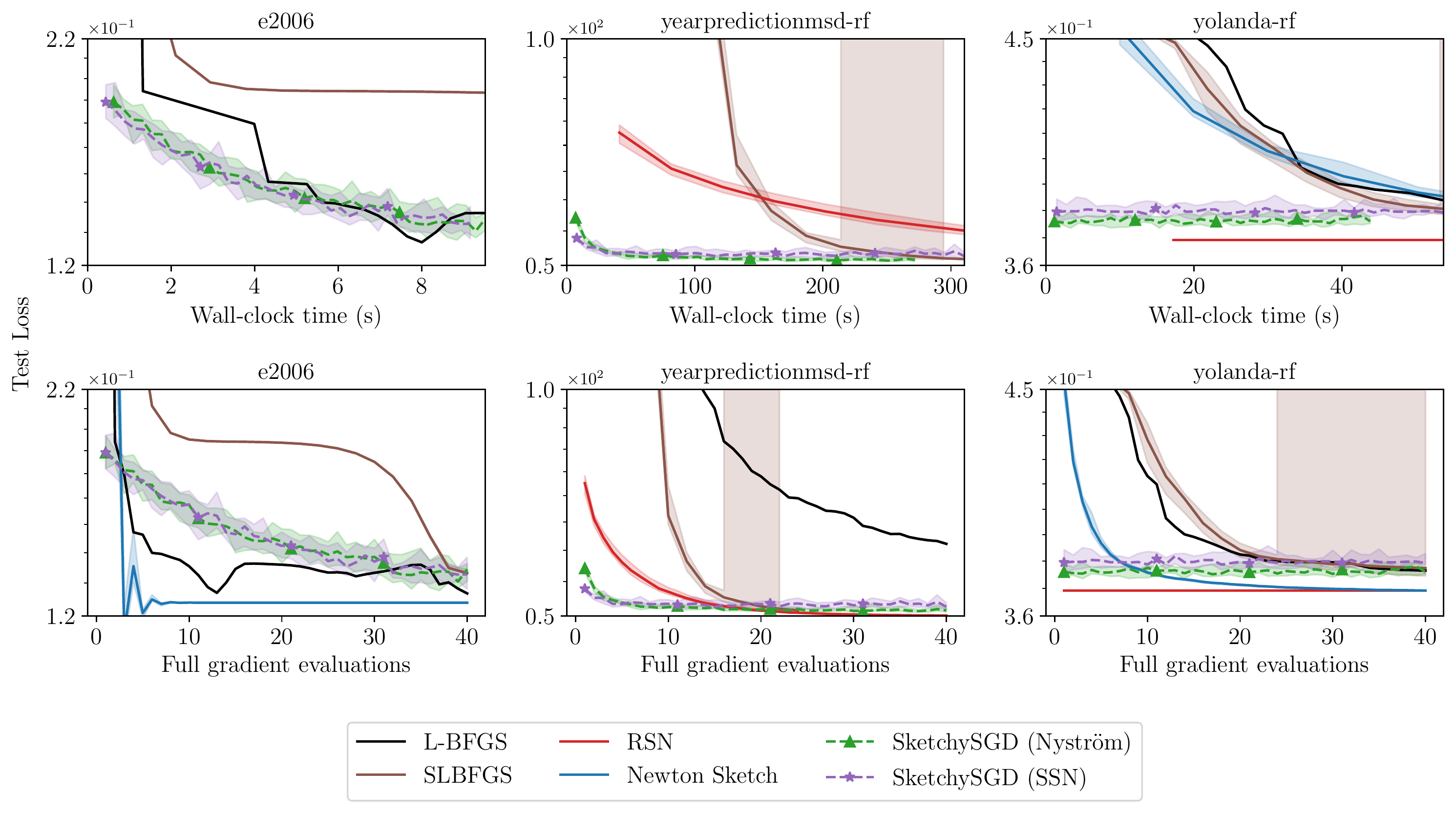}
    \caption{Test losses of quasi-Newton methods (L-BFGS, SLBFGS, RSN, Newton Sketch) on ridge regression.}
    \label{fig:performance_som_least_squares_test_loss}
\end{figure}

\begin{figure}[htbp]
    \centering
    \includegraphics[scale=0.5]{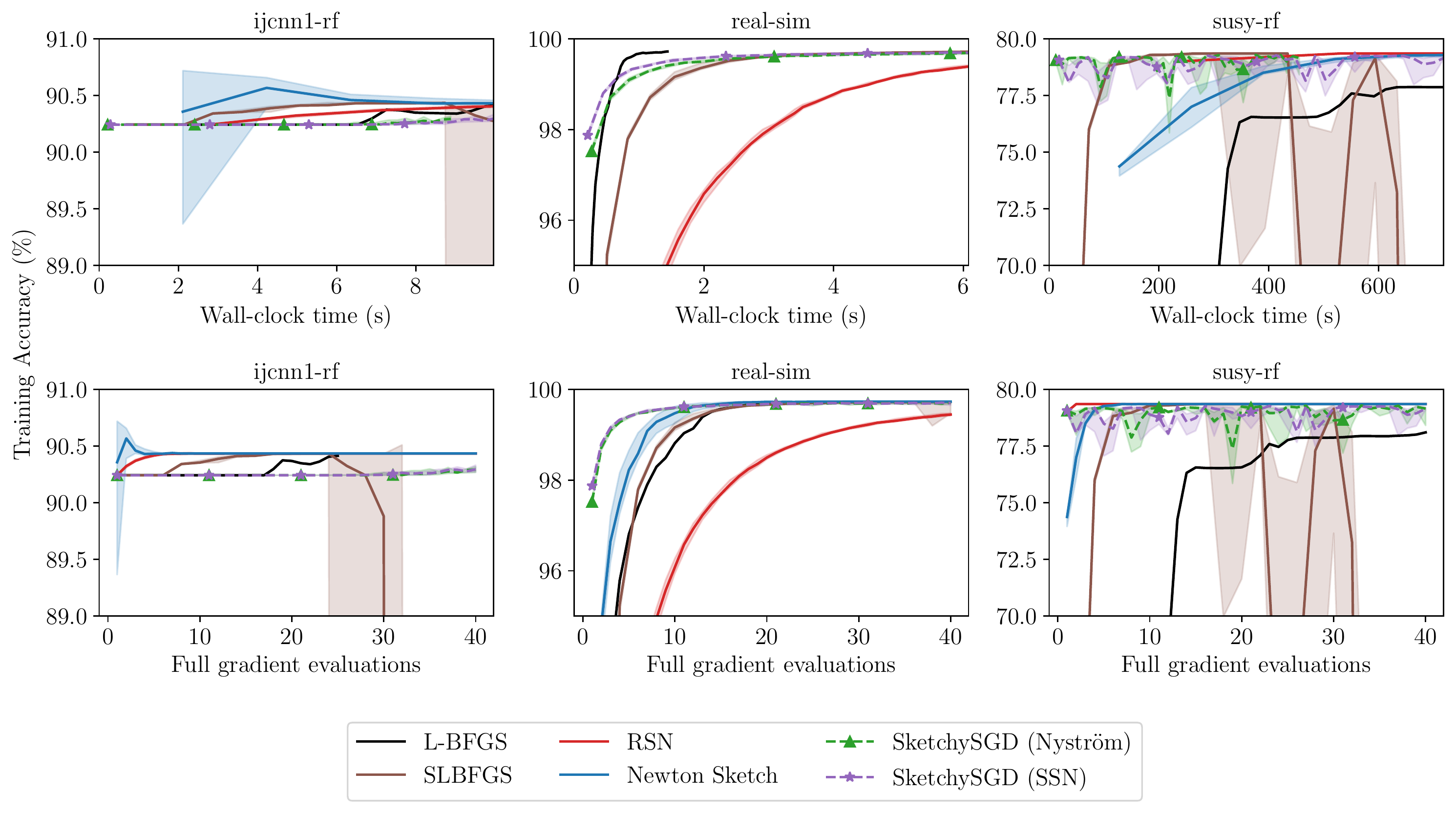}
    \caption{Training accuracies of quasi-Newton methods (L-BFGS, SLBFGS, RSN, Newton Sketch) on $l_2$-regularized logistic regression.}
    \label{fig:performance_som_logistic_train_acc}
\end{figure}

\begin{figure}[htbp]
    \centering
    \includegraphics[scale=0.5]{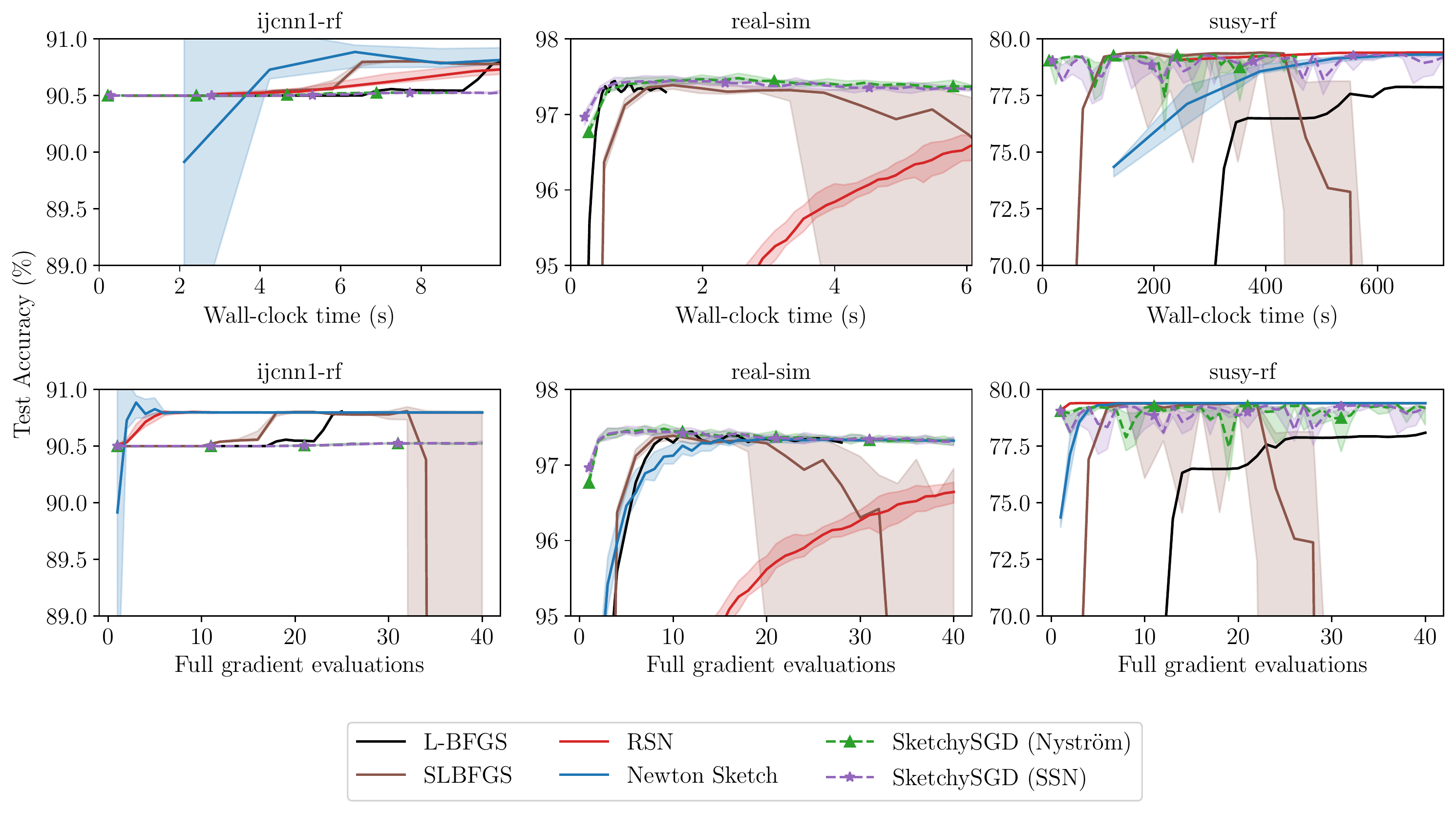}
    \caption{Test accuracies of quasi-Newton methods (L-BFGS, SLBFGS, RSN, Newton Sketch) on $l_2$-regularized logistic regression.}
    \label{fig:performance_som_logistic_test_acc}
\end{figure}

\begin{figure}[htbp]
    \centering
    \includegraphics[scale=0.5]{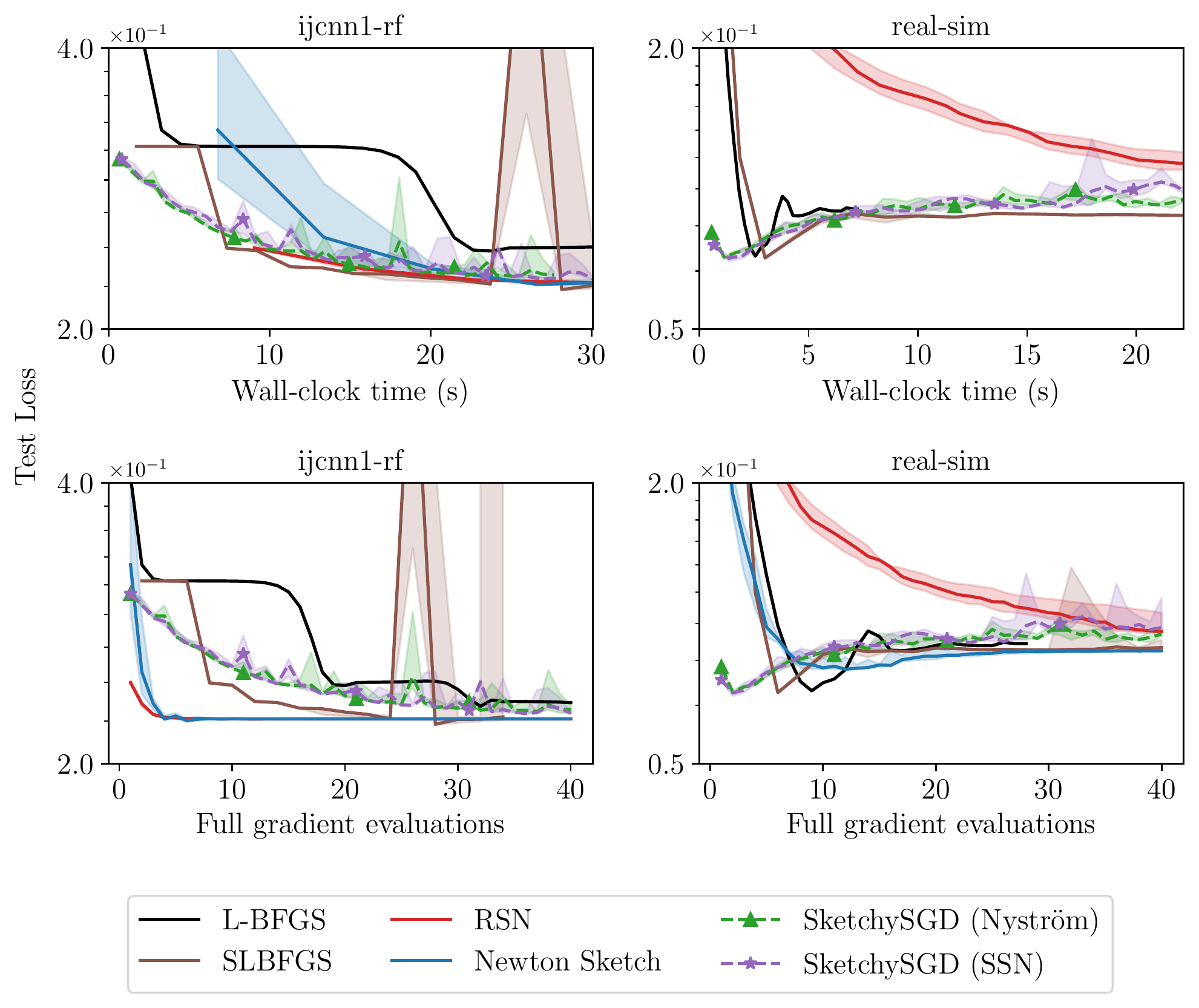}
    \caption{Test losses of quasi-Newton methods (L-BFGS, SLBFGS, RSN, Newton Sketch) on $l_2$-regularized logistic regression with augmented datasets.}
    \label{fig:performance_som_logistic_scaled_test_loss}
\end{figure}

\begin{figure}[htbp]
    \centering
    \includegraphics[scale=0.5]{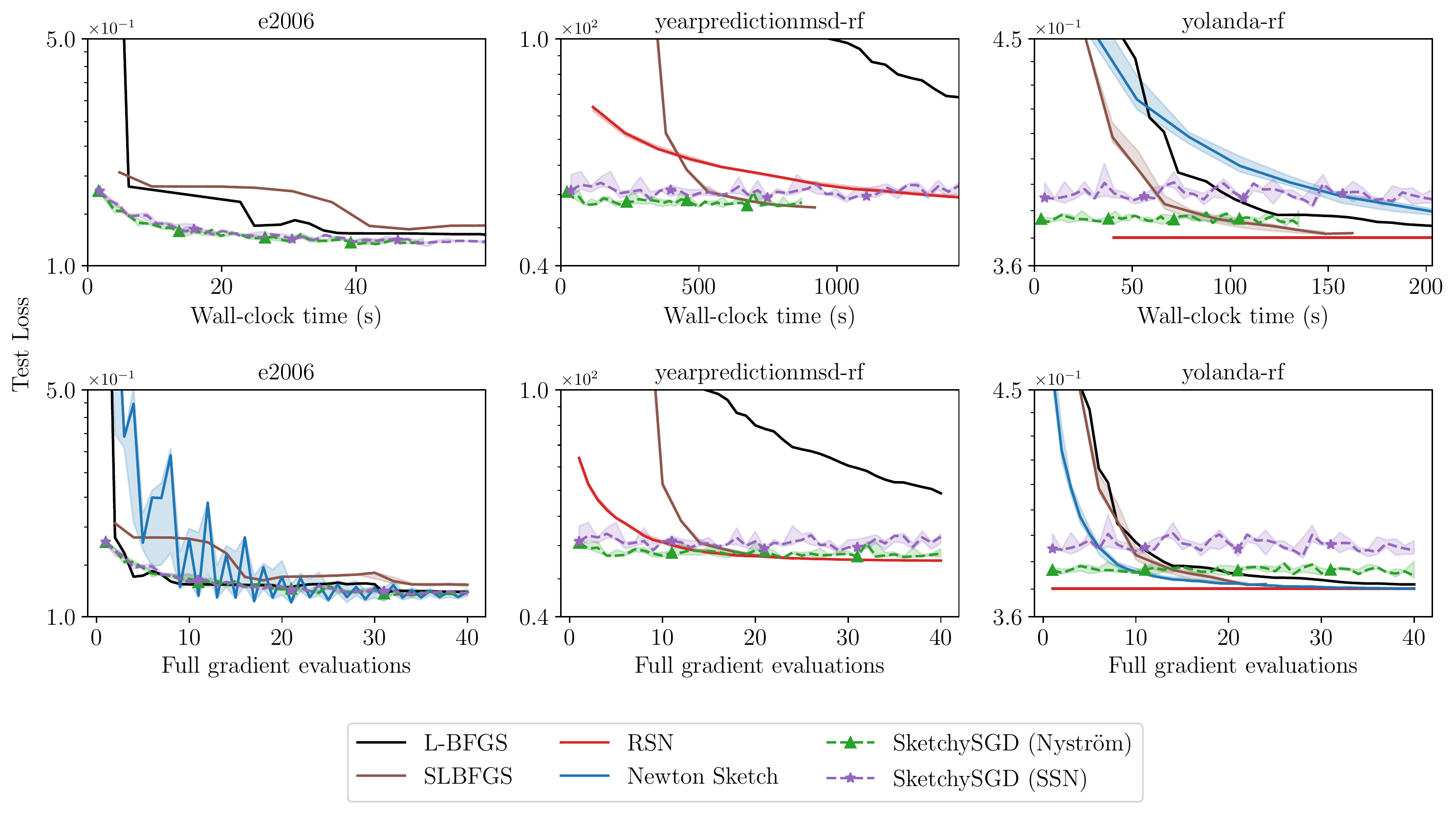}
    \caption{Test losses of quasi-Newton methods (L-BFGS, SLBFGS, RSN, Newton Sketch) on ridge regression with augmented datasets.}
    \label{fig:performance_som_least_squares_scaled_test_loss}
\end{figure}

\begin{figure}[htbp]
    \centering
    \includegraphics[scale=0.5]{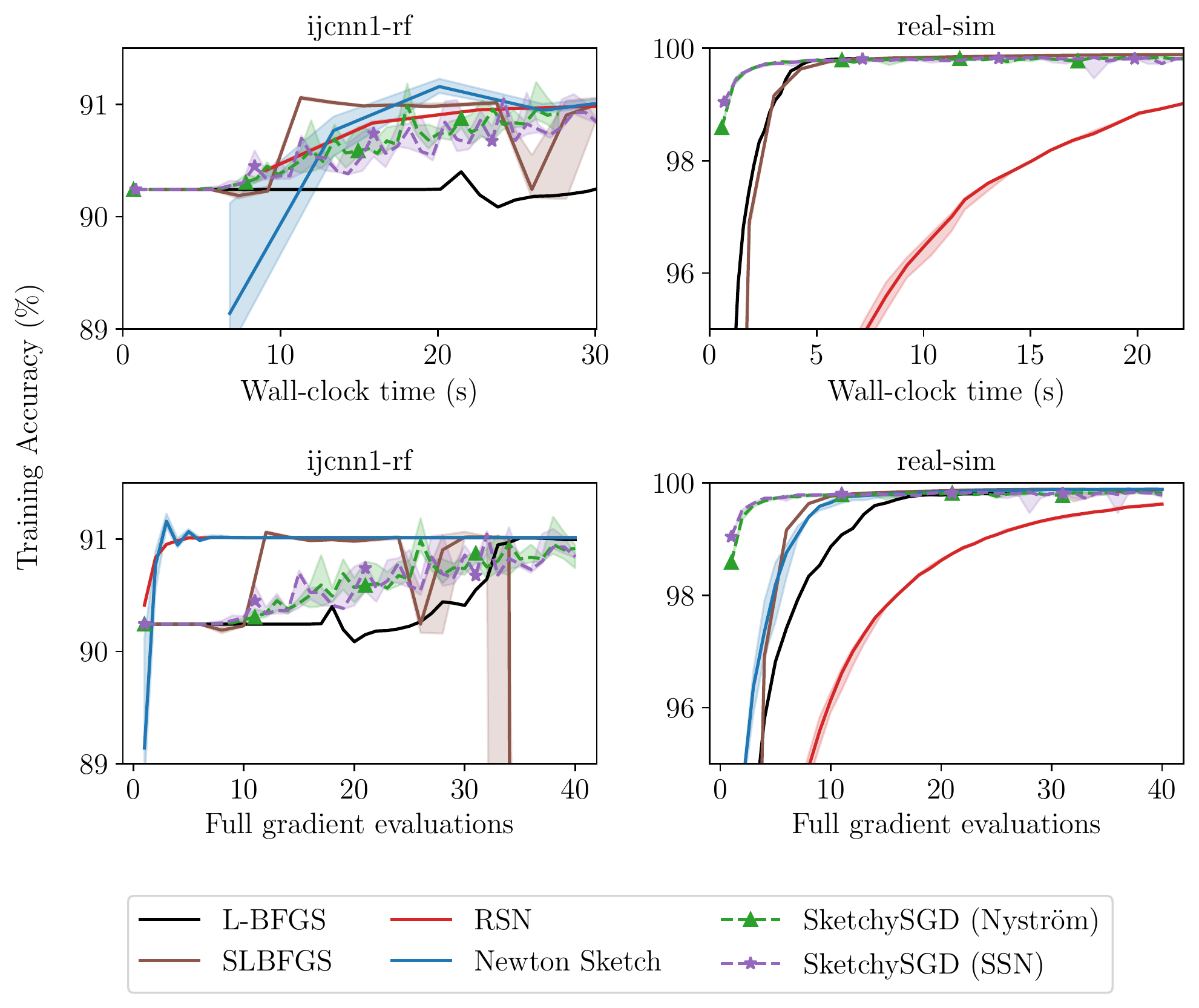}
    \caption{Training accuracies of quasi-Newton methods (L-BFGS, SLBFGS, RSN, Newton Sketch) on $l_2$-regularized logistic regression with augmented datasets.}
    \label{fig:performance_som_logistic_scaled_train_acc}
\end{figure}

\begin{figure}[htbp]
    \centering
    \includegraphics[scale=0.5]{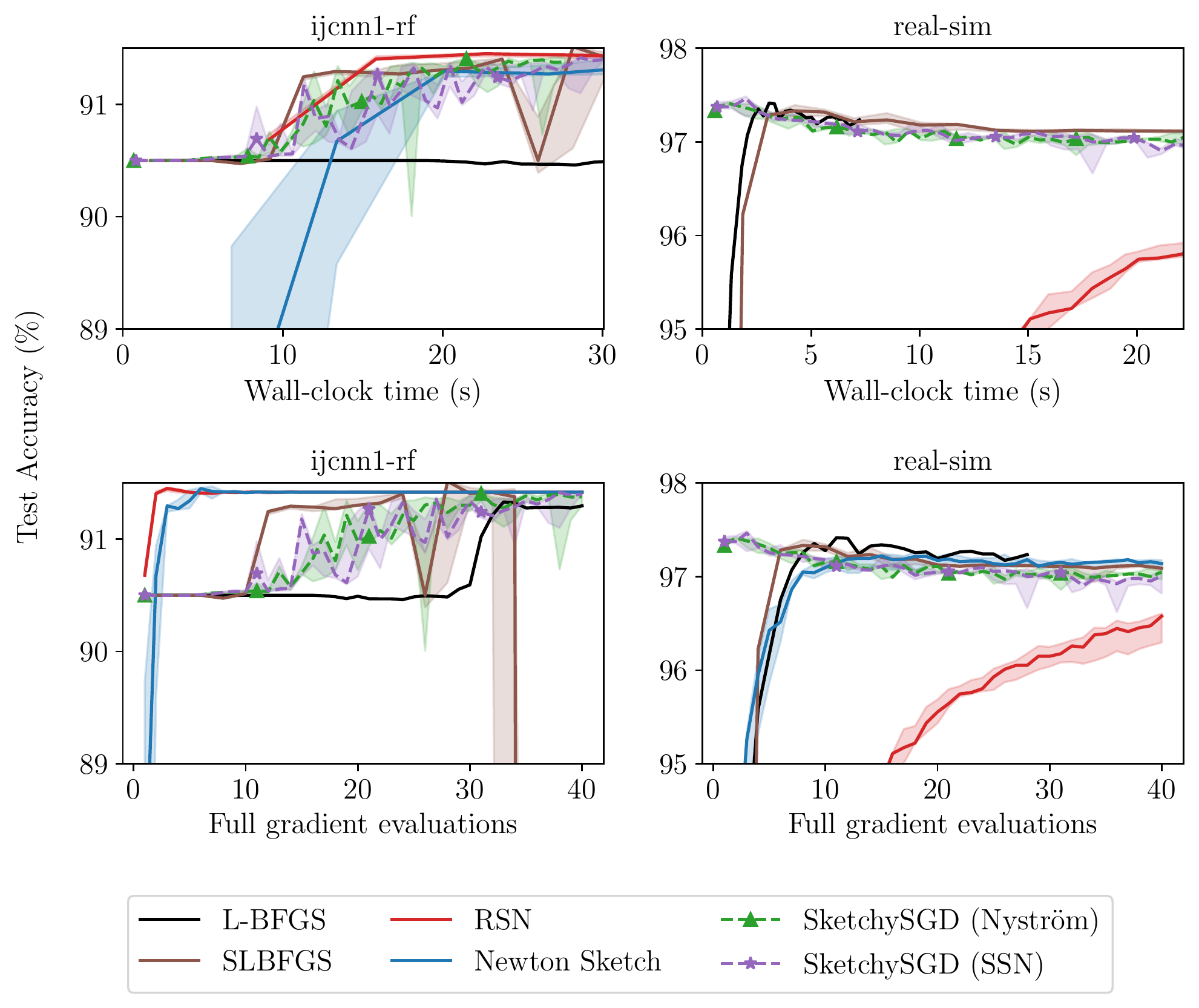}
    \caption{Test accuracies of quasi-Newton methods (L-BFGS, SLBFGS, RSN, Newton Sketch) on $l_2$-regularized logistic regression with augmented datasets.}
    \label{fig:performance_som_logistic_scaled_test_acc}
\end{figure}

\paragraph{Comparisons to PCG}
\Cref{fig:performance_pcg_least_squares_test_loss,fig:performance_pcg_least_squares_scaled_test_loss} display the test losses corresponding to \Cref{subsection:performance_pcg}, where we compare SketchySGD to several variants of PCG.

\begin{figure}[htbp]
    \centering
    \includegraphics[scale=0.5]{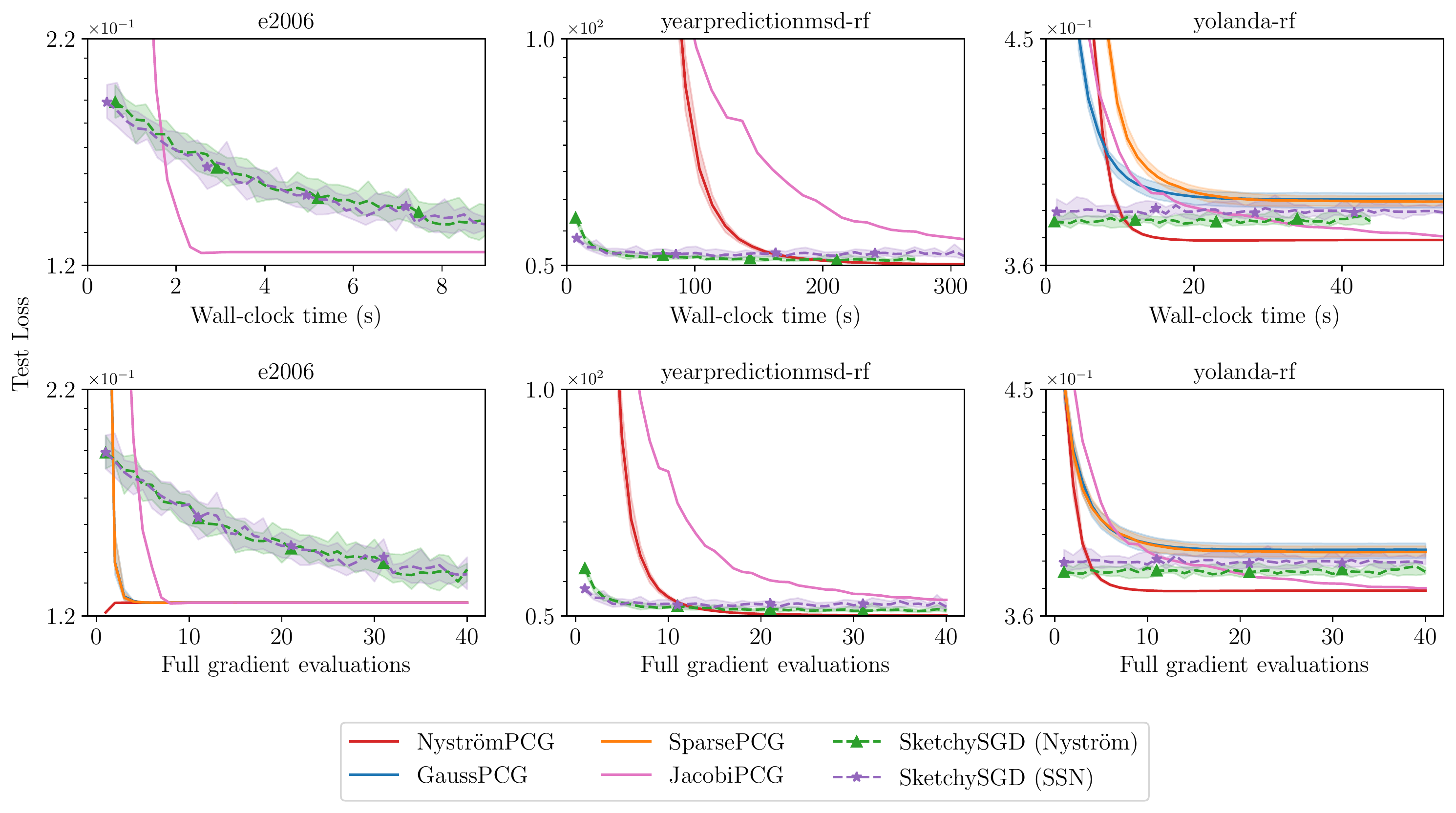}
    \caption{Test losses of PCG (Jacobi, Nystr\"{o}m, sketch-and-precondition w/ Gaussian and sparse embeddings) on ridge regression.}
    \label{fig:performance_pcg_least_squares_test_loss}
\end{figure}

\begin{figure}[htbp]
    \centering
    \includegraphics[scale=0.5]{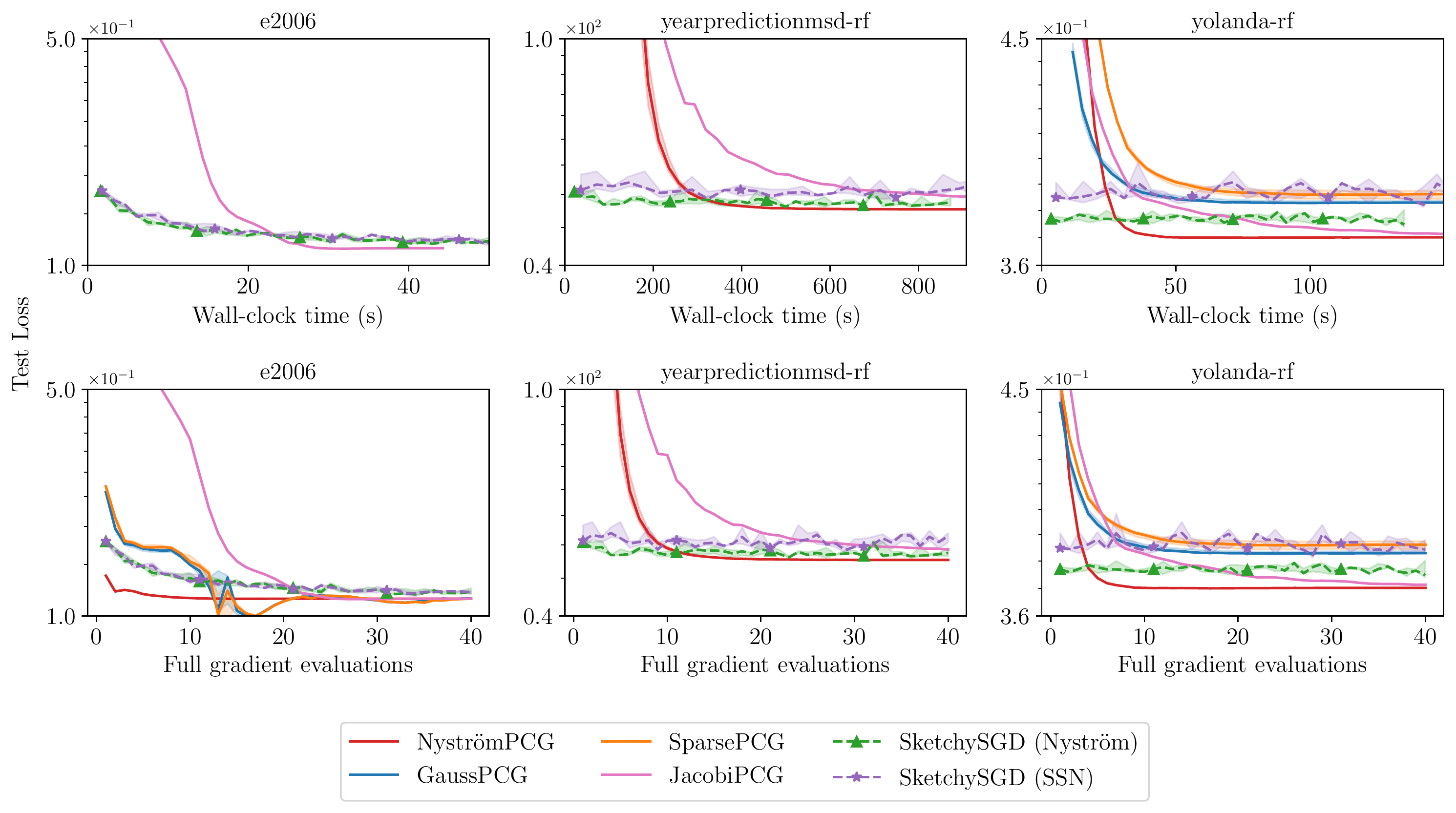}
    \caption{Test losses of PCG (Jacobi, Nystr\"{o}m, sketch-and-precondition w/ Gaussian and sparse embeddings) on ridge regression with augmented datasets.}
    \label{fig:performance_pcg_least_squares_scaled_test_loss}
\end{figure}

\paragraph{Large-scale logistic regression}
\Cref{fig:higgs_auto_test_acc,fig:higgs_tuned_test_acc} display the test accuracies corresponding to \Cref{subsection:large_scale}, where we compare SketchySGD to competitor methods with both default and tuned hyperparameters on a large-scale version of the HIGGS dataset.

\begin{figure}[htbp]
    \centering
    \includegraphics[scale=0.5]{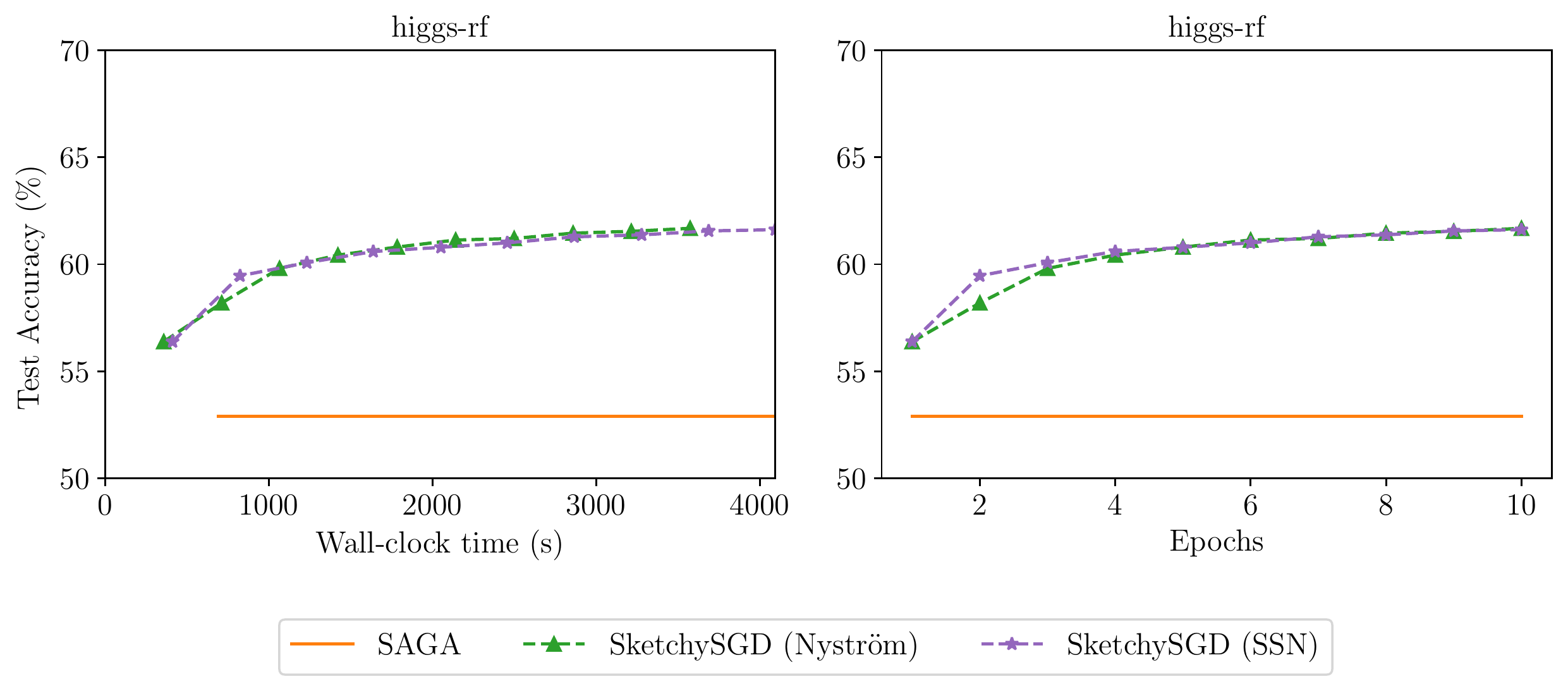}
    \caption{Comparison of test accuracies between SketchySGD and SAGA with default learning rate.}
    \label{fig:higgs_auto_test_acc}
\end{figure}

\begin{figure}[htbp]
    \centering
    \includegraphics[scale=0.5]{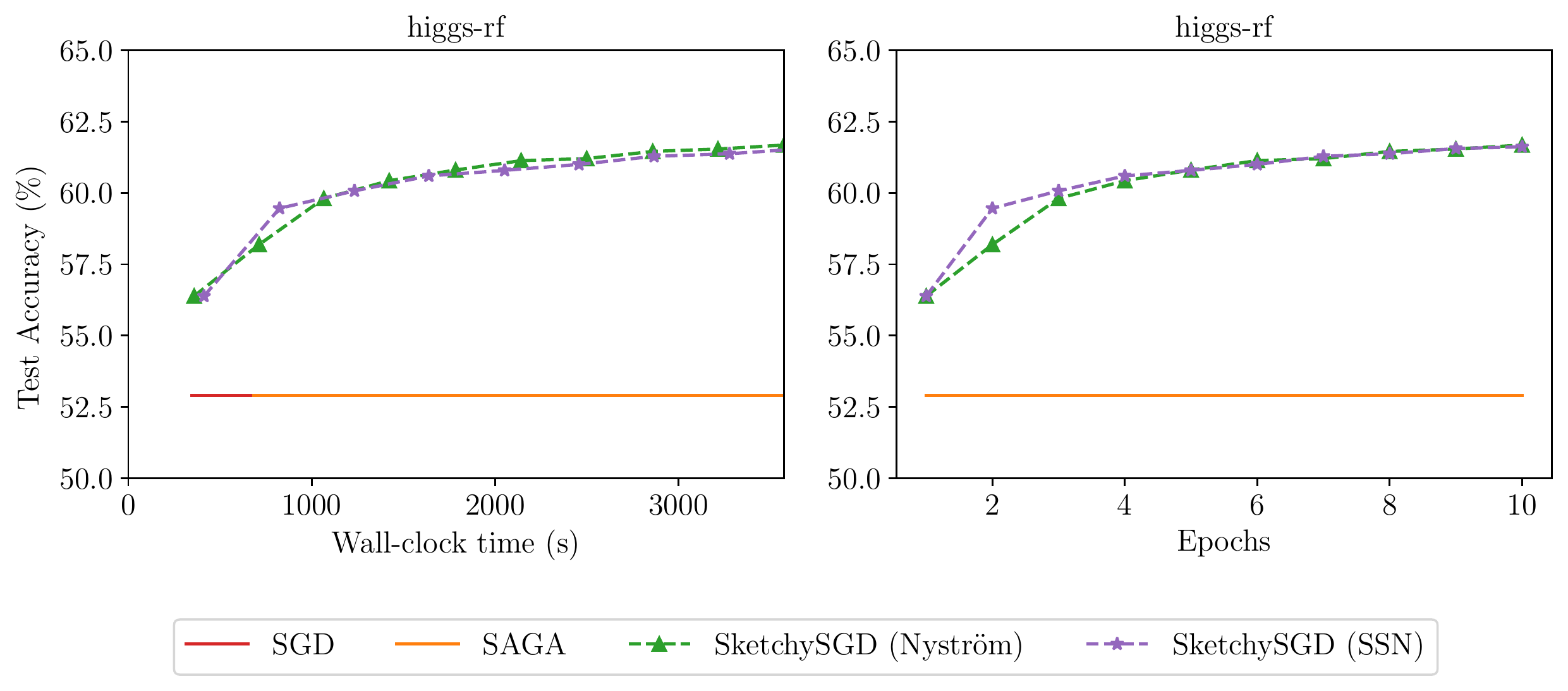}
    \caption{Comparison of test accuracies between SketchySGD, SGD and SAGA with tuned learning rates.}
    \label{fig:higgs_tuned_test_acc}
\end{figure}

\paragraph{Tabular deep learning}
\Cref{fig:deep_learning_train_loss,fig:deep_learning_test_loss,fig:deep_learning_train_acc} display the training loss/test loss/training accuracy plots corresponding to \Cref{subsection:performance_dl}, where we compare SketchySGD to popular deep learning optimizers.

\begin{figure}[htbp]
    \centering
    \includegraphics[scale=0.5]{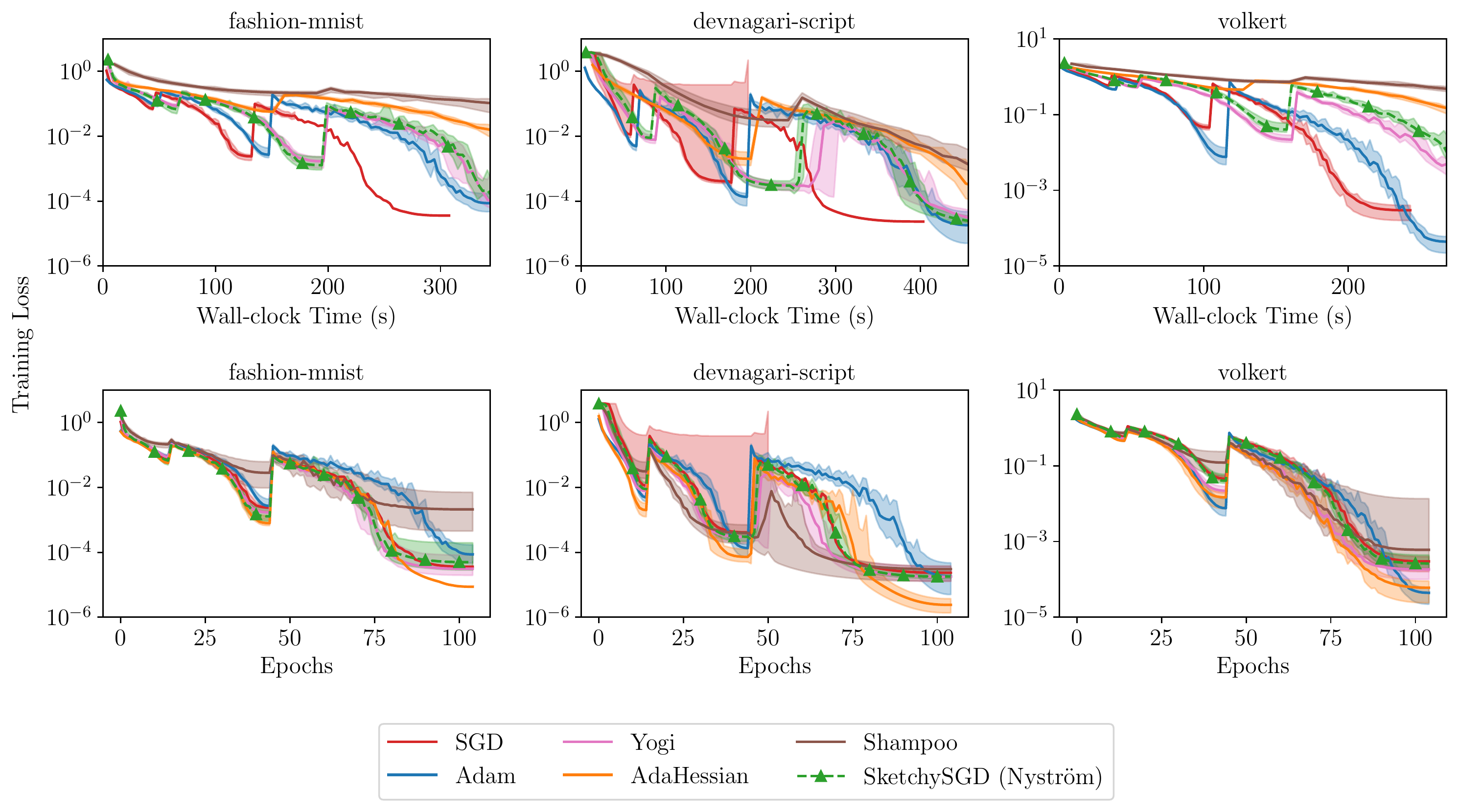}
    \caption{Training losses for SketchySGD and competitor methods on tabular deep learning tasks.}
    \label{fig:deep_learning_train_loss}
\end{figure}

\begin{figure}[htbp]
    \centering
    \includegraphics[scale=0.5]{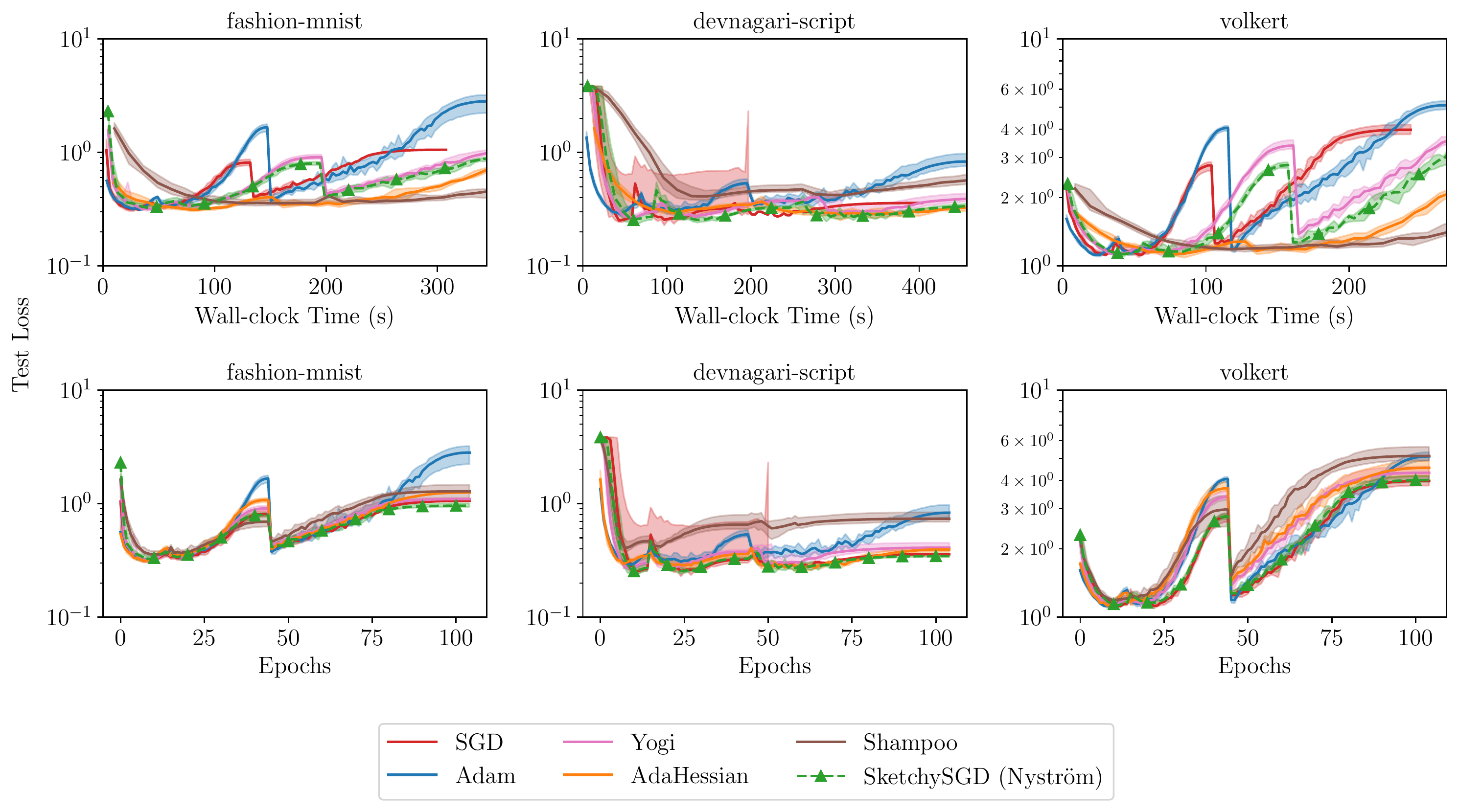}
    \caption{Test losses for SketchySGD and competitor methods on tabular deep learning tasks.}
    \label{fig:deep_learning_test_loss}
\end{figure}

\begin{figure}[htbp]
    \centering
    \includegraphics[scale=0.5]{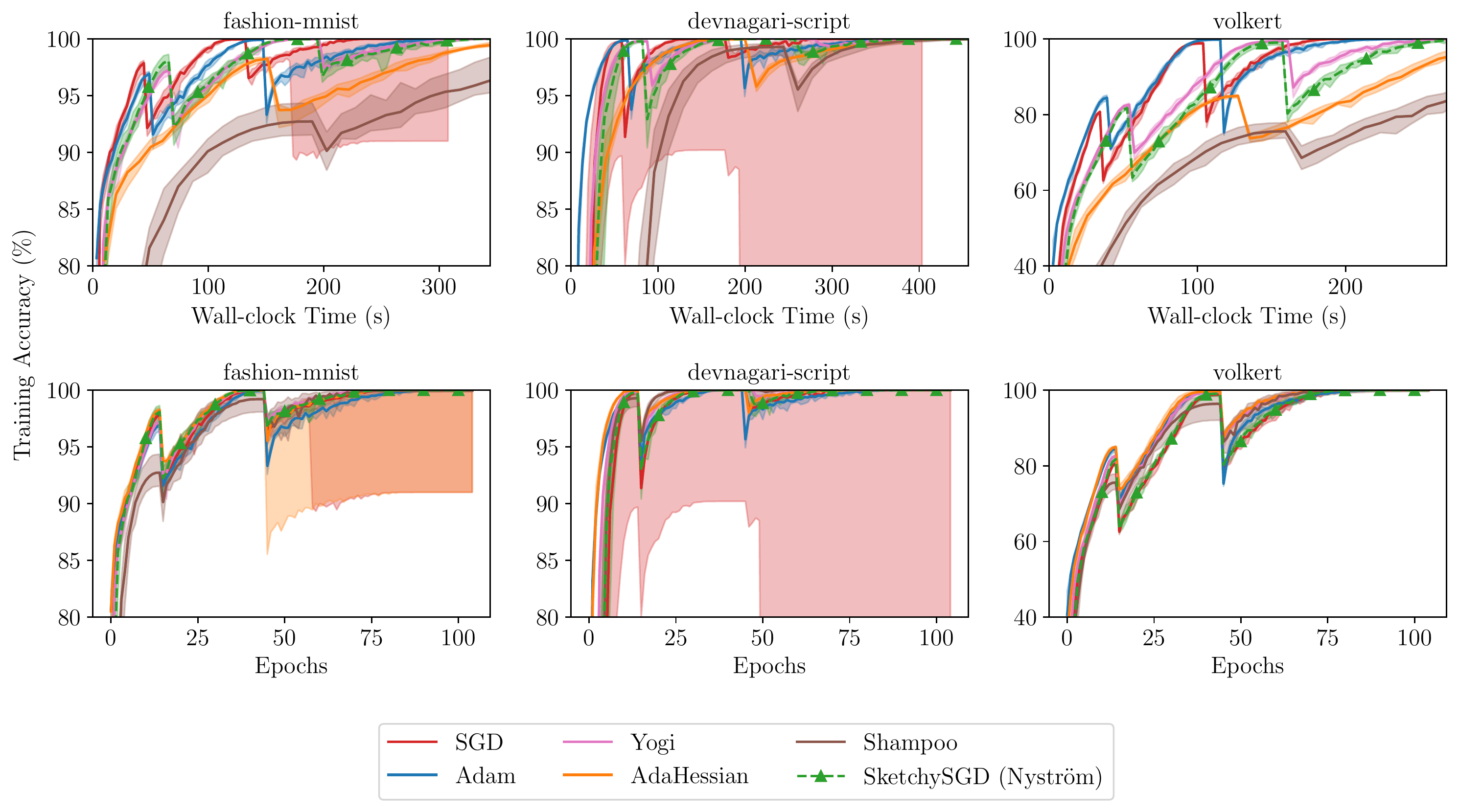}
    \caption{Training accuracies for SketchySGD and competitor methods on tabular deep learning tasks.}
    \label{fig:deep_learning_train_acc}
\end{figure}
\clearpage

\fi
\end{document}